\documentclass[reqno,10pt,centertags,draft]{amsart}
\usepackage{amsmath,amsthm,amscd,amssymb,latexsym,upref}
\date{\today}

\newcommand{\bbN}{{\mathbb{N}}}
\newcommand{\bbR}{{\mathbb{R}}}

\newcommand{\bbZ}{{\mathbb{Z}}}
\newcommand{\bbC}{{\mathbb{C}}}

\newcommand{\cA}{{\mathcal A}}
\newcommand{\cB}{{\mathcal B}}
\newcommand{\cC}{{\mathcal C}}
\newcommand{\cD}{{\mathcal D}}

\newcommand{\cH}{{\mathcal H}}

\newcommand{\cM}{{\mathcal M}}
\newcommand{\cN}{{\mathcal N}}

\newcommand{\cS}{{\mathcal S}}

\newcommand{\cV}{{\mathcal V}}
\newcommand{\cX}{{\mathcal X}}

\newcommand{\dott}{\,\cdot\,}
\newcommand{\no}{\notag}
\newcommand{\lb}{\label}
\newcommand{\f}{\frac}

\newcommand{\ol}{\overline}

\newcommand{\wti}{\widetilde}

\newcommand{\Oh}{O}

\newcommand{\ran}{\text{\rm{ran}}}

\newcommand{\dom}{\text{\rm{dom}}}

\newcommand{\supp}{\text{\rm{supp}}}

\newcommand{\bi}{\bibitem}
\newcommand{\hatt}{\widehat}
\newcommand{\beq}{\begin{equation}}
\newcommand{\eeq}{\end{equation}}
\newcommand{\ba}{\begin{align}}
\newcommand{\ea}{\end{align}}

\newcommand{\abs}[1]{\lvert#1\rvert}

\renewcommand{\Im}{\text{\rm Im}}
\renewcommand{\ln}{\text{\rm ln}}
\renewcommand{\ge}{\geqslant}
\renewcommand{\le}{\leqslant}

\newcommand{\norm}[1]{\left\Vert#1\right\Vert}

\newcommand{\Om}{\Omega}
\newcommand{\dOm}{{\partial\Omega}}
\newcommand{\si}{\sigma}

\newcommand{\ga}{\gamma}

\newcommand{\eps}{\varepsilon}

\newcommand{\LOm}{L^2(\Om;d^nx)}
\newcommand{\LdOm}{L^2(\dOm;d^{n-1} \omega)}

\allowdisplaybreaks \numberwithin{equation}{section}

\newtheorem{theorem}{Theorem}[section]

\newtheorem{lemma}[theorem]{Lemma}
\newtheorem{corollary}[theorem]{Corollary}
\newtheorem{hypothesis}[theorem]{Hypothesis}
\theoremstyle{definition}

\newtheorem{remark}[theorem]{Remark}

\begin{document}

\title[Robin-to-Dirichlet Maps and Krein-Type Resolvent Formulas]
{Generalized Robin Boundary Conditions, Robin-to-Dirichlet Maps, and  
Krein-Type Resolvent Formulas for Schr\"odinger Operators on Bounded 
Lipschitz Domains}
\author[F.\ Gesztesy and M.\ Mitrea]{Fritz Gesztesy and Marius Mitrea}
\address{Department of Mathematics,
University of Missouri, Columbia, MO 65211, USA}
\email{fritz@math.missouri.edu}
\urladdr{http://www.math.missouri.edu/personnel/faculty/gesztesyf.html}
\address{Department of Mathematics, University of
Missouri, Columbia, MO 65211, USA}
\email{marius@math.missouri.edu}
\urladdr{http://www.math.missouri.edu/personnel/faculty/mitream.html} 
\thanks{Based upon work partially supported by the US National Science
Foundation under Grant Nos.\ DMS-0400639 and FRG-0456306.}
\dedicatory{Dedicated with great pleasure to Vladimir Maz'ya on the 
occasion of his 70th birthday.}
\thanks{To appear in {\it  Perspectives in Partial Differential 
Equations, Harmonic Analysis and Applications}, D.\ Mitrea and 
M.\ Mitrea (eds.), Proceedings of Symposia in Pure Mathematics, 
American Mathematical Society, Providence, RI, 2008.}
\date{\today}

\begin{abstract}
We study generalized Robin boundary conditions, Robin-to-Dirichlet maps, 
and Krein-type resolvent formulas for Schr\"odinger operators on bounded 
Lipschitz domains in $\bbR^n$, $n\ge 2$. We also discuss the case of bounded 
$C^{1,r}$-domains, $(1/2)<r<1$.
\end{abstract}

\maketitle

\section{Introduction}\label{s1}

This paper is a continuation of the earlier papers \cite{GLMZ05} and 
\cite{GMZ07}, where we studied general, not necessarily self-adjoint, 
Schr\"odinger operators on $C^{1,r}$-domains $\Om\subset \bbR^n$, $n\in\bbN$, 
$n\ge 2$, with compact boundaries $\dOm$, $(1/2)<r<1$ (including unbounded 
domains, i.e., exterior domains) with Dirichlet and Neumann boundary conditions
on $\dOm$. Our results also applied to convex domains $\Om$ and to domains 
satisfying a uniform exterior ball condition. In addition, a careful 
discussion of locally singular potentials $V$ with close to optimal local 
behavior of $V$ was provided in \cite{GLMZ05} and \cite{GMZ07}. 

In this paper we push the envelope in a different direction: Rather than 
discussing potentials with close to optimal local behavior, we will assume 
that $V \in L^\infty(\Om; d^n x)$ and hence essentially 
replace it by zero nearly everywhere in this paper. On the other hand, 
instead of treating Dirichlet and Neumann boundary conditions at $\dOm$, we 
now consider generalized Robin and again Dirichlet boundary conditions, but 
under minimal smoothness conditions on the domain $\Om$, that is, we now 
consider Lipschitz domains $\Om$. Additionally, to reduce some technicalities, 
we will assume that $\Om$ is bounded throughout this paper. Occasionally we 
also discuss the case of bounded $C^{1,r}$-domains, $(1/2)<r<1$. The principal 
new result in this paper is a derivation of Krein-type resolvent formulas for 
Schr\"odinger operators on bounded Lipschitz domains $\Om$ in connection with 
the case of Dirichlet and generalized Robin boundary conditions on $\dOm$.

In Section \ref{s2} we provide a detailed discussion of self-adjoint 
Laplacians with generalized Robin (and Dirichlet) boundary conditions 
on $\dOm$. In Section \ref{s3} we then treat generalized Robin and 
Dirichlet boundary value problems and introduce associated Robin-to-Dirichlet 
and Dirichlet-to-Robin maps. Section \ref{s4} contains the principal new 
results of this paper; it is devoted to Krein-type resolvent formulas 
connecting Dirichlet and generalized Robin Laplacians with the help of the 
Robin-to-Dirichlet map. Appendix \ref{sA} collects useful material on Sobolev 
spaces and trace maps for $C^{1,r}$ and Lipschitz domains. 
Appendix \ref{sB} summarizes pertinent facts on sesquilinear forms and 
their associated linear operators. Estimates on the fundamental solution 
of the Helmholtz equation in $\bbR^n$, $n \ge 2$, are recalled in Appendix 
\ref{sC}. Finally, certain results on Calder\'on--Zygmund theory on 
Lipschitz surfaces of fundamental relevance to the material in the main 
body of this paper are presented in Appendix \ref{sD}.  

While we formulate and prove all results in this paper for self-adjoint 
generalized Robin Laplacians and Dirichlet Laplacians, we emphasize that 
all results in this paper immediately extend to closed Schr\"odinger operators 
$H_{\Theta,\Om} = -\Delta_{\Theta,\Om} + V$, 
$\dom\big(H_{\Theta,\Om}\big) = \dom\big(-\Delta_{\Theta,\Om}\big)$  
in $\LOm$ for (not necessarily real-valued) potentials $V$ satisfying 
$V \in L^\infty(\Om; d^n x)$, by consistently replacing $-\Delta$ by 
$-\Delta + V$, etc. More generally, all results extend directly to 
Kato--Rellich bounded potentials $V$ relative to $-\Delta_{\Theta,\Om}$ 
with bound less than one.  

Next, we briefly list most of the notational conventions used throughout 
this paper. Let $\cH$ be a separable complex Hilbert space, 
$(\dott,\dott)_{\cH}$ the scalar product in $\cH$ 
(linear in the second factor), and $I_{\cH}$ the identity operator in $\cH$.
Next, let $T$ be a linear operator mapping (a subspace of) a
Banach space into another, with $\dom(T)$ and $\ran(T)$ denoting the
domain and range of $T$. The spectrum (resp., essential spectrum) of a 
closed linear operator in $\cH$ will be 
denoted by $\sigma(\dott)$ (resp., $\sigma_{\rm ess}(\dott)$). 
The Banach spaces of bounded and compact linear 
operators in $\cH$ are denoted by $\cB(\cH)$ and $\cB_\infty(\cH)$, 
respectively. Similarly, $\cB(\cH_1,\cH_2)$ and $\cB_\infty (\cH_1,\cH_2)$ 
will be used for bounded and compact operators between two Hilbert spaces 
$\cH_1$ and $\cH_2$. Moreover, $\cX_1 \hookrightarrow \cX_2$ denotes the 
continuous embedding of the Banach space $\cX_1$ into the Banach space 
$\cX_2$. Throughout this manuscript, if $X$ denotes a Banach space, $X^*$ 
denotes the {\it adjoint space} of continuous conjugate linear functionals 
on $X$, that is, the {\it conjugate dual space} of $X$ (rather than the usual 
dual space of continuous linear functionals 
on $X$). This avoids the well-known awkward distinction between adjoint 
operators in Banach and Hilbert spaces (cf., e.g., the pertinent discussion 
in \cite[p.\,3--4]{EE89}). 

Finally, a notational comment: For obvious reasons in connection 
with quantum mechanical applications, we will, with a slight abuse of 
notation, dub $-\Delta$ (rather than $\Delta$) as the ``Laplacian'' in 
this paper. 

\section{Laplace Operators with Generalized Robin Boundary Conditions}
\label{s2}

In this section we primarily focus on various properties of general 
Laplacians $-\Delta_{\Theta,\Om}$ in $L^2(\Om;d^n x)$ including Dirichlet,
$-\Delta_{D,\Om}$, and Neumann, $-\Delta_{N,\Om}$, Laplacians, generalized 
Robin-type Laplacians, and Laplacians corresponding to classical Robin 
boundary conditions associated with open sets $\Omega\subset \bbR^n$, 
$n\in\bbN$, $n\geq 2$, introduced in Hypothesis \ref{h2.1} below. 

We start with introducing our assumptions on the set $\Omega$ and the 
boundary operator $\Theta$ which subsequently will be employed in defining 
the boundary condition on $\dOm$:

\begin{hypothesis}\lb{h2.1}
Let $n\in\bbN$, $n\geq 2$, and assume that $\Om\subset{\bbR}^n$ is
an open, bounded, nonempty Lipschitz domain.
\end{hypothesis}

We refer to Appendix \ref{sA} for more details on Lipschitz domains.

For simplicity of notation we will denote the identity operators in $\LOm$ and 
$\LdOm$ by $I_{\Om}$ and $I_{\dOm}$, respectively.
Also, we refer to Appendix \ref{sA} for our notation in
connection with Sobolev spaces.
 
\begin{hypothesis} \lb{h2.2}
Assume Hypothesis \ref{h2.1} and suppose that $a_{\Theta}$ is a closed 
sesquilinear form in $\LdOm$ with domain $H^{1/2}(\dOm)\times H^{1/2}(\dOm)$, 
bounded from below by $c_{\Theta}\in\bbR$ $($hence, in particular, $a_{\Theta}$
is symmetric$)$. Denote by $\Theta \ge c_{\Theta}I_{\dOm}$ the self-adjoint 
operator in $\LdOm$ 
uniquely associated with $a_{\Theta}$ $($cf.\ \eqref{B.25}$)$ and by 
$\wti \Theta \in \cB\big(H^{1/2}(\dOm),H^{-1/2}(\dOm)\big)$ the extension 
of $\Theta$ as discussed in \eqref{B.24a} and \eqref{B.28a}.
\end{hypothesis}

Thus one has  
\begin{align}
& \big\langle f,\wti\Theta\,g\big\rangle_{1/2} 
= \ol{\big\langle g, \wti \Theta \, f \big\rangle_{1/2}}, 
 \quad f, g \in H^{1/2}(\dOm).     \lb{2.1}  \\
& \big\langle f, \wti \Theta \, f \big\rangle_{1/2} 
\geq c_{\Theta}\|f\|_{L^2(\dOm;d^{n-1}\omega)}^2,  
\quad f \in H^{1/2}(\dOm).  \lb{2.2} 
\end{align}

Here the sesquilinear form 
\begin{equation}
\langle \dott, \dott \rangle_{s}={}_{H^{s}(\dOm)}\langle\dott,\dott 
\rangle_{H^{-s}(\dOm)}\colon H^{s}(\dOm)\times H^{-s}(\dOm) 
\to \bbC, \quad s\in [0,1],   
\end{equation}
(antilinear in the first, linear in the second factor), denotes the duality 
pairing between $H^s(\dOm)$ and  
\begin{equation}
H^{-s}(\dOm) = \big(H^s(\dOm)\big)^*,  \quad s\in [0,1],    \lb{2.3}
\end{equation}
such that
\begin{equation}
\langle f, g \rangle_s = \int_{\dOm} d^{n-1} \omega(\xi) \, \ol{f(\xi)} 
g(\xi), \quad 
f \in H^s(\dOm), \; g \in L^2(\dOm; d^{n-1} \omega) \hookrightarrow 
H^{-s}(\dOm),  
\; s\in [0,1],  \lb{2.4}
\end{equation}
and $d^{n-1} \omega$ denotes the surface measure on $\dOm$. 

Hypothesis \ref{h2.1} on $\Om$ is used throughout this paper. 
Similarly, Hypothesis \ref{h2.2} is assumed whenever the boundary 
operator $\wti \Theta$ is involved. (Later in this section, and the next, 
we will occasionally strengthen our hypotheses.) 

We introduce the boundary trace operator $\ga_D^0$ (the Dirichlet trace) by
\begin{equation}
\ga_D^0\colon C(\ol{\Om})\to C(\dOm), \quad \ga_D^0 u = u|_\dOm.   \lb{2.5}
\end{equation}
Then there exists a bounded, linear operator $\gamma_D$ (cf., e.g., 
\cite[Theorem 3.38]{Mc00}),
\begin{align}
\begin{split}
& \ga_D\colon H^{s}(\Om)\to H^{s-(1/2)}(\dOm) \hookrightarrow \LdOm,
\quad 1/2<s<3/2, \lb{2.6}  \\
& \ga_D\colon H^{3/2}(\Om)\to H^{1-\varepsilon}(\dOm) \hookrightarrow \LdOm,
\quad \varepsilon \in (0,1), 
\end{split}
\end{align}
whose action is compatible with that of $\ga_D^0$. That is, the two
Dirichlet trace  operators coincide on the intersection of their
domains. Moreover, we recall that 
\begin{equation}
\ga_D\colon H^{s}(\Om)\to H^{s-(1/2)}(\dOm) \, \text{ is onto for  
$1/2<s<3/2$.} \lb{2.6a}
\end{equation}

While, in the class of bounded Lipschitz subdomains in $\bbR^n$, 
the end-point cases $s=1/2$ and $s=3/2$ of 
$\gamma_D\in\cB\bigl(H^{s}(\Om),H^{s-(1/2)}(\dOm)\bigr)$ fail, we nonetheless
have
\begin{eqnarray}\label{A.62x}
\ga_D\in \cB\big(H^{(3/2)+\eps}(\Om), H^{1}(\dOm)\big), \quad \eps>0.
\end{eqnarray}
See Lemma \ref{lA.6x} for a proof. Below we augment this with the following
result: 

\begin{lemma}\label{Gam-L1}
Assume Hypothesis \ref{h2.1}. Then for each $s>-3/2$, the restriction 
to boundary operator \eqref{2.5} extends to a linear operator 
\begin{eqnarray}\label{Mam-1}
\gamma_D:\bigl\{u\in H^{1/2}(\Omega)\,\big|\,\Delta u\in H^{s}(\Omega)\bigr\}
\to L^2(\partial\Omega;d^{n-1}\omega),
\end{eqnarray}
is compatible with \eqref{2.6}, and is bounded when 
$\{u\in H^{1/2}(\Omega)\,|\,\Delta u\in H^{s}(\Omega)\bigr\}$ is equipped
with the natural graph norm $u\mapsto \|u\|_{H^{1/2}(\Omega)}
+\|\Delta u\|_{H^{s}(\Omega)}$. In addition, this operator has a
linear, bounded right-inverse $($hence, in particular, it is onto$)$. 

Furthermore, for each $s>-3/2$, the restriction 
to boundary operator \eqref{2.5} also extends to a linear operator 
\begin{eqnarray}\label{Mam-2}
\gamma_D:\bigl\{u\in H^{3/2}(\Omega)\,\big|\,\Delta u\in H^{1+s}(\Omega)\bigr\}
\to H^1(\partial\Omega), 
\end{eqnarray}
which, once again, is compatible with \eqref{2.6}, and is bounded when 
$\{u\in H^{3/2}(\Omega)\,|\,\Delta u\in H^{1+s}(\Omega)\bigr\}$ is equipped
with the natural graph norm $u\mapsto \|u\|_{H^{3/2}(\Omega)}
+\|\Delta u\|_{H^{1+s}(\Omega)}$. Once again, this operator has a
linear, bounded right-inverse $($hence, in particular, it is onto$)$. 
\end{lemma}
\begin{proof}
For each $s\in\bbR$ set $H^{s}_{\Delta}(\Omega)=
\{u\in H^{s}(\Omega) \,|\,\Delta u=0\mbox{ in }\Omega\}$ 
and observe that this is a closed subspace of $H^{s}(\Omega)$. 
In particular, $H^{s}_{\Delta}(\Omega)$ is a Banach space when 
equipped with the norm inherited from $H^{s}(\Omega)$. 
Next we recall the nontangential maximal operator $M$ defined in \eqref{Rb-2}. 
According to \cite{Fa88}, or Corollary~5.7 in \cite{JK95}, one has 
\begin{equation}\label{Mam-4.1}
H^{1/2}_{\Delta}(\Omega)=\bigl\{u\mbox{ harmonic in }\Omega\,\big|\, 
M(u)\in L^2(\partial\Omega;d^{n-1}\omega)\bigr\}
\end{equation}
and $u\mapsto \|M(u)\|_{L^2(\partial\Omega;d^{n-1}\omega)}$ is an 
equivalent norm on $H^{1/2}_{\Delta}(\Omega)$. 
To continue, fix some $\kappa>0$ and set 
$d(y)={\rm dist}\,(y,\partial\Omega)$ for 
$y\in\Omega$. According to \cite{Da77}, the nontangential trace operator
\begin{equation}\label{Mam-4}
(\gamma_{\rm n.t.}u)(x)=\lim_{\stackrel{\Omega\ni y\to x}
{|x-y|<(1+\kappa)\,d(y)}}u(y)  
\end{equation}
is then well-defined when considered as a mapping 
\begin{eqnarray}\label{Mam-5}
\gamma_{\rm n.t.}:\bigl\{u\mbox{ harmonic in }\Omega\,\big|\, 
M(u)\in L^2(\partial\Omega;d^{n-1}\omega)\bigr\}
\to L^2(\partial\Omega;d^{n-1}\omega).
\end{eqnarray}
Furthermore, the operator \eqref{Mam-4}, \eqref{Mam-5} is bounded. 

Granted these results, for a fixed $s>-3/2$ we may then attempt to define 
\begin{eqnarray}\label{Mam-6}
\wti\gamma_D: H^{1/2}_{\Delta}(\Omega)+H^{s+2}(\Omega)
\to L^2(\partial\Omega;d^{n-1}\omega)
\end{eqnarray}
by setting 
\begin{eqnarray}\label{Mam-7}
\wti\gamma_D(u+v)=\gamma_{\rm n.t.}u+\gamma_Dv, \quad
u\in H^{1/2}_{\Delta}(\Omega), \; v\in H^{s+2}(\Omega).
\end{eqnarray}
A moment's reflection shows that, in order to establish that 
the mapping \eqref{Mam-6}, \eqref{Mam-7} is well-defined, 
it suffices to prove that 
\begin{eqnarray}\label{Mam-8}
\gamma_{\rm n.t.}u=\gamma_Du\mbox{ in $L^2(\partial\Omega;d^{n-1}\omega)$ 
whenever }u\in H^{(1/2)+\varepsilon}(\Omega),\,\,\varepsilon>0. 
\end{eqnarray}
in the case when $\Omega$ is a bounded Lipschitz domain which is 
star-like with respect to the origin in $\bbR^n$ (cf. \eqref{Lip-S}). 

Assuming that this is 
the case, pick $u\in H^{(1/2)+\varepsilon}(\Omega)$ for some $\varepsilon>0$, 
and for each $t\in(0,1)$ set $u_t(x)=u(tx)$, $x\in\Omega$. 
We claim that 
\begin{eqnarray}\label{Mam-9}
u_t\to u\,\,\mbox{ in }\,\,
H^{(1/2)+\varepsilon}(\Omega)\,\mbox{ as }\,t\to 1. 
\end{eqnarray}
To justify this, it suffices to prove that this is the case when 
$u\in C^\infty(\overline{\Omega})$ as the result in its full generality 
then follows from a standard density argument. However, for every 
$u\in C^\infty(\overline{\Omega})$ one trivially has
$u_t\to u$ as $t\to 1$ in $H^1(\Omega)$, hence in 
$H^{(1/2)+\varepsilon}(\Omega)$. Having disposed of \eqref{Mam-9}, 
we may then conclude that $\gamma_D u_t\to \gamma_D u$ in 
$L^2(\partial\Omega;d^{n-1}\omega)$ as $t\to 1$. 
Since for each $t\in(0,1)$ we have $u_t\in C(\overline{\Omega})$,  
it follows that $\gamma_D u_t=\gamma^0_D u_t=u_t|_{\partial\Omega}$. 
Thus, altogether, 
\begin{eqnarray}\label{Mam-10}
u_t|_{\partial\Omega}\to \gamma_D u\,\,\mbox{ in }\,\,
L^2(\partial\Omega;d^{n-1}\omega)\,\mbox{ as }\,t\to 1. 
\end{eqnarray}
On the other hand, for almost every $x\in\partial\Omega$, and every 
$t\in(0,1)$, we have that $y=tx$ belongs to $\Omega$, converges to $x$ as 
$t\to 1$, and $|x-y|\leq (1+\kappa)\,{\rm dist}\,(y,\partial\Omega)$ for some 
sufficiently large $\kappa=\kappa(\Omega)>0$ (independent of $x$ and $t$). 
This implies that 
\begin{eqnarray}\label{Mam-11}
u_t(x)\to (\gamma_{\rm n.t.}u)(x) \,\text{ pointwise, for a.e.\  
$x\in\partial\Omega$, as $t\to 1$.} 
\end{eqnarray}
Combining \eqref{Mam-10}, \eqref{Mam-11} we therefore conclude that 
the functions $\gamma_{\rm n.t.}u,\,\gamma_D u\in 
L^2(\partial\Omega;d^{n-1}\omega)$ coincide pointwise a.e. 
on $\partial\Omega$. This proves \eqref{Mam-8} and finishes the 
justification of the fact that the mapping \eqref{Mam-6}, \eqref{Mam-7}
is well-defined. 

Granted \eqref{Mam-6}, \eqref{Mam-7} is well-defined, it is implicit in its own 
definition that the mapping \eqref{Mam-6}, \eqref{Mam-7} is also bounded  
when we equip $H^{1/2}_{\Delta}(\Omega)+H^{s+2}(\Omega)$ with the 
canonical norm 
\begin{eqnarray}\label{Mam-12}
w\mapsto\inf_{\stackrel{w=u+v}{u\in H^{1/2}_{\Delta}(\Omega),\,
v\in H^{s+2}(\Omega)}}\,\|u\|_{H^{1/2}_{\Delta}(\Omega)}
+\|v\|_{H^{s+2}(\Omega)}. 
\end{eqnarray}

The same type of argument as above (i.e., restricting attention 
to pieces of $\Omega$ which are star-like Lipschitz domains, and using 
dilations with respect to the respective center of star-likeness) 
shows the following: If $w\in C(\overline{\Omega})$ can be 
decomposed as $u+v$ with $u\in H^{1/2}_{\Delta}(\Omega)$ and 
$v\in H^{s+2}(\Omega)$ for some $s>-3/2$, then 
$w|_{\partial\Omega}=\gamma_{\rm n.t.}u+\gamma_Dv$. In other words,
the action of the trace operator $\wti\gamma_D$ in \eqref{Mam-6}, \eqref{Mam-7} 
is compatible with that of \eqref{2.5}. This completes the study of 
the nature and properties of $\wti\gamma_D$ in \eqref{Mam-6}, \eqref{Mam-7}. 

Consider next the claim made about \eqref{Mam-1}. As regards
its boundedness and the fact that this acts in a compatible fashion 
with \eqref{2.6}, it suffices to prove that 
\begin{eqnarray}\label{Mam-13}
\{u\in H^{1/2}(\Omega)\,|\,\Delta u\in H^{s}(\Omega)\bigr\}
\hookrightarrow H^{1/2}_{\Delta}(\Omega)+H^{s+2}(\Omega),\quad s>-3/2,
\end{eqnarray}
continuously. To see this, pick $u\in H^{1/2}(\Omega)$ such 
that $\Delta u\in H^{s}(\Omega)$ and extend (cf. \cite{Ry99}) $\Delta u$ 
to a compactly supported distribution $w\in H^{s}(\bbR^n)$. Next, set 
\begin{equation}
v(x)=\int_{\bbR^n}d^ny\,E_n(x-y)w(y), \quad x\in\Omega, 
\end{equation}
where
\begin{equation}
E_n(x)=\begin{cases} \frac{1}{2\pi}\ln(|x|),  & n=2, \\ 
\frac{1}{(2-n)\omega_{n-1}}|x|^{2-n},  & n\geq 3, \end{cases}
\end{equation}
is the standard fundamental 
solution for the Laplacian in $\bbR^n$ (cf.\ \eqref{C.1} for $z=0$). Here 
$\omega_{n-1}=2\pi^{n/2}/\Gamma(n/2)$ ($\Gamma(\dott)$ the Gamma function, 
cf.\ \cite[Sect.\ 6.1]{AS72}) represents the area of the unit sphere 
$S^{n-1}$ in $\bbR^n$. Then $v\in H^{s+2}(\Omega)$ and $\Delta v=\Delta u$ 
in $\Omega$. As a consequence, the function $w=u-v$ is harmonic
and belongs to $H^{1/2}(\Omega)$, that is, $u=w+v$ with 
$w\in H^{1/2}_{\Delta}(\Omega)$, $v\in H^{s+2}(\Omega)$. Furthermore, the
estimate 
\begin{eqnarray}\label{Mam-14}
\|w\|_{H^{1/2}(\Omega)}+\|v\|_{H^{s+2}(\Omega)}\leq C\bigl(
\|u\|_{H^{1/2}(\Omega)}+\|\Delta u\|_{H^{s}(\Omega)}\bigr)
\end{eqnarray}
for some $C=C(\Omega,s)>0$ is implicit in the above construction. 
Thus, the inclusion \eqref{Mam-13} is well-defined and continuous, 
so that the claims about the boundedness of \eqref{Mam-1}, as well as
the fact that this acts in a compatible fashion with \eqref{2.6}, follow 
from this and the fact that $\wti\gamma_D$ in \eqref{Mam-6}, \eqref{Mam-7} 
is well-defined and bounded. As far as the existence of a linear, bounded,
right-inverse is concerned, it suffices to point out \eqref{Mam-4.1} and
recall that the mapping \eqref{Mam-5} is onto (cf. \cite{Da77}). 

We now digress momentarily for the purpose of developing an integration
by parts formula which will play a significant role shortly. First,
if $\Omega$ is a bounded star-like Lipschitz domain in $\bbR^n$ and
$G$ is a vector field with components in 
$H^{1/2}_{\Delta}(\Omega)+H^{s+2}(\Omega)$, $s>-3/2$, such that
${\rm div} (G)\in L^1(\Omega)$, then 
\begin{eqnarray}\label{Mam-a.1}
\int_{\Omega}dx^n\,{\rm div} (G)
=\int_{\partial\Omega}d^{n-1}\omega\,\nu\cdot\wti\gamma_D G.
\end{eqnarray}
Indeed, if as before $G_t(x)=G(tx)$, $x\in\Omega$, $t\in(0,1)$, then 
\begin{eqnarray}\label{Mam-a.2}
{\rm div} (G_t)=t({\rm div} (G))_t
\, \mbox{ in the sense of distributions in }\,\Omega.
\end{eqnarray}
Writing \eqref{Mam-a.1} for $G_t$ in place of $G$, with $0<t<1$, and 
then passing to the limit $t\to 1$ yields the desired result. 
As a corollary of \eqref{Mam-a.1} and \eqref{Mam-13}, we also have that 
\eqref{Mam-a.1} holds if $\Omega$ is a bounded star-like Lipschitz domain 
in $\bbR^n$ and $G$ is a vector field with components in 
$\{u\in H^{1/2}(\Omega)\,|\,\Delta u\in H^{s}(\Omega)\}$, $s>-3/2$, such that
${\rm div} (G)\in L^1(\Omega)$. Since the latter space is a module over 
$C^\infty_0(\bbR^n)$ and any Lipschitz domain is locally star-like, 
a simple argument based on a smooth partition of unity shows
that the star-likeness condition on $\Omega$ can be eliminated. 
More precisely, 
\begin{eqnarray}\label{Ht-r3}
\left.
\begin{array}{l}
\mbox{Hypothesis \ref{h2.1},}
\\[4pt]
G\in\bigl\{u\in H^{1/2}(\Omega)\,|\,\Delta u\in H^{s}(\Omega)\bigr\}^n,\,\,
s>-3/2,
\\[4pt]
{\rm div} (G)\in L^1(\Omega; d^n x)
\end{array}
\right\}\Longrightarrow \eqref{Mam-a.1}\,\,\mbox{ holds}.
\end{eqnarray} 

Moving on, consider the operator \eqref{Mam-2}. To get started, we fix
$s>-3/2$ and assume that the function $u\in H^{3/2}(\Omega)$ is such that 
$\Delta u\in H^{1+s}(\Omega)$. Then, by the second line in \eqref{2.6}, 
\begin{eqnarray}\label{Mam-15}
\gamma_D u\in H^{1-\varepsilon}(\partial\Omega)\, \mbox{ for every }\,
\varepsilon>0.
\end{eqnarray}

To continue, we recall the discussion (results and notation) in the paragraph
containing \eqref{def-TAU}--\eqref{A.64} in Appendix \ref{sA}.
For every $j,k\in\{1,...,n\}$, we now claim that 
\begin{eqnarray}\label{Mam-17}
\frac{\partial(\gamma_D u)}{\partial\tau_{j,k}}=
\nu_j\gamma_D(\partial_k u)-\nu_k\gamma_D(\partial_j u).
\end{eqnarray}
Since the functions $\partial_ju,\partial_k u$ belong to the space 
$\{w\in H^{1/2}(\Omega)\,|\,\Delta w\in H^{s}(\Omega)\bigr\}$, we may then 
conclude from \eqref{Mam-17} and \eqref{Mam-1} that 
\begin{eqnarray}\label{Mam-18}
\frac{\partial(\gamma_D u)}{\partial\tau_{j,k}}
\in L^2(\partial\Omega;d^{n-1}\omega), 
\end{eqnarray}
and, in addition, 
\begin{eqnarray}\label{Mam-19}
\Bigl\|\frac{\partial(\gamma_D u)}{\partial\tau_{j,k}}
\Bigr\|_{L^2(\partial\Omega;d^{n-1}\omega)}
\leq C\bigl(\|u\|_{H^{3/2}(\Omega)}+\|\Delta u\|_{H^{1+s}(\Omega)}\bigr),
\end{eqnarray}
for every $j,k\in\{1,...,n\}$. In concert with \eqref{Mam-19} and 
\eqref{Mam-15}, the characterization in (\ref{A.64}) then entails that 
$\gamma_D u\in H^1(\partial\Omega)$ and 
$\|\gamma_D u\|_{H^1(\partial\Omega)}\leq 
C\bigl(\|u\|_{H^{3/2}(\Omega)}+\|\Delta u\|_{H^{1+s}(\Omega)}\bigr)$. 
In summary, the proof of the claims made about \eqref{Mam-2} is finished, 
modulo establishing \eqref{Mam-17}. 

To deal with \eqref{Mam-17}, let $\psi\in C^\infty_0(\bbR^n)$ 
and fix $j,k\in\{1,...,n\}$. Consider next the vector fields 
\begin{eqnarray}\label{Mam-21}
\begin{array}{l}
F_{j,k}=\bigl(0,...,0,u\partial_k\psi,0...,0,-u\partial_j\psi,0...,0\bigr),
\\[4pt]
G_{j,k}=\bigl(0,...,0,\psi\partial_k u,0...,0,-\psi\partial_j u,0...,0\bigr),
\end{array}
\end{eqnarray}
with the nonzero components on the $j$-th and $k$-th slots. 
Then $F_{j,k},\,G_{j,k}$ have components in the space 
$\{u\in H^{1/2}(\Omega)\,|\,\Delta u\in H^{s}(\Omega)\}$ with 
$s>-3/2$ and satisfy
\begin{eqnarray}\label{Mam-22}
{\rm div} (F_{j,k})=-{\rm div} (G_{j,k})
=(\partial_ju\partial_k\psi-\partial_ku\partial_j\psi) \in L^2(\Omega;d^nx),
\end{eqnarray}
in the sense of distributions. Also, 
\begin{eqnarray}\label{Mam-23}
\begin{array}{l}
\nu\cdot\wti\gamma_D(F_{j,k})
=(\gamma_Du)\bigl(\nu_k\partial_j\psi-\nu_j\partial_k\psi\bigr),
\\[4pt]
\nu\cdot\wti\gamma_D(G_{j,k})=\psi\bigl(\nu_k\gamma_D(\partial_j u)
-\nu_j\gamma_D(\partial_ku)\bigr).
\end{array}
\end{eqnarray}
Hence, using \eqref{Ht-r3}, we obtain 
\begin{align}\label{Mam-24}
\int_{\partial\Omega}d^{n-1}\omega\,(\gamma_Du)
\bigl(\nu_k\partial_j\psi-\nu_j\partial_k\psi\bigr)
&= \int_{\partial\Omega}d^{n-1}\omega\,\nu\cdot\wti\gamma_D(F_{j,k})
\nonumber\\
&=\int_{\Omega}d^nx\,{\rm div} (F_{j,k})
=-\int_{\Omega}d^nx\,{\rm div} (G_{j,k})
\nonumber\\
&= -\int_{\partial\Omega}d^{n-1}\omega\,\psi\bigl(\nu_k\gamma_D(\partial_j u)
-\nu_j\gamma_D(\partial_ku)\bigr).
\end{align}
This justifies \eqref{Mam-17} and shows that the operator \eqref{Mam-2} 
is well-defined and bounded. Clearly, this acts in a compatible fashion 
with \eqref{2.6} and \eqref{Mam-1}. To finish the proof of Lemma \ref{Gam-L1},
there remains to show that this operator also has a bounded, linear, 
right-inverse. This, however, is a consequence of the well-posedness 
of the boundary value problem 
\begin{eqnarray}\lb{Kaz}
u\in H^{3/2}(\Om),\quad\Delta u=0\mbox{ in }\Omega,\quad\ga_D(u)=f\in H^1(\dOm),
\end{eqnarray}
a result which appears in \cite{Ve84}.
\end{proof}

Next, we introduce the operator $\ga_N$ (the strong Neumann trace) by
\begin{align}\lb{2.7} 
\ga_N = \nu\cdot\ga_D\nabla \colon H^{s+1}(\Om)\to \LdOm, \quad 1/2<s<3/2, 
\end{align}
where $\nu$ denotes the outward pointing normal unit vector to
$\partial\Om$. It follows from \eqref{2.6} that $\ga_N$ is also a
bounded operator. We seek to extend the action of the Neumann trace
operator \eqref{2.7} to other (related) settings. To set the stage, 
assume Hypothesis~\ref{h2.1} and recall that the inclusion 
\begin{equation}\lb{inc-1}
\iota:H^s(\Omega)\hookrightarrow \bigl(H^1(\Omega)\bigr)^*,\quad s>-1/2,
\end{equation}
is well-defined and bounded. We then introduce the weak Neumann trace 
operator 
\begin{equation}\lb{2.8}
\wti\ga_N\colon\big\{u\in H^1(\Om)\,\big|\,\Delta u\in H^s(\Om)\big\} 
\to H^{-1/2}(\dOm),\quad s>-1/2,  
\end{equation}
as follows: Given $u\in H^1(\Om)$ with $\Delta u \in H^s(\Om)$ 
for some $s>-1/2$, we set (with $\iota$ as in \eqref{inc-1})
\begin{align} \lb{2.9}
\langle \phi, \wti\ga_N u \rangle_{1/2}
=\int_\Om d^n x\,\ol{\nabla \Phi(x)} \cdot \nabla u(x)  
+ {}_{H^1(\Om)}\langle \Phi, \iota(\Delta u)\rangle_{(H^1(\Om))^*}, 
\end{align}
for all $\phi\in H^{1/2}(\dOm)$ and $\Phi\in H^1(\Om)$ such that
$\ga_D\Phi = \phi$. We note that this definition is
independent of the particular extension $\Phi$ of $\phi$, and that
$\wti\ga_N$ is a bounded extension of the Neumann trace operator
$\ga_N$ defined in \eqref{2.7}.

The end-point case $s=1/2$ of \eqref{2.7} is discussed separately below. 

\begin{lemma}\label{Neu-tr}
Assume Hypothesis \ref{h2.1}. 
Then the Neumann trace operator \eqref{2.7} also extends to
\begin{eqnarray}\label{MaX-1}
\wti\gamma_N:\bigl\{u\in H^{3/2}(\Omega)\,\big|\,\Delta u\in L^2(\Omega;d^nx)\bigr\}
\to L^2(\partial\Omega;d^{n-1}\omega) 
\end{eqnarray}
in a bounded fashion when the space 
$\{u\in H^{3/2}(\Omega)\,|\,\Delta u\in L^2(\Omega;d^nx)\bigr\}$ is equipped
with the natural graph norm $u\mapsto \|u\|_{H^{3/2}(\Omega)}
+\|\Delta u\|_{L^2(\Omega;d^nx)}$. This extension is compatible 
with \eqref{2.8} and has a linear, bounded, right-inverse $($hence, 
as a consequence, it is onto$)$.  

Moreover, the Neumann trace operator \eqref{2.7} further extends to
\begin{eqnarray}\label{MaX-1U}
\wti\gamma_N:\bigl\{u\in H^{1/2}(\Omega)\,\big|\,\Delta u\in L^2(\Omega;d^nx)\bigr\}
\to H^{-1}(\dOm) 
\end{eqnarray}
in a bounded fashion when the space 
$\{u\in H^{1/2}(\Omega)\,|\,\Delta u\in L^2(\Omega;d^nx)\bigr\}$ is equipped
with the natural graph norm $u\mapsto \|u\|_{H^{1/2}(\Omega)}
+\|\Delta u\|_{L^2(\Omega;d^nx)}$. Once again, this extension is compatible 
with \eqref{2.8} and has a linear, bounded, right-inverse $($thus, in particular,
it is onto$)$.  
\end{lemma}
\begin{proof}
Fix $\psi\in C^\infty(\overline{\Omega})$. 
Applying \eqref{Ht-r3} to the vector field $G=\overline{\psi}\nabla u$ yields 
\begin{eqnarray}\label{MaX-2}
\int_{\partial\Omega}d^{n-1}\omega\,\overline{\psi}\,\nu\cdot
\gamma_D(\nabla u)
=\int_{\Omega}dx^n\,\overline{\nabla\psi}\cdot\nabla u
+\int_{\Omega}dx^n\,\overline{\psi}\,\Delta u.
\end{eqnarray}
Consider now $\phi\in H^{1/2}(\dOm)$ and $\Phi\in H^1(\Om)$ such that
$\ga_D\Phi = \phi$. Since $C^\infty(\overline{\Omega})\hookrightarrow H^1(\Om)$
is dense, it is possible to select a sequence 
$\psi_j\in C^\infty(\overline{\Omega})$, $j\in\bbN$, such that 
$\psi_j\to\Phi$ in $H^1(\Om)$ as $j\to\infty$. This entails 
$\nabla\psi_j\to\nabla\Phi$ in $L^2(\Om;d^nx)$ and 
$\psi_j|_{\partial\Omega}\to\phi$ in $H^{1/2}(\partial\Om)$ as $j\to\infty$.
Writing \eqref{MaX-2} for $\psi_j$ in place of $\psi$ and passing to the limit 
$j\to\infty$ then yields 
\begin{eqnarray}\label{MaX-3}
\int_{\partial\Omega}d^{n-1}\omega\,\overline{\phi}\,\nu\cdot
\gamma_D(\nabla u)
=\int_{\Omega}dx^n\,\overline{\nabla\Phi}\cdot\nabla u
+\int_{\Omega}dx^n\,\overline{\Phi}\,\Delta u.
\end{eqnarray}
This shows that the Neumann trace of $u$ in the sense of 
\eqref{2.8}, \eqref{2.9} is actually $\nu\cdot\gamma_D(\nabla u)$. 
In addition, 
\begin{align}\label{MaX-4}
\|\wti\gamma_N u\|_{L^2(\partial\Om;d^{n-1}\omega)}
&=\|\nu\cdot\gamma_D(\nabla u)\|_{L^2(\partial\Om;d^{n-1}\omega)}
\nonumber\\[4pt]
&\leq  \|\gamma_D(\nabla u)\|_{L^2(\partial\Om;d^{n-1}\omega)^n}
\nonumber\\
&\leq  C\bigl(\|\nabla u\|_{H^{1/2}(\Om)^n}
+\|\Delta(\nabla u)\|_{H^{-1}(\Om)^n}\bigr)
\nonumber\\
&= C\bigl(\|\nabla u\|_{H^{1/2}(\Om)^n}
+\|\nabla(\Delta u)\|_{H^{-1}(\Om)^n}\bigr)
\nonumber\\
&\leq  C\bigl(\|u\|_{H^{3/2}(\Om)}+\|\Delta u\|_{L^2(\Om;d^nx)}\bigr),
\end{align}
where we have used the boundedness of the Dirichlet trace operator in 
\eqref{Mam-1} with $s=-1$. This shows that, in the context of \eqref{MaX-1}, 
the Neumann trace operator 
\begin{eqnarray}\label{MaX-5}
\wti\gamma_N u=\nu\cdot\gamma_D(\nabla u)
\end{eqnarray}
has is well-defined, linear, bounded and is compatible 
with \eqref{2.8}. The fact that this has a linear, bounded, 
right-inverse is a consequence of the well-posedness result in 
Theorem \ref{t3.2}, proved later. 

As far as \eqref{MaX-1U} is concerned, let us temporarily introduce 
\begin{eqnarray}\label{MaX-2U}
\wti\gamma_n:\bigl\{u\in H^{1/2}(\Omega)\,\big|\,\Delta u\in L^2(\Omega;d^nx)\bigr\}
\to H^{-1}(\dOm)=\bigl(H^1(\dOm)\bigr)^*,
\end{eqnarray}
by setting 
\begin{eqnarray}\label{MaX-3U}
\langle \phi, \wti\ga_n u \rangle_{1}
= \langle \wti\ga_N(\Phi),\ga_D u \rangle_{0}
+\langle \Phi,\Delta u\rangle_{L^2(\Om;d^nx)}
-\langle \Delta\Phi,u\rangle_{L^2(\Om;d^nx)}, 
\end{eqnarray}
for all $\phi\in H^1(\dOm)$, where $\Phi\in H^{3/2}(\Om)$ is such that
$\ga_D\Phi = \phi$ and $\Delta\Phi\in L^2(\Om;d^nx)$. That such a $\Phi$ 
can be found (with the additional properties that the dependence 
$\phi\mapsto\Phi$ linear, and that $\Phi$ satisfies a natural estimate) 
is a consequence of the fact that the mapping \eqref{Mam-2} has a linear, 
bounded, right-inverse. 
Let us also note that the first pairing in the right hand-side of 
\eqref{MaX-3U} is meaningful, thanks to the first part of 
Lemma~\ref{Gam-L1} and what we have established in connection with 
\eqref{MaX-1}. 

We now wish to show that the definition \eqref{MaX-3U} is independent of 
the particular choice of $\Phi$. For this purpose, we recall the 
following useful approximation result: 
\begin{eqnarray}\lb{Tan-C26}
C^\infty(\ol{\Omega})\hookrightarrow 
\{u\in H^s(\Omega)\,|\,\Delta u\in L^2(\Omega;d^nx)\} \,\text{ densely, 
whenever $s<2$}, 
\end{eqnarray}
where the latter space is equipped with the natural graph norm 
$u\mapsto \|u\|_{H^s(\Omega)}+\|\Delta u\|_{L^2(\Omega;d^nx)}$. 
When $s=1$ this appears as Lemma~1.5.3.9 on p.\,60 of \cite{Gr85}, 
and the extension to $s<2$ has been worked out, along similar lines, in 
\cite{CD98}. Returning to the task ast hand, by linearity and density
is suffices to show that 
\begin{eqnarray}\label{MaX-4U}
\langle \wti\ga_N(\Phi),\ga_D u \rangle_{0}
+\langle \Phi,\Delta u\rangle_{L^2(\Om;d^nx)}
-\langle \Delta\Phi,u\rangle_{L^2(\Om;d^nx)}=0
\end{eqnarray}
whenever $\Phi\in H^{3/2}(\Om)$ is such that
$\ga_D\Phi = 0$, $\Delta\Phi\in L^2(\Om;d^nx)$, and 
$u\in C^\infty(\ol{\Om})$. Note, however, that by \eqref{2.9} with the
roles of $\Phi$ and $u$ reversed we have 
\begin{align} \lb{2.9U}
\langle \wti\ga_N(\Phi),\ga_D u\rangle_{0}
=\int_\Om d^n x\,\ol{\nabla \Phi(x)} \cdot \nabla u(x)  
+\langle \Delta\Phi,u\rangle_{L^2(\Om;d^nx)}, 
\end{align}
so matters are reduce to showing that
\begin{eqnarray}\label{MaX-5U}
\int_\Om d^n x\,\ol{\nabla \Phi(x)} \cdot \nabla u(x)  
=-\langle \Phi,\Delta u\rangle_{L^2(\Om;d^nx)}.
\end{eqnarray}
Nonetheless, this is a consequence of Green's formula \eqref{Ht-r3} written 
for the vector field $G= \ol{\Phi}\nabla u$ (which has the 
property that $\ga_DG=0$). In summary, the operator 
\eqref{MaX-2U}, \eqref{MaX-3U} is well-defined, linear and bounded. 

Next, we will show that this operator is compatible with 
\eqref{2.8}, \eqref{2.9}. After re-denoting $\wti\ga_n$ by $\wti\ga_N$, then 
this becomes the extension of the weak Neumann trace operator, considered 
in \eqref{MaX-1U}. To this end, assume that 
$u\in H^1(\Om)$ has $\Delta u\in L^2(\Om;d^nx)$. Our goal is to show that 
\begin{eqnarray}\lb{2.9Hg}
\langle \phi, \wti\ga_N u \rangle_{1/2}
=\langle \phi, \wti\ga_n u \rangle_{1}
\end{eqnarray}
for every $\phi\in H^1(\dOm)$ or, equivalently, 
\begin{eqnarray}\lb{2.9Hj}
\int_\Om d^n x\,\ol{\nabla \Phi(x)} \cdot \nabla u(x)  
= \langle \wti\ga_N(\Phi),\ga_D u \rangle_{0}
-\langle \Delta\Phi,u\rangle_{L^2(\Om;d^nx)}, 
\end{eqnarray}
for $\Phi\in H^{3/2}(\Om)$ such that $\Delta\Phi\in L^2(\Om;d^nx)$.
However, 
\begin{eqnarray}\lb{2.9Hs}
\langle \wti\ga_N(\Phi),\ga_D u \rangle_{0}
=\langle \wti\ga_N(\Phi),\ga_D u \rangle_{1/2}
=\int_\Om d^n x\,\ol{\nabla \Phi(x)} \cdot \nabla u(x)  
+\langle \Delta\Phi,u\rangle_{L^2(\Om;d^nx)}, 
\end{eqnarray}
where the first equality is a consequence of 
what we have proved about the operator \eqref{MaX-1}, and the 
second follows from \eqref{2.9} with the roles of $u$ and $\Phi$ reversed.
This justifies \eqref{2.9Hj} and finishes the proof of the lemma. 
\end{proof}

For future purposes, we shall need yet another extension of the concept of 
Neumann trace. This requires some preparations 
(throughout, Hypothesis~\ref{h2.1} is enforced). First, we recall that, 
as is well-known (see, e.g., \cite{JK95}), one has the natural identification 
\begin{eqnarray}\lb{jk-9}
\big(H^{1}(\Om)\big)^*\equiv
\big\{u\in H^{-1}(\bbR^n)\,\big|\, \supp \, (u)\subseteq\overline{\Omega}\big\}.
\end{eqnarray} 
Note that the latter is a closed subspace of $H^{-1}(\bbR^n)$. 
In particular, if $R_{\Omega}u= u|_{\Omega}$ denotes the operator 
of restriction to $\Omega$ (considered in the sense of distributions), then 
\begin{eqnarray}\lb{jk-10}
R_{\Omega}:\big(H^{1}(\Om)\big)^*\to H^{-1}(\Om)
\end{eqnarray}
is well-defined, linear and bounded. Furthermore, the 
composition of $R_\Omega$ in \eqref{jk-10} with $\iota$ from \eqref{inc-1}
is the natural inclusion of $H^s(\Om)$ into $H^{-1}(\Om)$. 
Next, given $z\in\bbC$, set 
\begin{eqnarray}\lb{2.88X}
W_z(\Om)= \bigl\{(u,f)\in H^1(\Om)\times\bigl(H^1(\Om)\bigr)^*\,\big|\,
(-\Delta-z)u=f|_{\Omega}\mbox{ in }\mathcal{D}'(\Om)\bigr\}, 
\end{eqnarray}
equipped with the norm inherited from 
$H^1(\Om)\times\bigl(H^1(\Om)\bigr)^*$. We then denote by
\begin{equation}\lb{2.8X}
\wti\ga_{\cN}\colon W_z(\Om)\to H^{-1/2}(\dOm)  
\end{equation}
the ultra weak Neumann trace operator defined by
\begin{align}\lb{2.9X}
\begin{split}
\langle\phi,\wti\ga_{\cN} (u,f)\rangle_{1/2}
&=  \int_\Om d^n x\,\ol{\nabla \Phi(x)} \cdot \nabla u(x)  \\ 
& \quad  -z\,\int_\Om d^n x\,\ol{\Phi(x)}u(x)  
-{}_{H^1(\Om)}\langle \Phi, f\rangle_{(H^1(\Om))^*}, \quad 
(u,f)\in W_z(\Om), 
\end{split}
\end{align}
for all $\phi\in H^{1/2}(\dOm)$ and $\Phi\in H^1(\Om)$ such that 
$\ga_D\Phi=\phi$. Once again, this definition is independent of the 
particular extension $\Phi$ of $\phi$. Also, as was the case of the Dirichlet 
trace, the ultra weak Neumann trace operator \eqref{2.8X}, \eqref{2.9X} 
is onto (this is a corollary of Theorem \ref{t3.XV}). For additional details 
we refer to equations \eqref{A.11}--\eqref{A.16}. 

The relationship between the ultra weak Neumann trace operator 
\eqref{2.8X}, \eqref{2.9X} and the weak Neumann trace operator 
\eqref{2.8}, \eqref{2.9} can be described as follows: Given 
$s>-1/2$ and $z\in\bbC$, denote by 
\begin{eqnarray}\lb{2.10X}
j_z:\{u\in H^1(\Om)\,\big|\,\Delta u\in H^s(\Om)\big\} \to W_z(\Om)
\end{eqnarray}
the injection 
\begin{eqnarray}\lb{2.11X}
j_z(u)= (u,\iota(-\Delta u -zu)),\quad u\in H^1(\Om),\; 
\Delta u\in H^s(\Om),
\end{eqnarray}
where $\iota$ is as in \eqref{inc-1}. Then 
\begin{eqnarray}\lb{2.12X}
\wti\gamma_{\cN}\circ j_z=\wti\ga_N.
\end{eqnarray}
Thus, from this perspective, $\wti\ga_{\cN}$ can also be regarded as a bounded 
extension of the Neumann trace operator $\ga_N$ defined in \eqref{2.7}.


Moving on, we now wish to discuss generalized Robin Laplacians in 
Lipschitz subdomains of $\bbR^n$. Before initiating this discussion 
in earnest, however, we formulate and prove the following useful result:
\begin{lemma} \lb{l2.2a}
Assume Hypothesis~\ref{h2.1}. Then for every $\varepsilon >0$ 
there exists a $\beta(\varepsilon)>0$ 
$($$\beta(\varepsilon)\underset{\varepsilon\downarrow 0}{=}
\Oh(1/\varepsilon)$$)$ such that
\begin{equation}\lb{2.38c}
\|\gamma_D u\|_{\LdOm}^2 \le 
\varepsilon \|\nabla u\|_{\LOm^n}^2 + \beta(\varepsilon) \|u\|_{\LOm}^2 \, 
\text{ for all }\,u\in H^1(\Om).    
\end{equation}
\end{lemma}
\begin{proof} 
Since $\Omega$ is a bounded Lipschitz domain, there exists an  
$h\in C_0^\infty(\bbR^n)^n$ with real-valued components and 
$\kappa>0$ such that (cf., \cite[Lemma 1.5.1.9, p.\ 40]{Gr85})
\begin{equation}
(\nu \cdot h)_{\bbC^n} \ge \kappa \, \text{ a.e.\ on $\dOm$.}
\end{equation}
Thus,
\begin{align}
\|\gamma_D u\|_{\LdOm}^2 & \le \f{1}{\kappa} \int_{\dOm} d^{n-1} \omega \, 
(\nu \cdot h)_{\bbC^n} |\gamma_D u|^2  \no \\
& = \f{1}{\kappa} \int_{\Om} d^n x \, {\rm div} (|u|^2 h),   \no \\
& = \f{1}{\kappa} \bigg(\int_{\Om} d^n x \, \big(\nabla |u|^2,h\big)_{\bbC^n} 
+ |u|^2 {\rm div} (h)\bigg),   \quad u \in H^1(\Om),   \lb{2.38g}
\end{align}
using the divergence theorem in the second step. 
Since for arbitrary $\varepsilon>0$, 
\begin{equation}
|2u \nabla u| \le \varepsilon |\nabla u|^2 +(1/\varepsilon) |u|^2, 
\quad u \in H^1(\Om),
\end{equation}
and $h\in C_0^\infty(\bbR^n)^n$, one arrives at \eqref{2.38c}. 
\end{proof}

Next we describe a family of self-adjoint Laplace operators 
$-\Delta_{\Theta,\Om}$ in $L^2(\Om; d^n x)$ indexed by the boundary 
operator $\Theta$. We will refer to $-\Delta_{\Theta,\Om}$ as the 
generalized Robin Laplacian. 

\begin{theorem}  \lb{t2.3}
Assume Hypothesis \ref{h2.2}. Then the generalized Robin Laplacian, 
 $-\Delta_{\Theta,\Om}$, defined by 
\begin{equation}
-\Delta_{\Theta,\Om} = -\Delta, \quad \dom(-\Delta_{\Theta,\Om}) = 
\big\{u\in H^1(\Om)\,\big|\, \Delta u \in L^2(\Om;d^nx); \, 
\big(\wti\gamma_N + \wti \Theta \gamma_D\big) u =0 
\text{ in $H^{-1/2}(\dOm)$}\big\},  
\lb{2.20}
\end{equation}
is self-adjoint and bounded from below in $L^2(\Om;d^nx)$. Moreover,
\begin{equation}
\dom\big(|-\Delta_{\Theta,\Om}|^{1/2}\big) = H^1(\Om).   \lb{2.21}
\end{equation}
\end{theorem}
\begin{proof}
We introduce the sesquilinear form $a_{-\Delta_{\Theta,\Om}}(\dott,\dott)$ 
on the domain $H^1(\Om) \times H^1(\Om)$ by 
\begin{equation}
a_{-\Delta_{\Theta,\Om}}(u,v) 
= a_{-\Delta_{0,\Om}}(u,v) + 
\big\langle \gamma_D u, \wti \Theta \gamma_D v \big\rangle_{1/2}, 
\quad u, v \in H^1(\Om),   \lb{2.22}
\end{equation}
where $a_{-\Delta_{0,\Om}}(\dott,\dott)$ on $H^1(\Om) \times H^1(\Om)$ 
denotes the Neumann Laplacian form
\begin{equation}
a_{-\Delta_{0,\Om}}(u,v) 
= \int_{\Om} d^nx \, \ol {(\nabla u)(x)}\cdot(\nabla v)(x), 
\quad u, v \in H^1(\Om).   \lb{2.23}
\end{equation}
One verifies that $a_{-\Delta_{\Theta,\Om}}(\dott,\dott)$ is well-defined on 
$H^1(\Om)\times H^1(\Om)$ since 
$\gamma_D\in\cB\big(H^1(\Om),H^{1/2}(\dOm)\big)$,  
$\wti \Theta \in \cB\big(H^{1/2}(\dOm),H^{-1/2}(\dOm)\big)$, and 
$(\Theta+(1-c_{\Theta})I_{\dOm})^{1/2}\in\cB\big(H^{1/2}(\dOm),\LdOm\big)$ 
(cf.\ \eqref{B.40}). This also implies that 
\begin{equation}\lb{2.24}
(\Theta+(1-c_{\Theta})I_{\dOm})^{1/2}\gamma_D \in \cB\big(H^1(\Om),\LdOm\big).
\end{equation}
Employing \eqref{2.1} and \eqref{2.2}, $a_{-\Delta_{\Theta,\Om}}$ is 
symmetric and bounded from below by $c_{\Theta}$. Next, we intend to show 
that $a_{-\Delta_{\Theta,\Om}}$ is a closed form in $\LOm\times\LOm$. 
For this purpose we rewrite 
$\big\langle\gamma_D u,\wti\Theta\gamma_D v\big\rangle_{1/2}$ as 
\begin{align}
& \big\langle \gamma_D u, \wti \Theta \gamma_D v \big\rangle_{1/2}  \no \\
& \quad = \big((\Theta+(1-c_{\Theta})I_{\dOm})^{1/2}\gamma_D u,
(\Theta+(1-c_{\Theta})I_{\dOm})^{1/2}\gamma_D v\big)_{\LdOm}  \no \\
& \qquad -(1- c_{\Theta}) (\gamma_D u,\gamma_D v)_{\LdOm}, 
\quad u, v \in H^1(\Om),   \lb{2.25}
\end{align}
(cf.\ \eqref{B.28}, \eqref{B.28a}), and notice that the last form on the 
right-hand side of \eqref{2.25} is nonclosable in $\LOm$ since $\gamma_D$ 
is nonclosable as an operator defined on a dense subspace from $\LOm$ into 
$\LdOm$ (cf.\ the discussion in connection with \eqref{B.42}). 

To deal with this noncloseability issue, we now split off the last form 
on the right-hand side of \eqref{2.25} and hence introduce
\begin{align} 
& b_{-\Delta_{\Theta,\Om}}(u,v) = (\nabla u, \nabla v)_{\LOm^n}  \no \\
& \quad + \big((\Theta+(1-c_{\Theta})I_{\dOm})^{1/2}\gamma_D u,
(\Theta+(1-c_{\Theta})I_{\dOm})^{1/2}\gamma_D v\big)_{\LdOm}   \no \\
& \quad + d_b (u,v)_{\LOm}, 
\quad u, v \in H^1(\Om),   \lb{2.26}
\end{align}
for $d_b>0$. Then due to the nonnegativity of the second form on the 
right-hand side in \eqref{2.26}, $b_{-\Delta_{\Theta,\Om}}$ is 
$H^1(\Om)$-coercive, that is, for some $c_1>0$, 
\begin{equation}
b_{-\Delta_{\Theta,\Om}}(u,u) \ge c_1 \|u\|_{H^1(\Om)}^2,   \lb{2.27}
\end{equation}
where $ \|u\|_{H^1(\Om)}^2 = \|\nabla u\|_{\LOm^n}^2+\|u\|_{\LOm}^2$. 
Next, we note that by \eqref{2.24}, 
\begin{align}
& \Big|\big((\Theta+(1-c_{\Theta})I_{\dOm})^{1/2}\gamma_D u,
(\Theta+(1-c_{\Theta})I_{\dOm})^{1/2}\gamma_D v\big)_{\LdOm}\Big| \lb{2.28} \\
& \quad \le \big\|(\Theta +(1 -c_{\Theta})I_{\dOm})^{1/2} 
\gamma_D\big\|_{\cB(H^1(\Om),\LdOm)}^2 
\|u\|_{H^1(\Om)} \|v\|_{H^1(\Om)}, \quad u, v \in H^1(\Om).   \no 
\end{align}
Since trivially, $\|\nabla u\|_{\LOm}^2 
+d_b\|u\|_{\LOm}^2\le c\|u\|_{H^1(\Om)}^2$ for some $c>0$, one infers 
that $b_{-\Delta_{\Theta,\Om}}$ is also $H^1(\Om)$-bounded, that is, for 
some $c_2>0$, 
\begin{equation}
b_{-\Delta_{\Theta,\Om}}(u,u) \le c_2 \|u\|_{H^1(\Om)}^2.    \lb{2.29}
\end{equation}
Thus, the symmetric form $b_{-\Delta_{\Theta,\Om}}$ is $H^1(\Om)$-bounded 
and $H^1(\Om)$-coercive and hence densely defined and closed in 
$\LOm \times \LOm$ by the discussion following \eqref{B.43a}.

To deal with the nonclosable form $(\gamma_D u,\gamma_D v)_{\LdOm}$,  
$u, v \in H^1(\Om)$, it suffices to note that by Lemma \ref{l2.2a} this 
form is infinitesimally bounded with respect to the Neumann Laplacian form 
$a_{-\Delta_{0,\Om}}$ on $H^1(\Om) \times H^1(\Om)$, and since the form   
$\big((\Theta+1-c_{\Theta})^{1/2}\gamma_D u,
(\Theta+1-c_{\Theta})^{1/2}\gamma_D v\big)_{\LdOm}$, $u, v \in H^1(\Om)$, 
is nonnegative, the form $(\gamma_D u,\gamma_D v)_{\LdOm}$ is also 
infinitesimally bounded with respect to the form $b_{-\Delta_{\Theta,\Om}}$. 
By the discussion in connection with \eqref{B.45}, \eqref{B.46}, the form 
$a_{-\Delta_{\Theta,\Om}}$ (possibly shifted by an irrelevant real constant) 
defined on $H^1(\Om) \times H^1(\Om)$, is thus densely defined in 
$\LOm \times \LOm$, bounded from below, and closed.

According to \eqref{B.30} we thus introduce the operator 
$-\Delta_{\Theta,\Om}$ in $L^2(\Om;d^n x)$ by
\begin{align}
& \dom(-\Delta_{\Theta,\Om}) = \bigg\{v \in H^1(\Om)\,\bigg|\, \, 
\text{there exists an $w_v \in L^2(\Om;d^n x)$ such that}  \no \\
& \quad \int_{\Om} d^n x \, \ol{\nabla w} \, \nabla v 
+ \big\langle \gamma_D w, \wti \Theta \gamma_D v \big\rangle_{1/2}  
=  \int_{\Om} d^n x \, \ol{w} w_v
\text{ for all $w\in H^1(\Om)$}\bigg\},    \no  \\
& -\Delta_{\Theta,\Om} u = w_u, \quad u \in \dom(-\Delta_{\Theta,\Om}).
\lb{2.31} 
\end{align} 
By the formalism displayed in \eqref{B.19}--\eqref{B.40} 
(cf., in particular \eqref{B.25}), $-\Delta_{\Theta,\Om}$ is self-adjoint 
in $L^2(\Om;d^n x)$ and \eqref{2.21} holds. We recall that
\begin{equation}\label{H-zer}
H_0^1(\Om)=\{u\in H^1(\Om)\,|\,\gamma_Du=0\mbox{ on }\partial\Omega\}.
\end{equation}
Taking $v\in C_0^\infty(\Omega) \hookrightarrow H^1_0(\Om) 
\hookrightarrow H^1(\Om)$, one concludes   
\begin{equation}\lb{2.32}
\int_{\Om}d^nx\,{\ol v}w_u=-\int_{\Om}d^nx\,{\ol v}\,\Delta u\,\
\text{ for all $v\in C_0^\infty(\Om)$, and hence }\,w_u 
= - \Delta u \, \text{ in } \, \cD^\prime(\Om),     
\end{equation}
with $\cD^\prime(\Om) = C_0^\infty(\Omega)^\prime$ the space of 
distributions in $\Om$.

Next, we suppose that $u \in \dom(-\Delta_{\Theta,\Om})$ and $v\in H^1(\Om)$. 
We recall that $\gamma_D\colon H^1(\Om) \to H^{1/2}(\dOm)$ and compute 
\begin{align}
\int_{\Om} d^n x \, \ol{\nabla v} \, \nabla u 
& =  - \int_{\Om} d^n x \, {\ol v} \, \Delta u + 
\langle \gamma_D v, \wti\gamma_N u \rangle_{1/2}    \no \\
& = \int_{\Om} d^n x \, {\ol v} w_u 
+ \big\langle \gamma_D v, \big(\wti\gamma_N + \wti \Theta \gamma_D\big) u 
\big\rangle_{1/2} 
- \big\langle \gamma_D v, \wti \Theta \gamma_D u \big\rangle_{1/2}   \no \\
& = \int_{\Om} d^n x \, \ol{\nabla v} \, \nabla u  
+ \big\langle \gamma_D v, \big(\wti\gamma_N + \wti \Theta \gamma_D\big) u 
\big\rangle_{1/2}, 
\lb{2.33}
\end{align} 
where we used the second line in \eqref{2.31}. Hence, 
\begin{equation}
\big\langle \gamma_D v, \big(\wti\gamma_N + \wti \Theta \gamma_D\big) u 
\big\rangle_{1/2} =0.   
\lb{2.34}
\end{equation}
Since $v\in H^1(\Om)$ is arbitrary, and the map 
$\gamma_D\colon H^1(\Om) \to H^{1/2}(\dOm)$ is actually onto, 
one concludes that 
\begin{equation}
\big(\wti\gamma_N + \wti \Theta \gamma_D\big) u = 0 
\, \text{ in } \, H^{-1/2}(\dOm).   
\lb{2.35}
\end{equation}
Thus,
\begin{equation}
\dom(-\Delta_{\Theta,\Om}) \subseteq \big\{v\in H^1(\Om)\,\big|\, 
\Delta v \in L^2(\Om; d^n x); 
\, \big(\wti\gamma_N + \wti \Theta \gamma_D\big) v = 0 
\text{ in } H^{-1/2}(\dOm)\big\}.  \lb{2.36}
\end{equation}
Finally, assume that $u\in \big\{v\in H^1(\Om)\,\big|\, 
\Delta v \in L^2(\Om; d^n x); \, \big(\wti\gamma_N 
+ \wti\Theta \gamma_D\big) v = 0\big\}$,  
$w\in H^1(\Om)$, and let $w_u =-\Delta u \in L^2(\Om; d^n x)$. Then, 
\begin{align}
\int_{\Om} d^n x \, {\ol w} w_u 
&= - \int_{\Om} d^n x \, {\ol w} \, {\rm div} (\nabla u)   \no \\
&= \int_{\Om} d^n x \, \ol{\nabla w} \, \nabla u 
- \langle \gamma_D w, \wti \gamma_N u \rangle_{1/2} \no \\
&= \int_{\Om} d^n x \, \ol{\nabla w} \, \nabla u 
+ \big\langle \gamma_D w, \wti \Theta \gamma_D u \big\rangle_{1/2}.    
\lb{2.37}
\end{align}
Thus, applying \eqref{2.31}, one concludes that 
$u \in \dom(-\Delta_{\Theta,\Om})$ and hence
\begin{equation}
\dom(-\Delta_{\Theta,\Om}) \supseteq \big\{v\in H^1(\Om)\,\big|\, 
\Delta v \in L^2(\Om; d^n x);\,\big(\wti\gamma_N+\wti\Theta\gamma_D\big) v=0 
 \text{ in } H^{-1/2}(\dOm)\big\},  \lb{2.38}
\end{equation}
finishing the proof of Theorem \ref{t2.3}.
\end{proof}

\begin{corollary}  \lb{c2.3A}
Assume Hypothesis \ref{h2.2}. Then the generalized Robin Laplacian, 
 $-\Delta_{\Theta,\Om}$, has purely discrete spectrum bounded from below, 
 in particular, 
\begin{equation}
\sigma_{\rm ess}(-\Delta_{\Theta,\Om}) = \emptyset.  
\lb{2.21a}
\end{equation}
\end{corollary}
\begin{proof} Since $\dom\big(|-\Delta_{\Theta,\Om}|^{1/2}\big) = H^1(\Om)$, 
by \eqref{2.21}, and $H^1(\Om)$ embeds compactly into $\LOm$ (cf., e.g., 
\cite[Theorem V.4.17]{EE89}), one infers that 
$(-\Delta_{\Theta,\Om}  +  I_{\Om})^{-1/2}\in \cB_\infty(\LOm)$. Consequently, 
one obtains 
\begin{equation}\lb{Nb-G1}
(-\Delta_{\Theta,\Om}  +  I_{\Om})^{-1}\in \cB_\infty(\LOm),
\end{equation}
which is equivalent to \eqref{2.21a}.
\end{proof}

The important special case where $\Theta$ corresponds to the operator of 
multiplication by a real-valued, essentially bounded function $\theta$ 
leads to Robin boundary conditions we discuss next:

\begin{corollary} \lb{c2.3a}
In addition to Hypothesis \ref{h2.1}, assume that $\Theta$ is the operator 
of multiplication in $L^2(\dOm; d^{n-1} \omega)$ by the real-valued 
function $\theta$ satisfying $\theta \in L^\infty(\dOm; d^{n-1} \omega)$. 
Then $\Theta$ satisfies the conditions in Hypothesis \ref{h2.2}
resulting in the self-adjoint and bounded from below Laplacian $-\Delta_{\theta,\Om}$ 
in $L^2(\Om; d^n x)$ 
with Robin boundary conditions on $\dOm$ in \eqref{2.20} given by
\begin{equation}
(\wti\gamma_N + \theta \gamma_D) u= 0\,\text{ in } \, H^{-1/2}(\dOm).   
\lb{2.38a}
\end{equation}
\end{corollary}
\begin{proof} By Lemma \ref{l2.2a}, the sesquilinear form 
\begin{equation} 
\langle \gamma_D u, \theta \gamma_D v \rangle_{1/2}, 
\quad u, v \in H^1(\Om),   \lb{2.38d}
\end{equation}
is infinitesimally form bounded with respect to the Neumann Laplacian form 
$a_{-\Delta_{0,\Om}}$. By \eqref{B.45} and \eqref{B.46} this in turn 
proves that the form $a_{-\Delta_{\Theta,\Om}}$ in \eqref{2.22} is closed 
and one can now follow the proof of Theorem \ref{t2.3} from \eqref{2.31} on, 
step by step. 
\end{proof}

\begin{remark}\lb{r2.4}
$(i)$ In the case of a smooth boundary $\partial\Om$, the boundary 
conditions in \eqref{2.38a} are also called ``classical'' boundary 
conditions (cf., e.g., \cite{Si78}); in the more general case of bounded 
Lipschitz domains we also refer to \cite{AW03} and \cite[Ch.\ 4]{Wa02} 
in this context. Next, we point out that, in \cite{LaSh}, the authors 
have dealt with the case of Laplace operators in bounded Lipschitz domains, 
equipped with local boundary conditions of Robin-type, with 
boundary data in $L^p(\dOm;d^{n-1}\omega)$, and produced nontangential 
maximal function estimates. For the case $p=2$, when our setting agrees
with that of \cite{LaSh}, some of our results in this section and the 
following are a refinement of those in \cite{LaSh}. Maximal $L^p$-regularity 
and analytic contraction semigroups of Dirichlet and Neumann Laplacians on 
bounded Lipschitz domains were studied in \cite{Wo07}. Holomorphic 
$C_0$-semigroups of the Laplacian with Robin boundary conditions on 
bounded Lipschitz domains have been discussed in \cite{Wa06}. 
Moreover, Robin boundary conditions for elliptic boundary value problems 
on arbitrary open domains were first studied by Maz'ya \cite{Ma81}, 
\cite[Sect.\ 4.11.6]{Ma85}, and subsequently in \cite{DD97} 
(see also \cite{Da00} which treats the case of the Laplacian). 
In addition, Robin-type boundary conditions involving measures on the 
boundary for very general domains $\Omega$ were intensively discussed 
in terms of quadratic forms and capacity methods in the literature, and 
we refer, for instance, to \cite{AW03}, \cite{AW03a}, \cite{BW06}, 
\cite{Wa02}, and the references therein.  \\
$(ii)$ In the special case $\theta=0$ (resp., $\wti \Theta =0$), that is, 
in the case of the Neumann Laplacian, we will also use the notation 
\begin{equation}
-\Delta_{N,\Om} = -\Delta_{0,\Om}.     \lb{2.38b}
\end{equation}
\end{remark}

The case of the Dirichlet Laplacian $-\Delta_{D,\Om}$ associated with 
$\Om$ formally corresponds to $\Theta =\infty$ and so we isolate it 
in the next result:

\begin{theorem}  \lb{t2.5}
Assume Hypothesis \ref{h2.1}. Then the Dirichlet Laplacian, 
$-\Delta_{D,\Om}$, defined by 
\begin{align}
-\Delta_{D,\Om} = -\Delta, \quad \dom(-\Delta_{D,\Om}) &= 
\big\{u\in H^1(\Om)\,\big|\, \Delta u \in L^2(\Om;d^n x); \, 
\gamma_D u =0 \text{ in $H^{1/2}(\dOm)$}\big\}   \no \\
&= \big\{u\in H_0^1(\Om)\,\big|\, \Delta u \in L^2(\Om;d^n x)\big\},  \lb{2.39}
\end{align}
is self-adjoint and strictly positive in $L^2(\Om;d^nx)$. Moreover,
\begin{equation}
\dom\big((-\Delta_{D,\Om})^{1/2}\big) = H^1_0(\Om).   \lb{2.40}
\end{equation}
\end{theorem}
\begin{proof}
We introduce the sesquilinear form $a_{D,\Om}(\dott,\dott)$ on the domain 
$H^1_0 (\Om) \times H^1_0 (\Om)$ by 
\begin{equation}
a_{D,\Om} (u,v) = \int_{\Om} d^nx \, \ol {(\nabla u)(x)} \, (\nabla v)(x), 
\quad u, v \in H^1_0(\Om).   \lb{2.41}
\end{equation}
Clearly, $a_{D,\Om}$ is symmetric, nonnegative, and well-defined on 
$H_0^1(\Om) \times H_0^1(\Om)$. 
Since $\Om$ is bounded, that is, $|\Om|<\infty$, $H_0^1(\Om)$-coercivity of 
$a_{D,\Om}$ then immediately follows from Poincar{\' e}'s inequality for 
$H_0^1(\Om)$-functions (cf., e.g., \cite[Theorem\ I.7.6]{Wl87}).

Next we introduce the operator $-\Delta_{D,\Om}$ in $L^2(\Om;d^n x)$ by
\begin{align}
& \dom(-\Delta_{D,\Om}) = \bigg\{v \in H_0^1(\Om)\,\bigg|\, \, 
\text{there exists an $w_v \in L^2(\Om;d^n x)$ such that}  \no \\
& \hspace*{2.8cm} \int_{\Om} d^n x \, \ol{\nabla w} \, 
\nabla v=\int_{\Om}d^nx\,\ol{w}w_v\text{ for all $w\in H_0^1(\Om)$}\bigg\},
\no  \\
& -\Delta_{D,\Om} u = w_u, \quad u \in \dom(-\Delta_{D,\Om}).    \lb{2.41a} 
\end{align}
By the formalism displayed in \eqref{B.1}--\eqref{B.15}, $-\Delta_{D,\Om}$ is 
self-adjoint in $L^2(\Om;d^n x)$ and \eqref{2.40} holds. Taking 
$v\in C_0^\infty(\Omega) \hookrightarrow H^1_0(\Om)$, one concludes   
\begin{equation}
\int_{\Om}d^nx\,{\ol v}w_u=-\int_{\Om}d^n x\,{\ol v}\,\Delta u \, \text{ in } 
\, \cD^\prime(\Om) \, 
\text{ and hence } \, w_u = - \Delta u \, \text{ in } \, \cD^\prime(\Om).     
\lb{2.41b}
\end{equation}

Since $v\in H^1_0(\Om)$ if and only if $v\in H^1(\Om)$ and $\gamma_D v=0$ in 
$H^{1/2} (\dOm)$ (cf., e.g., \cite[Corollary 1.5.1.6 ]{Gr85}), and 
$v\in \dom(-\Delta_{D,\Om})$ implies $\Delta v \in L^2(\Om; d^n x)$, 
one computes for $u \in \dom(-\Delta_{D,\Om})$ and $v\in H^1_0(\Om)$ that 
\begin{equation}
\int_{\Om} d^n x \, \ol{\nabla v} \, \nabla u 
= - \int_{\Om}  d^n x \, \ol v \Delta u 
= \int_{\Om} d^n x \, \ol v w_u.
\end{equation}
Thus, $w_u = - \Delta u \in L^2(\Om; d^n x)$ and hence 
\begin{equation}
\dom(-\Delta_{D,\Om}) \subseteq \big\{v\in H_0^1(\Om)\,\big|\, 
\Delta v \in L^2(\Om; d^n x)\big\}.  \lb{2.41f}
\end{equation}
Finally, assume that $u\in \big\{v\in H_0^1(\Om)\,\big|\, 
\Delta v \in L^2(\Om; d^n x)\big\}$,  
$w\in H_0^1(\Om)$, and let $w_u =-\Delta u \in L^2(\Om; d^n x)$. Then, 
\begin{equation}
\int_{\Om} d^n x \, {\ol w} w_u 
= - \int_{\Om} d^n x \, {\ol w} \, {\rm div} (\nabla u)  
= \int_{\Om} d^n x \, \ol{\nabla w} \, \nabla u,    \lb{2.41g}
\end{equation}
since $\gamma_D w =0$ in $L^2(\dOm; d^{n-1} \omega)$.
Thus, applying \eqref{2.41a}, one concludes that 
$u\in \dom(-\Delta_{D,\Om})$ and hence  
\begin{equation}
\dom(-\Delta_{D,\Om}) \supseteq \big\{v\in H_0^1(\Om)\,\big|\, 
\Delta v \in L^2(\Om; d^n x)\big\},  \lb{2.41h}
\end{equation}
finishing the proof of Theorem \ref{t2.5}.
\end{proof}

Since $\Om$ is open and bounded, it is well-known that $-\Delta_{D,\Om}$ has 
purely discrete spectrum contained in $(0,\infty)$, in particular,  
\begin{equation}
\sigma_{\rm ess}(-\Delta_{D,\Om})=\emptyset
\end{equation} 
(this follows from \eqref{2.40} since 
$H^1_0(\Om)$ embeds compactly into $\LOm$; the latter fact holds for arbitrary 
open, bounded sets $\Om\subset\bbR^n$, cf., e.g., \cite[Theorem V.4.18]{EE89}).

While the principal objective of this paper was to prove the results in 
this section and the subsequent for minimally smooth domains $\Omega$, 
it is of interest to study similar problems when Hypothesis \ref{h2.1} 
is further strengthen to: 

\begin{hypothesis} \lb{h2.8}
Let $n\in\bbN$, $n\geq 2$, and assume that $\Omega\subset{\bbR}^n$ is
a bounded domain of class $C^{1,r}$ for some $1/2 < r <1$. 
\end{hypothesis}

We refer to Appendix \ref{sA} for some details on $C^{1,r}$-domains.

Correspondingly, the natural strengthening of Hypothesis \ref{h2.2} reads: 

\begin{hypothesis} \lb{h2.9}
In addition to Hypothesis \ref{h2.2} and \ref{h2.8} assume that 
\begin{equation}
\wti \Theta \in \cB_{\infty}\big(H^{3/2}(\dOm), H^{1/2}(\dOm)\big).    
\lb{2.43}
\end{equation} 
\end{hypothesis}

We note that a sufficient condition for \eqref{2.43} to hold is
\begin{equation}
\wti \Theta \in \cB\big(H^{3/2-\varepsilon}(\dOm), H^{1/2}(\dOm)\big) \, 
\text{ for some } \, \varepsilon > 0.     \lb{2.45} 
\end{equation} 

\noindent{\bf Notational comment.} To avoid introducing an additional 
sub- or superscript into our notation of $-\Delta_{\Theta,\Om}$ and 
$-\Delta_{D,\Om}$, we will use the same symbol for these operators 
irrespective of whether the pair of Hypothesis \ref{h2.1} and \ref{h2.2} 
or the pair of Hypothesis \ref{h2.8} and \ref{h2.9} is involved. 
Our results will be carefully stated so that it is always evident 
which set of hypotheses is used.  

Next, we discuss certain regularity results for fractional powers of 
the resolvents of the Dirichlet and Robin Laplacians, first in Lipschitz
then in smoother domains. 

\begin{lemma} \lb{l2.6}
Assume Hypothesis \ref{h2.1} in connection with $-\Delta_{D,\Om}$ and 
Hypothesis \ref{h2.2} in connection with $-\Delta_{\Theta,\Om}$. Then 
the following boundedness properties hold for all $q\in [0,1]$ and 
$z\in\bbC\backslash[0,\infty)$, 
\begin{align}
(-\Delta_{D,\Om}-zI_{\Om})^{-q/2},\, (-\Delta_{\Theta,\Om}-zI_{\Om})^{-q/2}
\in\cB\big(\LOm,H^{q}(\Om)\big). \lb{2.41i}
\end{align}
\end{lemma}

The fractional powers in \eqref{2.41i} (and in subsequent analogous cases) 
are defined via the functional calculus implied by the spectral theorem 
for self-adjoint operators. As discussed in \cite[Lemma A.2]{GLMZ05} in the 
closely related situation of Lemma \ref{l2.12}, the key ingredients in 
proving Lemma \ref{l2.6} are the inclusions
\begin{equation}
\dom(-\Delta_{D,\Om}) \subset H^1(\Om), \quad
\dom(-\Delta_{\Theta,\Om}) \subset H^1(\Om)
\end{equation}
and real interpolation methods. The above results should be compared 
with its analogue for smoother domains. Specifically, we have: 

\begin{lemma} \lb{l2.12}
Assume Hypothesis \ref{h2.8} in connection with $-\Delta_{D,\Om}$ and 
Hypothesis \ref{h2.9} in connection with $-\Delta_{\Theta,\Om}$. Then the 
following boundedness properties hold for all $q\in [0,1]$ and 
$z\in\bbC\backslash[0,\infty)$, 
\begin{align}
(-\Delta_{D,\Om}-zI_{\Om})^{-q},\, (-\Delta_{\Theta,\Om}-zI_{\Om})^{-q}
\in\cB\big(\LOm,H^{2q}(\Om)\big). \lb{2.83}
\end{align}
\end{lemma}

As explained in \cite[Lemma\ A.2]{GLMZ05}, the key ingredients in 
proving Lemma \ref{l2.12} are the inclusions
\begin{equation}
\dom(-\Delta_{D,\Om}) \subset H^2(\Om), \quad
\dom(-\Delta_{\Theta,\Om}) \subset H^2(\Om)
\end{equation}
and real interpolation methods.

Moving on, we now consider mapping properties of powers of the 
resolvents of generalized Robin Laplacians multiplied (to the left) 
by the Dirichlet boundary trace operator:

\begin{lemma} \lb{l2.7}
Assume Hypothesis \ref{h2.1} and let $\eps > 0$, 
$z\in\bbC\backslash[0,\infty)$. Then,
\begin{align}
\ga_D (-\Delta_{\Theta,\Om}-zI_{\Om})^{-(1+\eps)/4} \in
\cB\big(\LOm,\LdOm\big).  \lb{2.42}
\end{align}
\end{lemma}

As in \cite[Lemma 6.9]{GLMZ05}, Lemma \ref{l2.7} follows from Lemma \ref{l2.6}
and from \eqref{2.6} and \eqref{2.7}. Once again, we wish to contrast this
with the corresponding result for smoother domains, recorded below. 

\begin{lemma} \lb{l2.13}
Assume Hypothesis \ref{h2.8} in connection with $-\Delta_{D,\Om}$ and 
Hypothesis \ref{h2.9} in connection with $-\Delta_{\Theta,\Om}$, and let 
$\eps>0$, $z\in\bbC\backslash[0,\infty)$. Then,
\begin{align}
\ga_N(-\Delta_{D,\Om}-zI_{\Om})^{-(3+\eps)/4},
\ga_D(-\Delta_{\Theta,\Om}-zI_{\Om})^{-(1+\eps)/4} \in
\cB\big(\LOm,\LdOm\big).  \lb{2.85}
\end{align}
\end{lemma}

As in \cite[Lemma 6.9]{GLMZ05}, Lemma \ref{l2.13} follows from 
Lemma \ref{l2.12} and from \eqref{2.6} and \eqref{2.7}.
In contrast to Lemma \ref{l2.13} 
under the stronger Hypothesis \ref{h2.9}, we cannot obtain an analog 
of \eqref{2.85} for $-\Delta_{D,\Om}$ under the weaker Hypothesis \ref{h2.1}.

The analog of Theorem \ref{t2.3} for smoother domains reads as follows:

\begin{theorem}  \lb{t2.10}
Assume Hypothesis \ref{h2.9}. Then the generalized Robin Laplacian,  
$-\Delta_{\Theta,\Om}$, defined by 
\begin{equation}
-\Delta_{\Theta,\Om} = -\Delta, \quad \dom(-\Delta_{\Theta,\Om}) 
=\big\{u\in H^2(\Om)\,\big|\,  
\big(\gamma_N +\wti\Theta\gamma_D\big) u =0 
\text{ in $H^{1/2} (\dOm)$}\big\},    
\lb{2.46}
\end{equation}
is self-adjoint and bounded from below in $L^2(\Om;d^nx)$. Moreover,
\begin{equation}
\dom\big(|-\Delta_{\Theta,\Om}|^{1/2}\big) = H^1(\Om).   
\lb{2.47}
\end{equation}
\end{theorem}
\begin{proof}
We adapt the proof of \cite[Lemma A.1]{GLMZ05}, dealing with the special 
case of Neumann boundary conditions (i.e., in the case $\wti \Theta =0$),  
to the present situation. For convenience of the reader we produce a 
complete proof below.

By Theorem \ref{t2.3}, the operator $T_{\Theta,\Om}$ in $L^2(\Om;d^nx)$, 
defined by 
\begin{equation}
T_{\Theta,\Om} = -\Delta, \quad \dom(T_{\Theta,\Om}) = 
\big\{u\in H^1(\Om)\,\big|\, \Delta u \in L^2(\Om;d^nx); \, 
\big(\wti\gamma_N + \wti \Theta \gamma_D\big) u =0 
\text{ in $H^{-1/2}(\dOm)$}\big\},  
\lb{2.47a}
\end{equation}
is self-adjoint and bounded from below, and 
\begin{equation}
\dom\big(|T_{\Theta,\Om}|^{1/2}\big) = H^1(\Om)   
\lb{2.47b}
\end{equation}
holds. Thus, we need to prove that $\dom(T_{\Theta,\Om}) \subseteq H^2(\Om)$.  
 
Consider $u\in\dom(T_{\Theta,\Om})$ and set 
$f= - \Delta u + u\in L^2(\Omega;d^nx)$. Viewing $f$ as an element in 
$\big(H^1(\Omega)\big)^*$, the classical Lax--Milgram Lemma implies 
that $u$ is the unique solution of the boundary-value problem
\begin{equation}\label{BVP}
\left\{
\begin{array}{l}
(-\Delta + I_{\Omega})u=f\in L^2(\Omega)\hookrightarrow
\bigl(H^{1}(\Omega)\bigr)^*, \\[.5mm]
u\in H^{1}(\Omega), \\[.5mm]
\big(\wti\gamma_N + \wti\Theta \gamma_D\big) u=0.
\end{array}
\right.
\end{equation}
One convenient way to actually show that  
\begin{equation}\label{goal}
u\in H^{2}(\Omega),
\end{equation}
is to use layer potentials. Specifically, let $E_n(z;x)$ be the 
fundamental solution of the Helmholtz differential expression 
$(-\Delta -z)$ in $\bbR^n$, $n\in\bbN$, $n\geq 2$, that is,
\begin{align}
& E_n(z;x) = \begin{cases}
(i/4) \big(2\pi |x|/z^{1/2}\big)^{(2-n)/2} H^{(1)}_{(n-2)/2} 
\big(z^{1/2}|x|\big), & n\geq 2, 
\; z\in\bbC\backslash \{0\}, \\
\f{-1}{2\pi} \ln(|x|), & n=2, \; z=0, \\ 
\f{1}{(n-2)\omega_{n-1}}|x|^{2-n}, & n\geq 3, \; z=0, 
\end{cases}     \lb{2.52} \\
& \hspace*{6.8cm} \Im\big(z^{1/2}\big)\geq 0,\; x\in\bbR^n\backslash\{0\}. 
\no
\end{align}
Here $H^{(1)}_{\nu}(\dott)$ denotes the Hankel function of the first kind 
with index $\nu\geq 0$ (cf.\ \cite[Sect.\ 9.1]{AS72}). 

We also define the associated single layer potential
\begin{equation}\label{sing-layer}
({\mathcal S_z}g)(x)=\int_{\partial\Omega}d^{n-1}\omega(y)\,E_n(z;x-y)g(y),
\quad x\in\Omega, \; z\in\bbC, 
\end{equation}
where $g$ is an arbitrary measurable function on $\partial\Omega$. As
is well-known (the interested reader may consult, e.g., \cite{MMT01},
\cite{Ve84} for jump relations in the context of Lipschitz domains), if
\begin{equation}\label{Ksharp}
(K^{\#}_zg)(x)={\rm p.v.}\int_{\partial\Omega}d^{n-1} \omega(y)\,
\partial_{\nu_x}E_n(z;x-y)g(y),
\quad x\in\partial\Omega, \; z\in\bbC, 
\end{equation}
stands for the so-called adjoint double layer on $\partial\Omega$,
the following jump formula holds
\begin{equation}\label{jump}
\wti\gamma_N {\mathcal S_{z}}g
=\big({\textstyle{-\frac12}}I_\dOm+K^{\#}_{z}\big)g.
\end{equation}
It should be noted that 
\begin{equation}
K^{\#}_z \in \cB\big(\LdOm), \quad z\in\bbC,   \lb{Kb}
\end{equation}
whenever $\Om$ is a bounded Lipschitz domain. See Lemma \ref{L-Kb}. 

Now, if we denote by $w$ the convolution of $f\in L^2(\Omega;d^nx)$
with $E_n(-1;\dott)$ in $\Omega$, then $w\in H^2(\Omega)$ and the 
solution $u$ of \eqref{BVP} is given by
\begin{equation}\label{sol}
u=w+{\mathcal S_{-1}}g
\end{equation}
for a suitably chosen function $g$ on $\partial\Omega$. Concretely, 
we shall then require that
\begin{equation} 
\big(\ga_N + \wti\Theta \gamma_D\big){\mathcal S_{-1}}g
= - \big(\ga_N + \wti\Theta \gamma_D\big)w, 
\end{equation}
or equivalently, 
\begin{equation}
\big({\textstyle{-\frac12}}I_\dOm+K^{\#}_{-1}\big)g 
+ \wti\Theta \gamma_D {\mathcal S_{-1}} g
= - \big(\ga_N + \wti\Theta \gamma_D\big)w\in H^{1/2}(\partial\Omega).
\end{equation}
By hypothesis, $\wti\Theta\in \cB_\infty\big(H^{3/2}(\Om),H^{1/2}(\dOm)\big)$ 
and hence 
\begin{equation}
\wti\Theta \gamma_D {\mathcal S_{-1}} \in 
\cB_\infty\big(H^{1/2}(\Om),H^{1/2}(\dOm)\big)
\end{equation}
as soon as one proves that ${\mathcal S_{-1}}$ satisfies 
\begin{equation}\label{goal-2}
{\mathcal S_{-1}}\in \cB\big(H^{1/2}(\partial\Omega), H^{2}(\Omega)\big). 
\end{equation}

To prove this, as a preliminary step we note (cf. \cite{MT00}) that
\begin{equation}\label{S-map}
{\mathcal S_{-1}}\colon H^{-s}(\partial\Omega)\rightarrow H^{-s+3/2}(\Omega)
\end{equation}
is well-defined and bounded for each $s\in [0,1]$, even when $\Omega$ 
is only a bounded Lipschitz domain. 
For a fixed, arbitrary $j\in\{1,...,n\}$, consider next the operator
$\partial_{x_j}{\mathcal S_{-1}}$ whose integral kernel is
$\partial_{x_j}E_n(-1;x-y)=-\partial_{y_j}E_n(-1;x-y)$. We write
\begin{equation}
\partial_{y_j}=\sum_{k=1}^n \nu_k(y)\nu_k(y)\partial_{y_j}
=\sum_{k=1}^n\nu_k(y)\frac{\partial}{\partial\tau_{k,j}(y)}
+\nu_j(y)\nu(y)\cdot\nabla_y
\end{equation}
where the tangential derivative operators 
$\partial/\partial\tau_{k,j}=\nu_k\partial_j-\nu_j\partial_k$,
$j,k=1,\dots,n$, satisfy \eqref{Pf-2}. Using the boundary integration by parts 
formula \eqref{Ibp-w} it follows that
\begin{equation}\label{imp-id}
\partial_j{\mathcal S_{-1}}h={\mathcal D_{-1}}(\nu_j h)+\sum_{k=1}^n
{\mathcal S_{-1}}\bigg(\frac{\partial(\nu_k h)}{\partial\tau_{k,j}}\bigg), 
\quad h\in H^{1/2}(\dOm), 
\end{equation}
where, for $z\in\bbC$,  
\begin{eqnarray}\label{Df-R1}
{\mathcal D_{z}}h(x)=\int_{\partial\Omega}d^{n-1}\omega(y)\,
\nu(y)\cdot\nabla_y[E_n(z;x-y)]h(y),\quad x\in\Omega,
\end{eqnarray}
is the so-called (acoustic) double layer potential operator. 
Its mappings properties on the scale of Sobolev spaces have been analyzed 
in \cite{MT00} and we note here that
\begin{equation}\label{D-map}
{\mathcal D_{-1}}\colon H^{s}(\partial\Omega)\rightarrow
H^{s+1/2}(\Omega),\quad 0\leq s\leq 1,
\end{equation}
requires only that $\Omega$ is Lipschitz.
Assuming that multiplication by (the components of) $\nu$ preserves
the space $H^{1/2}(\partial\Omega)$ (which is the case if, e.g.,
$\Om$ is of class $C^{1,r}$ for some $(1/2)<r< 1$; cf. Lemma \ref{lA.6}),
the desired conclusion about the operator \eqref{goal-2} follows from 
\eqref{S-map}, \eqref{imp-id} and \eqref{D-map}. 

Going further, from Theorem \ref{K-CPT} we know that 
\begin{equation}\label{goal-1}
K^{\#}_{-1}\in\cB_{\infty}\bigl(H^{1/2}(\partial\Omega)\bigr), 
\end{equation}
so 
${\textstyle{-\frac12}}I_\dOm+K^{\#}_{-1}+\wti\Theta\gamma_D{\mathcal S_{-1}}$ 
is a Fredholm operator in $H^{1/2}(\dOm)$ with index zero. This finishes the
proof of \eqref{goal}. Hence, the fact that 
$\dom(T_{\Theta,\Om})\subseteq H^{2}(\Omega)$ has been established.
\end{proof}

Again we isolate the Neumann Laplacian $-\Delta_{N,\Om}$, that is, 
the special case 
$\wti \Theta=0$ in \eqref{2.46}, under Hypothesis \ref{h2.8},
\begin{equation}
-\Delta_{N,\Om}=-\Delta, \quad 
\dom(-\Delta_{N,\Om})=\big\{u\in H^2(\Om)\,\big|\, \wti\ga_N u=0 
\text{ in } H^{1/2}(\dOm)\big\}.    \lb{2.73}
\end{equation}

Similarly, one can now treat the case of the Dirichlet Laplacian. 
This has originally been done under more general conditions on 
$\Om$ (assuming the boundary of $\Om$ to be compact rather than 
$\Om$ bounded) in \cite[Lemmas A.1]{GLMZ05}. For completeness we repeat 
the short argument below:

\begin{theorem}  \lb{t2.11}
Assume Hypothesis \ref{h2.8}. Then the Dirichlet Laplacian, 
$-\Delta_{D,\Om}$, defined by 
\begin{equation}
-\Delta_{D,\Om} = -\Delta, \quad \dom(-\Delta_{D,\Om}) = 
\big\{u\in H^2 (\Om)\,\big|\, \gamma_D u =0 \text{ in $H^{3/2}(\dOm)$}\big\},
\lb{2.48}
\end{equation}
is self-adjoint and strictly positive in $L^2(\Om; d^nx)$. Moreover,
\begin{equation}
\dom\big((-\Delta_{D,\Om})^{1/2}\big) = H^1_0(\Om).   \lb{2.49}
\end{equation}
\end{theorem}
\begin{proof}
For convenience of the reader we reproduce the short proof 
of \cite[Lemma A.1]{GLMZ05} in the special case of Dirichlet boundary 
conditions, given the proof of Theorem \ref{t2.10}.   

By Theorem \ref{t2.5}, the operator $T_{D,\Om}$ in $L^2(\Om;d^n x)$, 
defined by 
\begin{equation}
T_{D,\Om} = -\Delta, \quad \dom(T_{D,\Om}) = 
\big\{u\in H_0^1(\Om)\,\big|\, \Delta u \in L^2(\Om;d^n x); \, 
\gamma_D u =0 \text{ in $L^2(\dOm; d^{n-1} \omega)$}\big\},     \lb{2.49a}
\end{equation}
is self-adjoint and strictly positive, and 
\begin{equation}
\dom\big((T_{D,\Om})^{1/2}\big) = H^1_0(\Om)   \lb{2.49b}
\end{equation}
holds. Thus, we need to prove that $\dom(T_{D,\Om}) \subseteq H^2(\Om)$. 
To achieve this, we follow the proof of Theorem \ref{t2.10}, starting with 
the same representation \eqref{sol}. This time, the requirement on $g$ is 
that $\ga_D{\mathcal{S}_{-1}}g=h=\ga_Dw\in H^{3/2}(\partial\Omega)$. Thus, it 
suffices to know that
\begin{equation}\label{SS-iso}
\ga_D{\mathcal{S}_{-1}}\colon H^{1/2}(\partial\Omega)
\rightarrow H^{3/2}(\partial\Omega)
\end{equation}
is an isomorphism. When $\partial\Omega$ is of class $C^\infty$, it
has been proved in \cite[Proposition 7.9]{Ta96} that
$\ga_D{\mathcal{S}_{-1}}\colon H^{s}(\partial\Omega)\to 
H^{s+1}(\partial\Omega)$ is an
isomorphism for each $s\in\bbR$ and, if $\Om$ is of class $C^{1,r}$ with
$(1/2)<r< 1$, the validity range of this result is limited to
$-1-r<s<r$, which covers \eqref{SS-iso}. The latter fact follows from an
inspection of Taylor's original proof of \cite[Proposition 7.9]{Ta96}.
Here we just note that the only significant difference is that if
$\partial\Omega$ is of class $C^{1,r}$ (instead of class $C^\infty$), then
$S$ is a pseudodifferential operator whose symbol exhibits a limited
amount of regularity in the space-variable. Such classes of operators have
been studied in, e.g., \cite{MMT01}, \cite[Chs.\ 1, 2]{Ta91}.
\end{proof}

\begin{remark} \lb{r2.15}
We emphasize that all results in this section extend to closed 
Schr\"odinger operators 
\begin{equation}
H_{\Theta,\Om} = -\Delta_{\Theta,\Om} + V, \quad  
\dom\big(H_{\Theta,\Om}\big) = \dom\big(-\Delta_{\Theta,\Om}\big)  \lb{2.86} 
\end{equation} 
for (not necessarily real-valued) potentials $V$ satisfying 
$V \in L^\infty(\Om; d^n x)$, consistently replacing $-\Delta$ by 
$-\Delta + V$, etc. More generally, all results extend to Kato--Rellich 
bounded potentials $V$ relative to $-\Delta_{\Theta,\Om}$ with bound less 
than one. Extensions to potentials permitting stronger local singularities, 
and an extensions to (not necessarily bounded) Lipschitz domains with 
compact boundary, will be pursued elsewhere.
\end{remark}

\section{Generalized Robin and Dirichlet Boundary Value Problems \\
and Robin-to-Dirichlet and Dirichlet-to-Robin Maps} \label{s3}

This section is devoted to generalized Robin and Dirichlet boundary 
value problems associated with the Helmholtz differential expression 
$-\Delta - z$ in connection with the open set $\Omega$. In addition, 
we provide a detailed discussion of Robin-to-Dirichlet maps, 
$M_{\Theta,D,\Om}^{(0)}$,  in $\LdOm$.

In this section we strengthen Hypothesis \ref{h2.2} by adding assumption 
\eqref{3.5} below: 

\begin{hypothesis} \lb{h3.1}
In addition to Hypothesis \ref{h2.2} suppose that 
\begin{equation}\lb{3.5}
\wti \Theta \in \cB_{\infty}\big(H^1(\dOm),L^2(\dOm;d^{n-1} \omega)\big).   
\end{equation}
\end{hypothesis}

\noindent We note that \eqref{3.5} is satisfied whenever there 
exists some $\varepsilon>0$ such that
\begin{equation}\lb{3.5bis}
\wti \Theta \in \cB\big(H^{1-\varepsilon}(\dOm),L^2(\dOm;d^{n-1} \omega)\big). 
\end{equation}

We recall the definition of the weak Neumann trace operator $ \wti\ga_N$ in 
\eqref{2.8}, \eqref{2.9} and start with the Helmholtz Robin boundary value 
problems:

\begin{theorem} \lb{t3.2} 
Assume Hypothesis \ref{h3.1} and suppose that 
$z\in\bbC\backslash\si(-\Delta_{\Theta,\Om})$. Then for every $g\in\LdOm$, 
the following generalized Robin boundary value problem,
\begin{equation} \lb{3.6}
\begin{cases}
(-\Delta - z)u = 0 \text{ in }\,\Om,\quad u \in H^{3/2}(\Om), \\
\big(\wti\ga_N + \wti \Theta \gamma_D\big) u = g \text{ on } \,\dOm,
\end{cases}
\end{equation}
has a unique solution  $u=u_\Theta$. This solution $u_\Theta$ satisfies
\begin{eqnarray}\lb{3.6a}
\begin{array}{l}
\ga_D u_\Theta \in H^1(\dOm), \quad \wti \ga_N u_\Theta 
\in L^2(\dOm;d^{n-1}\omega),
\\[4pt]
\|\ga_D u_\Theta\|_{H^1(\dOm)}+\|\wti\ga_N u_\Theta\|_{L^2(\dOm;d^{n-1}\omega)},
\leq C\|g\|_{\LdOm}
\end{array}
\end{eqnarray}
and
\begin{equation}
\|u_\Theta\|_{H^{3/2}(\Omega)} \leq C\|g\|_{\LdOm}, \lb{3.7}
\end{equation}
for some constant constant $C= C(\Theta,\Omega,z)>0$. Finally,    
\begin{equation}
\big[\ga_D (-\Delta_{\Theta,\Om}-\ol{z}I_\Om)^{-1}\big]^* \in
\cB\big(\LdOm, H^{3/2}(\Om)\big),   \lb{3.8}
\end{equation}
and the solution $u_\Theta$ is given by the formula  
\begin{equation}
u_\Theta = \big(\ga_D (-\Delta_{\Theta,\Om}-\ol{z}I_\Om)^{-1}\big)^*g. 
\lb{3.9}
\end{equation}
\end{theorem}
\begin{proof}
It is clear from Lemma \ref{Gam-L1} and Lemma \ref{Neu-tr} that 
the boundary value problem \eqref{3.6} has a meaningful formulation and
that any solution satisfies the first line in \eqref{3.6a}.
Uniqueness for \eqref{3.6} is an immediate consequence of the fact that
$z\in\bbC\backslash \sigma(-\Delta_{\Theta,\Omega})$. As for existence, 
as in the proof of Theorem \ref{t2.10}, we look for a candidate expressed as
\begin{equation}
u(x) = (\cS_z h)(x), \quad x\in\Om   \lb{3.11}
\end{equation}
for some $h\in L^2(\dOm; d^{n-1} \omega)$. This ensures that 
$u\in H^{3/2}(\Om)$ and $(-\Delta - z)u = 0$ in $\Om$. Above, the 
single layer potential $\cS_z$ has been defined in \eqref{sing-layer}.
The boundary condition $\big(\wti\ga_N + \wti \Theta \gamma_D\big) u = g$ on 
$\dOm$ is then equivalent to
\begin{equation}
\big(\wti\ga_N + \wti \Theta \ga_D\big)(\cS_z h) =g,
\end{equation}
respectively, to
\begin{equation}
\big({\textstyle{-\frac12}}I_\dOm+K^{\#}_{z}\big) h 
+ \wti \Theta \gamma_D {\mathcal S_{z}} h = g.     \lb{3.12}
\end{equation}
Here $K^{\#}_{z}$ has been defined in \eqref{Ksharp}. 

To obtain unique solvability of \eqref{3.12} for $h \in \LdOm$, 
given $g\in\LdOm$, at least when $z\in\bbC\backslash D$, where 
$D\subset\bbC$ is a discrete set, we proceed in a series of steps. 
The first step is to observe that the operator in question 
is Fredholm with index zero for every $z\in\bbC$. To see this, we decompose 
\begin{equation}
\big({\textstyle{-\frac12}}I_\dOm+K^{\#}_{z}\big) = 
\big({\textstyle{-\frac12}}I_\dOm+K^{\#}_{0}\big) 
+ \big( K^{\#}_{z} - K^{\#}_{0}\big),     
\lb{3.13}
\end{equation}
and recall that $\big(K^{\#}_{z} - K^{\#}_{0}\big)\in \cB_\infty(\LdOm)$ 
(cf. Lemma \ref{L-Kb}) and that 
${\textstyle{-\frac12}}I_\dOm+K^{\#}_{0}$ is a Fredholm operator 
in $\LdOm$ with Fredholm index equal to zero as proven by Verchota \cite{Ve84}.
In addition, we note that $\wti \Theta \ga_D \cS_{z} \in \cB_\infty(\LdOm)$ 
which follows from Hypothesis \ref{h3.1} and the fact that the 
following operators are bounded
\begin{eqnarray}\label{Ine-2}
\begin{array}{l}
\cS_{z}:\LdOm\to\{u\in H^{3/2}(\Om)\,|\,\Delta u\in L^2(\Omega;d^nx)\},
\\[4pt]
\gamma_D:\{u\in H^{3/2}(\Om)\,|\,\Delta u\in L^2(\Omega;d^nx)\}\to H^1(\dOm),
\end{array}
\end{eqnarray} 
(where the space $\{u\in H^{3/2}(\Om)\,|\,\Delta u\in L^2(\Omega;d^nx)\}$ 
is equipped with the natural graph norm 
$u\mapsto \|u\|_{H^{3/2}(\Om)}+\|\Delta u\|_{L^2(\Omega;d^nx)}$). 
See Lemma \ref{Gam-L1} and Theorem \ref{T-SH}. 
Thus, $\big({\textstyle{-\frac12}}I_\dOm+K^{\#}_{z}\big)+\wti\Theta\gamma_D 
{\mathcal S_{z}}$ is a Fredholm operator in $\LdOm$ with Fredholm index 
equal to zero, for every $z\in\bbC$. In particular, it is invertible 
if and only if it is injective. 

In the second step, we study the injectivity of 
$\big({\textstyle{-\frac12}}I_\dOm+K^{\#}_{z}\big)+\wti\Theta\gamma_D 
{\mathcal S_{z}}$ on $\LdOm$. For this purpose we now suppose that 
\begin{equation}
\big({\textstyle{-\frac12}}I_\dOm+K^{\#}_{z}\big) k  
+ \wti\Theta\gamma_D {\mathcal S_{z}} k = 0 \, \text{ for some } \, k\in \LdOm.
\end{equation}
Introducing $w=\cS_{z}k$ in $\Om$ one then infers that $w$ satisfies
\begin{equation}  
\begin{cases}
(-\Delta - z)w = 0 \text{ in }\,\Om,\quad w \in H^{3/2}(\Om), \\
\big(\wti\ga_N + \wti\Theta \gamma_D\big) w = 0 \text{ on }\,\dOm.
\end{cases}     \lb{3.17}
\end{equation}
Thus one obtains,
\begin{align}
0\leq \int_{\Om} d^n x \, |\nabla w|^2 & = \sum_{j=1}^n \int_{\Om} d^n x \, 
\ol{\partial_j w} \partial_j w = - \int_{\Om} d^n x \, \ol{\Delta w} w  
+\sum_{j=1}^n\int_{\dOm} d^{n-1}\omega\,\big(\ga_D\ol{\partial_j w}\big)\,
\nu_j\ga_D w   \no \\
& = \ol z \int_{\Om} d^n x \, |w|^2 
+  (\ga_D w, \wti\ga_N w)_{\LdOm}  
= \ol z \int_{\Om} d^n x \, |w|^2 
+ \langle \ga_D w, \wti\ga_N w \rangle_{1/2}   \no \\ 
&= \ol z \int_{\Om} d^n x \, |w|^2 - \big\langle \ga_D w, \wti \Theta \ga_D w 
\big\rangle_{1/2}.    \lb{3.18}
\end{align}
At this point we will first consider the case when $z\in\bbC\backslash\bbR$
(so that, in particular, $\Im (z)\not=0$). In this scenario, recalling 
\eqref{2.1} and taking the imaginary parts of the two most extreme sides 
of \eqref{3.18} imply that $\int_{\Om} d^n x \, |w|^2=0$ and, hence, 
$w=0$ in $\Om$. 

To continue, let $\wti\ga^{ext}_{N}$ and 
$\ga^{ext}_{D}$ denote, respectively, the Neumann and Dirichlet traces 
for the exterior domain $\bbR^n\backslash\overline{\Omega}$. Also, 
parallel to \eqref{sing-layer}, set 
\begin{equation}\label{si2-layer}
({\mathcal S_{ext,z}}g)(x)=\int_{\partial\Omega}d^{n-1}\omega(y)\,
E_n(z;x-y)g(y),\quad x\in\bbR^n\backslash\overline{\Omega}, \; z\in\bbC, 
\end{equation}
where $g$ is an arbitrary measurable function on $\partial\Omega$. Then, due
to the weak singularity in the integral kernel of $\cS_z$, 
\begin{equation}\label{jum-S1}
\ga_{D}\cS_{z}g=\ga^{ext}_{D}\cS_{ext,z}g\quad\mbox{ for every }\,\,
g\in L^2(\dOm;d^{n-1}\omega),
\end{equation}
whereas the counterpart of \eqref{jump} becomes
\begin{equation}\label{jumpX}
\wti\gamma^{ext}_N {\mathcal S_{ext,z}}g
=\big({\textstyle{\frac12}}I_\dOm+K^{\#}_{z}\big)g
\quad\mbox{ for every }\,\,g\in L^2(\dOm;d^{n-1}\omega). 
\end{equation}
Compared with \eqref{jump}, the change in sign is due to the fact 
that the outward unit normal to $\bbR^n\backslash\overline{\Omega}$ is $-\nu$.

Moving on, if we set $w^{ext}(x)=(\cS_{ext,z} k)(x)$ for 
$x\in\bbR^n\backslash\overline{\Omega}$, then from what we have proved so far
\begin{equation}\lb{3.19}
0= \ga_D w = \ga_D \cS_z k=\ga^{ext}_{D}w^{ext}
\quad\mbox{ in }\,\,L^2(\dOm;d^{n-1}\omega). 
\end{equation}\lb{3.18XX}
Fix now a sufficiently large $R>0$ such that 
$\overline{\Omega}\subset B(0;R)$ and write the analogue of \eqref{3.18}  
for the restriction of $w^{ext}$ to $B(0;R)\backslash\overline{\Omega}$: 
\begin{align} \label{New-Inx}
\int_{B(0;R)\backslash\overline{\Omega}}d^n x \, |\nabla w^{ext}|^2 
&= \ol z \int_{B(0;R)\backslash\overline{\Omega}} d^n x \, |w^{ext}|^2 
-\langle \ga^{ext}_D w^{ext}, \wti\ga^{ext}_N w^{ext} \rangle_{1/2} 
\nonumber\\
& \quad -\int_{|\xi|=R} d^{n-1}\omega(\xi)\, w^{ext}(\xi)\frac{\xi}{|\xi|}
\cdot\nabla w^{ext}(\xi). 
\end{align}
In view of \eqref{3.19}, the above identity reduces to 
\begin{equation}\label{New-InY}
\int_{B(0;R)\backslash\overline{\Omega}}d^n x \, |\nabla w^{ext}|^2 
=\ol z \int_{B(0;R)\backslash\overline{\Omega}} d^n x \, |w^{ext}|^2 
-\int_{|\xi|=R} d^{n-1}\omega(\xi)\, w^{ext}(\xi)\frac{\xi}{|\xi|}
\cdot\nabla w^{ext}(\xi). 
\end{equation}
Recall that we are assuming $z\in\bbC\backslash\bbR$. 
Given that, by \eqref{C.15} (and the comment following right after it), 
the integral kernel of $\cS_{ext,z} k$ has exponential decay at infinity,  
it follows that $w^{ext}$ decays exponentially at infinity. Thus, after 
passing to limit $R\to\infty$, we arrive at 
\begin{equation}\lb{3.18x}
\int_{\bbR^n\backslash\overline{\Omega}} d^n x \, |\nabla w^{ext}|^2 
= \ol z \int_{\bbR^n\backslash\overline{\Omega}} d^n x \, |w^{ext}|^2. 
\end{equation}
Consequently, taking the imaginary parts of both sides we arrive 
at the conclusion that $w^{ext}=0$ in $\bbR^n\backslash\overline{\Omega}$.  
With this in hand, we may then invoke \eqref{jump}, \eqref{jumpX}  
to deduce that 
\begin{equation}\label{Hf-2}
k=\wti\gamma^{ext}_{N}\cS_z k-\wti\gamma_{N}\cS_z k
=\wti\gamma^{ext}_{N}w^{ext}-\wti\gamma_{N}w=0,
\end{equation}
given that $w,w^{ext}$ vanish in $\Omega$ and 
$\bbR^n\backslash\overline{\Omega}$, respectively. 
Hence, $k=0$ in $\LdOm$. This proves that the operator 
$\big({\textstyle{-\frac12}}I_\dOm+K^{\#}_{z}\big) 
+ \wti \Theta \gamma_D {\mathcal S_{z}}$ is injective, hence, invertible 
on $\LdOm$ whenever $z\in\bbC\backslash\bbR$. 

In the third step, the goal is to extend this result to other values of 
the parameter $z$. To this end, fix some $z_0\in\bbC\backslash\bbR$, and for 
$z\in\bbC$, consider 
\begin{equation}
\cA_z = [\big({\textstyle{-\frac12}}I_\dOm+K^{\#}_{z_0}\big) 
+ \wti \Theta \gamma_D {\mathcal S_{z_0}}]^{-1}
[\big(K^{\#}_{z}-K^{\#}_{z_0}\big) + \wti \Theta \gamma_D ({\mathcal S_{z}}
- {\mathcal S_{z_0}})].
\end{equation}
Observe that the operator-valued mapping 
$z\mapsto \cA_z\in\cB\big(\LdOm\big)$ is analytic and, thanks to 
Lemma \ref{L-Kb}, $\cA_z \in \cB_\infty\big(\LdOm\big)$. The analytic 
Fredholm theorem then yields invertibility of $I + \cA_z$ except for 
$z$ in a discrete set $D\subset \bbC$. Thus,
\begin{equation}
\big({\textstyle{-\frac12}}I_\dOm+K^{\#}_{z}\big) 
+ \wti \Theta \gamma_D {\mathcal S_{z}}
=[\big({\textstyle{-\frac12}}I_\dOm+K^{\#}_{z_0}\big) 
+ \wti \Theta \gamma_D {\mathcal S_{z_0}}][I + \cA_z]
\end{equation}
is invertible for $z\in\bbC\backslash D$ where $D$ is a 
discrete set which, by the invertibility result proved in the previous 
paragraph, is contained in $\bbR$. 

The above argument proves unique solvability of \eqref{3.6} for 
$z\in\bbC\backslash D$, where $D$ is a discrete subset of $\bbR$. 
The representation \eqref{3.11} and the fact that 
$\ga_D \cS_z\colon \LdOm \to H^1(\dOm)$ boundedly then yields 
$\ga_D u_\Theta\in H^1(\dOm)$. Moreover, by \eqref{jump} and \eqref{3.11},
\begin{equation}
\wti\ga_N u_\Theta = \wti\ga_N \cS_z h 
= \big({\textstyle{-\frac12}}I_\dOm+K^{\#}_{z}\big) h \in \LdOm
\end{equation}
since by \eqref{Kb}, $K^{\#}_{z} \in \cB\big(\LdOm\big)$. 
This proves  \eqref{3.6a} when $z\in\bbC\backslash D$. 
For $z\in\bbC\backslash D$, the natural estimate \eqref{3.7} is a 
consequence of the integral representation formula \eqref{3.11} 
and \eqref{Sz-M6}. 

Next, fix a complex number
$z\in\bbC\backslash (D\cup\sigma(-\Delta_{\Theta,\Omega}))$ along with two 
functions, $v\in \LOm$ and $g\in L^2(\partial\Omega;d^{n-1}\omega)$. Also, 
let $u_{\Theta}$ solve \eqref{3.6}. One computes 
\begin{align}
\big(u_\Theta,v\big)_{\LOm} &=
\big(u_\Theta,(-\Delta-\ol{z})
(-\Delta_{\Theta,\Om} -\ol{z}I_\Om)^{-1}v\big)_{\LOm}
\no
\\ &=
\big((-\Delta-z)u_\Theta,(-\Delta_{\Theta,\Om} -\ol{z}I_\Om)^{-1}v\big)_{\LOm}
\no
\\ &\quad +
\big(\wti\ga_N u_\Theta, \ga_D
(-\Delta_{\Theta,\Om} -\ol{z}I_\Om)^{-1}v\big)_{\LdOm}  \no
\\ &\quad -
\big\langle \ga_D u_\Theta, \wti\ga_N
\big(-\Delta_{\Theta,\Om} -\ol{z}I_\Om\big)^{-1}v \big\rangle_{1/2} \no
\\  &=
\big(\wti\ga_N u_\Theta, \ga_D
(-\Delta_{\Theta,\Om} -\ol{z}I_\Om)^{-1}v\big)_{\LdOm}  \no
\\ &\quad +
\big\langle \ga_D u_\Theta, \wti \Theta \ga_D
(-\Delta_{\Theta,\Om} -\ol{z}I_\Om)^{-1}v \big\rangle_{1/2} \no
\\ 
&=
\big(\wti\ga_N u_\Theta, \ga_D
(-\Delta_{\Theta,\Om} -\ol{z}I_\Om)^{-1}v\big)_{\LdOm}  \no
\\ &\quad +
\ol{\big\langle \ga_D (-\Delta_{\Theta,\Om} -\ol{z}I_\Om)^{-1}v, 
\wti\Theta \ga_D u_\Theta \big\rangle_{1/2}}  \no 
\\ 
&=
\big(\wti\ga_N u_\Theta, \ga_D
(-\Delta_{\Theta,\Om} -\ol{z}I_\Om)^{-1}v\big)_{\LdOm}  \no
\\ &\quad +
\big\langle \wti \Theta \ga_D u_\Theta, 
\ga_D (-\Delta_{\Theta,\Om} -\ol{z}I_\Om)^{-1}v\big\rangle_{1/2}
\no
\\ 
&=
\big\langle \big(\wti\ga_N + \wti \Theta \ga_D\big) u_\Theta, \ga_D
(-\Delta_{\Theta,\Om} -\ol{z}I_\Om)^{-1}v\big\rangle_{1/2}  \no 
\\ 
&=
\big(\big(\wti\ga_N + \wti \Theta \ga_D\big) u_\Theta, \ga_D
(-\Delta_{\Theta,\Om} -\ol{z}I_\Om)^{-1}v\big)_{\LdOm}  \no 
\\ 
&=\big(g, \ga_D (-\Delta_{\Theta,\Om} -\ol{z}I_\Om)^{-1}v\big)_{\LdOm} \no
\\ &=
\big(\big(\ga_D (-\Delta_{\Theta,\Om} -\ol{z}I_\Om)^{-1}\big)^*g,v\big)_{\LOm}.
\end{align}
Since $v\in \LOm$ was arbitrary, this yields 
\begin{equation}
u_\Theta = \big(\ga_D (-\Delta_{\Theta,\Om}-\ol{z}I_\Om)^{-1}\big)^*g
\,\,\mbox{ in }\,\,\LOm,\quad\mbox{for }\,\,
z\in\bbC\backslash (D\cup\sigma(-\Delta_{\Theta,\Omega})), 
\lb{3.9W}
\end{equation}
which proves \eqref{3.9} for
$z\in\bbC\backslash (D\cup\sigma(-\Delta_{\Theta,\Omega}))$. 
From this and \eqref{3.7}, the membership \eqref{3.8} also follows when 
$z\in\bbC\backslash (D\cup\sigma(-\Delta_{\Theta,\Omega}))$. 

The extension to the more general case when 
$z\in\bbC\backslash\sigma(-\Delta_{\Theta,\Omega})$
is then done by resorting to analytic continuation with respect to $z$.
More specifically, fix $z_0\in\bbC\backslash\sigma(\Delta_{\Theta,\Om})$. 
Then there exists $r>0$ such that 
\begin{eqnarray}\label{An-C1}
\bigl(\ol{B(z_0,r)}\backslash \{z_0\}\bigr)\cap D=\emptyset,\quad
\ol{B(z_0,r)}\cap\sigma(-\Delta_{\Theta,\Omega})=\emptyset,
\end{eqnarray}
since $D$ is discrete and $\sigma(-\Delta_{\Theta,\Omega})$ is closed. 
We may then write 
\begin{eqnarray}\label{An-C2}
[\gamma_D(-\Delta_{\Theta,\Om} -\ol{z_0}I_\Om)^{-1}]^*
= \frac{1}{2\pi i}\oint_{C(z_0;r)} dz \, (z-z_0)^{-1}
[\gamma_D (-\Delta_{\Theta,\Om} -\ol{z}I_\Om)^{-1}]^*
\end{eqnarray}
as operators in $\cB(H^{-1/2}(\partial\Omega),L^2(\Omega;d^nx))$, 
where $C(z_0;r) \subset \bbC$ denotes the counterclockwise oriented 
circle with center $z_0$ and radius $r$. (This follows from dualizing 
the fact $\gamma_D(-\Delta_{\Theta,\Om} - z_0 I_\Om)^{-1} \in 
\cB(L^2(\Omega;d^nx), H^{1/2}(\partial\Omega))$, which in turn follows from the mapping properties $(-\Delta_{\Theta,\Om} - z_0 I_\Om)^{-1} \in 
\cB(L^2(\Omega;d^nx), H^{1}(\partial\Omega))$ and 
$\gamma_D \in \cB\big(H^{1}(\partial\Omega), H^{1/2}(\partial\Omega)\big)$.) 
However, granted \eqref{An-C1}, what we have shown so far yields that 
$[\gamma_D (-\Delta_{\Theta,\Om} -\ol{z}I_\Om)^{-1}]^*\in 
\cB(L^2(\partial\Omega;d^{n-1}\omega),H^{3/2}(\Omega))$ 
whenever $|z-z_0|=r$, with a bound 
\begin{eqnarray}\label{An-C3}
\|[\gamma_D (-\Delta_{\Theta,\Om} -\ol{z}I_\Om)^{-1}]^*\| 
_{\cB(L^2(\partial\Omega;d^{n-1}\omega),H^{3/2}(\Omega))}
\leq C=C(\Omega,\Theta,z_0,r)
\end{eqnarray}
independent of the complex parameter $z\in\partial B(z_0,r)$. 
This estimate and Cauchy's representation formula \eqref{An-C2} then imply 
that $[\gamma_D (-\Delta_{\Theta,\Om} -\ol{z_0}I_\Om)^{-1}]^*\in 
\cB(L^2(\partial\Omega;d^{n-1}\omega),H^{3/2}(\Omega))$. 
This further entails that 
$u= [\gamma_D (-\Delta_{\Theta,\Om} -\ol{z_0}I_\Om)^{-1}]^*g$ solves
\eqref{3.6}, written with $z_0$ in place of $z$, and satisfies \eqref{3.7}.
Finally, the memberships in \eqref{3.6a} (along with naturally accompanying
estimates) follow from Lemma \ref{Gam-L1} and Lemma \ref{Neu-tr}. 
This shows that \eqref{3.8}, along with the well-posedness of \eqref{3.6}
and all the desired properties of the solution, hold whenever 
$z\in\bbC\backslash\sigma(\Delta_{\Theta,\Om})$.
\end{proof}
The special case $\Theta=0$ of Theorem \ref{t3.2}, corresponding to 
the Neumann Laplacian, deserves to be mentioned separately. 
\begin{corollary} \lb{t3.2H} 
Assume Hypothesis \ref{h2.1} and suppose that 
$z\in\bbC\backslash\si(-\Delta_{N,\Om})$. Then for every $g\in\LdOm$, 
the following Neumann boundary value problem,
\begin{equation} \lb{3.6H}
\begin{cases}
(-\Delta - z)u = 0 \text{ in }\,\Om,\quad u \in H^{3/2}(\Om), \\
\wti\ga_N u = g \text{ on } \,\dOm,
\end{cases}
\end{equation}
has a unique solution  $u=u_N$. This solution $u_N$ satisfies
\begin{eqnarray}\lb{3.6aH}
\ga_D u_N\in H^1(\dOm)\quad\mbox{and}\quad 
\|\ga_D u_N\|_{H^1(\dOm)}+\|\wti\ga_N u_N\|_{L^2(\dOm;d^{n-1}\omega)} 
\leq C\|g\|_{\LdOm}
\end{eqnarray}
as well as 
\begin{equation}\lb{3.7H}
\|u_N\|_{H^{3/2}(\Omega)} \leq C\|g\|_{\LdOm}, 
\end{equation}
for some constant constant $C= C(\Theta,\Omega,z)>0$. Finally,    
\begin{equation}\lb{3.8H}
\big[\ga_D (-\Delta_{N,\Om}-\ol{z}I_\Om)^{-1}\big]^* \in
\cB\big(\LdOm, H^{3/2}(\Om)\big),   
\end{equation}
and the solution $u_N$ is given by the formula  
\begin{equation}\lb{3.9H}
u_\Theta = \big(\ga_D (-\Delta_{N,\Om}-\ol{z}I_\Om)^{-1}\big)^*g. 
\end{equation}
\end{corollary}

Next, we turn to the Dirichlet case originally treated in 
\cite[Theorem 3.1]{GMZ07} but under stronger regularity conditions 
on $\Omega$. In order to facilitate the subsequent considerations, 
we isolate a useful technical result in the lemma below.

\begin{lemma}\label{L-BbbZ}
Assume Hypothesis \ref{h2.1} and suppose that 
$z\in\bbC\backslash\si(-\Delta_{D,\Om})$. Then 
\begin{equation}\lb{Kr-a.1}
(-\Delta_{D,\Om} - zI_\Om)^{-1}:\LOm\to\bigl\{u\in H^{3/2}(\Omega)\,|\,
\Delta u\in L^2(\Omega;d^nx)\bigl\}
\end{equation}
is a well-defined bounded operator, where the space 
$\{u\in H^{3/2}(\Omega)\,|\,\Delta u\in L^2(\Omega;d^nx)\}$ is 
equipped with the natural graph norm 
$u\mapsto\|u\|_{H^{3/2}(\Omega)}+\|\Delta u\|_{L^2(\Omega;d^nx)}$. 
\end{lemma}
\begin{proof} 
Consider $z\in\bbC\backslash\si(-\Delta_{D,\Om})$, 
$f\in\LOm$ and set $w=(-\Delta_{D,\Om} - zI_\Om)^{-1}f$. 
It follows that $u$ is the unique solution of the problem
\begin{equation}\lb{Bv-1Z}
(-\Delta -z)w=f\, \text{ in } \, \Omega,\quad w\in H^1_0(\Omega).
\end{equation}
The strategy is to devise an alternative representation for $w$ 
from which it is clear that $w$ has the claimed regularity in $\Omega$. 
To this end, let $\widetilde{f}$ denote the extension of $f$ by zero to 
${\mathbb{R}}^n$ and denote by $E$ the operator of convolution by $E_n(z;\dott)$.
Since the latter is smoothing of order $2$, it follows that 
$v=\big(E\widetilde{f}\,\big)\big|_{\Omega}\in H^2(\Omega)$ and $(-\Delta -z)v=f$ in $\Omega$.
In particular, $g=-\gamma_D v\in H^1(\partial\Omega)$. 

We now claim that the problem 
\begin{equation}\lb{Bv-2Z}
(-\Delta -z)u=0\, \text{ in } \, \Omega,\quad u\in H^{3/2}(\Omega),
\quad \gamma_D u=g\text{ on } \,\partial\Omega,
\end{equation}
has a solution (satisfying natural estimates). To see this, we look for 
a solution in the form \eqref{3.11} for some $h\in L^2(\dOm; d^{n-1} \omega)$. 
This guarantees that $u\in H^{3/2}(\Om)$ by Theorem \ref{T-SH}, 
and $(-\Delta - z)u = 0$ in $\Om$. Ensuring that the boundary condition 
holds comes down to solving $\gamma_D{\mathcal{S}}_zh=g$. In this regard, 
we recall that
\begin{equation}\lb{Bv-3W}
\gamma_D{\mathcal{S}}_0 \colon \LdOm\to H^1(\partial\Omega) \, 
\text{ is invertible}
\end{equation}
(cf.\ \cite{Ve84}). With this in hand, by relying on Theorem \ref{T-SH} and 
arguing as in the proof of Theorem \ref{t3.2}, one can show that there 
exists a discrete set $D\subset\bbC$ such that 
\begin{equation}\lb{Bv-3T}
\gamma_D{\mathcal{S}}_z \colon \LdOm\to H^1(\partial\Omega) \, 
\text{ is invertible for }\, z\in\bbC\backslash D. 
\end{equation}
Thus, a solution of \eqref{Bv-2Z} is given by 
\begin{equation}\lb{Bv-4W}
u={\mathcal{S}}_z\bigl((\gamma_D{\mathcal{S}}_z)^{-1}(\gamma_D v)\bigr)
\, \text{ if }\, z\in\bbC\backslash D.
\end{equation}
Moreover, by Theorem \ref{T-SH}, this satisfies
\begin{equation}\lb{Bv-4Wb}
\|u\|_{H^{3/2}(\Omega)}\leq C(\Omega,z)\|g\|_{H^1(\partial\Omega)}
\leq C(\Omega,z)\|f\|_{\LOm},\quad z\in\bbC\backslash D.
\end{equation}
Consequently, if $z\in\bbC\backslash(D\cup\si(-\Delta_{D,\Omega}))$, then 
$u+v$ solves \eqref{Bv-1Z}. Hence, by uniqueness, $w=u+v$ in this case. 
This shows that $w=(-\Delta_{D,\Omega}-zI_{\Omega})^{-1}f$ belongs to 
$H^{3/2}(\Omega)$ and satisfies $\Delta w\in\LOm$ with 
\begin{equation}\lb{Bv-4Wc}
\|w\|_{H^{3/2}(\Omega)}+\|\Delta w\|_{\LOm}\leq C(\Omega,z)\|f\|_{\LOm},
\quad z\in\bbC\backslash(D\cup\si(-\Delta_{D,\Omega})).
\end{equation}
In summary, the above argument shows that 
the operator \eqref{Kr-a.1} is well-defined and bounded whenever 
$z\in\bbC\backslash(D\cup\si(-\Delta_{D,\Omega}))$. The extension to 
$z\in\bbC\backslash\si(-\Delta_{D,\Omega})$ is then achieved via analytic 
continuation (as in the last part of the proof of Theorem \ref{t3.2}).
\end{proof}

Having established Lemma \ref{L-BbbZ}, we can now readily prove the following
result. 

\begin{lemma}\label{L-Bbb}
Assume Hypothesis \ref{h2.1} and suppose that 
$z\in\bbC\backslash\si(-\Delta_{D,\Om})$. Then 
\begin{equation}\lb{3.23}
\wti\ga_N (-\Delta_{D,\Om} - zI_\Om)^{-1} \in \cB\big(\LOm,\LdOm\big),
\end{equation}
and 
\begin{equation}\lb{3.24}
\big[\wti\ga_N (-\Delta_{D,\Om} - zI_\Om)^{-1}\big]^* 
\in \cB\big(\LdOm,\LOm\big).  
\end{equation}
\end{lemma}
\begin{proof} Obviously, it suffices to only prove \eqref{3.23}. 
However, this is an immediate consequence of Lemma \ref{L-BbbZ}
and Lemma \ref{Neu-tr}.
\end{proof}

We note that Lemma \ref{L-Bbb} corrects an inaccuracy in the proof of 
\cite[Theorem 3.1]{GMZ07} in the following sense: The proof of (3.20) 
and (3.21) in \cite{GMZ07} relies on \cite[Lemma 2.4]{GMZ07}, which in 
turn requires the stronger assumptions \cite[Hypothesis 2.1]{GMZ07} on 
$\Om$ than merely the Lipschitz assumption on $\Om$. However, the current 
Lemmas \ref{l2.7} and \ref{L-Bbb} (and the subsequent Theorem \ref{t3.3}) 
show that (3.20) and (3.21) in \cite{GMZ07}, as well as the results  
stated in \cite[Theorem 3.1]{GMZ07}, are actually all correct.

After this preamble, we are ready to state the result about 
the well-posedness of the Dirichlet problem, alluded to above.

\begin{theorem} \lb{t3.3}
Assume Hypothesis \ref{h2.1} and suppose that 
$z\in\bbC\backslash\si(-\Delta_{D,\Om})$. Then for every $f \in H^1(\dOm)$, 
the following Dirichlet boundary value problem,
\begin{equation} \lb{3.31}
\begin{cases}
(-\Delta - z)u = 0 \text{ in }\, \Om, \quad u \in H^{3/2}(\Om), \\
\ga_D u = f \text{ on }\, \dOm,
\end{cases}
\end{equation}
has a unique solution $u=u_D$. This solution $u_D$ satisfies 
\begin{equation}\lb{Hh.3} 
\wti\ga_N u_D \in \LdOm\quad \mbox{and}\quad 
\|\wti\ga_N u_D\|_{\LdOm}\leq C_D\|f\|_{H^1(\dOm)},
\end{equation}
for some constant $C_D=C_D(\Omega,z)>0$. Moreover, 
\begin{equation}
\|u_D\|_{H^{3/2}(\Omega)} \leq C_D \|f\|_{H^1(\partial\Omega)}.  \lb{3.33}
\end{equation} 
Finally,      
\begin{equation}
\big[\wti\ga_N (-\Delta_{D,\Om}-{\ol z}I_\Om)^{-1}\big]^* \in 
\cB\big(H^1(\dOm), H^{3/2}(\Om)\big),   \lb{3.34} 
\end{equation}
and the solution $u_D$ is given by the formula
\begin{equation}
u_D = -\big[\wti\ga_N (-\Delta_{D,\Om}-\ol{z}I_\Om)^{-1}\big]^*f. 
\lb{3.35}
\end{equation}
\end{theorem}
\begin{proof}
Uniqueness for \eqref{3.31} is a direct consequence of the fact that 
$z\in\bbC\backslash \sigma(-\Delta_{D,\Omega})$. Existence, at least when 
$z\in\bbC\backslash  D$ for a discrete set $D\subset\bbC$, is implicit in the
proof of Lemma \ref{L-BbbZ} (cf. \eqref{Bv-2Z}). Note that a solution thus 
constructed obeys \eqref{3.33} and satisfies 
\eqref{Hh.3} (cf. Lemmas \ref{Gam-L1} and \ref{Neu-tr}). 

Next, we turn to the proof of \eqref{3.35}. Assume that 
$z\in\bbC\backslash (D\cup\sigma(-\Delta_{D,\Omega}))$ and denote by $u_D$ the 
unique solution of \eqref{3.31}. Also, recall \eqref{3.23}--\eqref{3.24}. 
Based on these and Green's formula, one computes
\begin{align}
(u_D,v)_{\LOm} &=
(u_D,(-\Delta-\ol{z})(-\Delta_{D,\Om}-\ol{z}I_\Om)^{-1}v)_{\LOm}
\no
\\ &=
((-\Delta-z)u_D, (-\Delta_{D,\Om}-\ol{z}I_\Om)^{-1}v)_{\LOm}
\no
\\ &\quad +
(\wti\ga_N u_D, \ga_D (-\Delta_{D,\Om}-\ol{z}I_\Om)^{-1}v)_{\LdOm}  \no
\\ &\quad -
\langle \ga_D u_D, \wti\ga_N
(-\Delta_{D,\Om}-\ol{z}I_\Om)^{-1}v\rangle_{1/2} \no
\\ &=
-\langle f, \wti\ga_N (-\Delta_{D,\Om}-\ol{z}I_\Om)^{-1}v\rangle_{1/2} 
\no
\\ &=
-\big(\big(\wti\ga_N (-\Delta_{D,\Om}-\ol{z}I_\Om)^{-1}\big)^*f,v\big)_{\LOm} 
\end{align}
for any $v\in\LOm$. This proves \eqref{3.35} with the operators involved 
understood in the sense of \eqref{3.24}. Given \eqref{3.33}, one 
obtains \eqref{3.34} granted that 
$z\in\bbC\backslash (D\cup\sigma(-\Delta_{D,\Omega}))$. 

Finally, the extension of the above results to the more general case in which 
$z\in\bbC\backslash \sigma(-\Delta_{D,\Omega})$ is done using analytic 
continuation, as in the last part of the proof of Theorem \ref{t3.2}. 
\end{proof}

Assuming Hypothesis \ref{h3.1}, we now introduce the
Dirichlet-to-Robin map $M_{D,\Theta,\Om}^{(0)}(z)$ 
associated with $(-\Delta-z)$ on $\Om$, as follows,
\begin{align}
M_{D,\Theta,\Om}^{(0)}(z) \colon
\begin{cases}
H^1(\dOm) \to \LdOm,  \\
\hspace*{10mm} f \mapsto - \big(\wti\ga_N +\wti\Theta\gamma_D\big)u_D,
\end{cases}  \quad z\in\bbC\backslash\si(-\Delta_{D,\Om}), \lb{3.44}
\end{align}
where $u_D$ is the unique solution of
\begin{align}
(-\Delta-z)u = 0 \,\text{ in }\Om, \quad u \in
H^{3/2}(\Om), \quad \ga_D u = f \,\text{ on }\dOm.   \lb{3.45}
\end{align}

Continuing to assume Hypothesis \ref{h3.1}, we next introduce the 
Robin-to-Dirichlet map $M_{\Theta,D,\Om}^{(0)}(z)$ associated with 
$(-\Delta-z)$ on $\Om$, as follows,
\begin{align}
M_{\Theta,D,\Om}^{(0)}(z) \colon \begin{cases} \LdOm \to H^1(\dOm),
\\
\hspace*{20.8mm} g \mapsto \ga_D u_{\Theta}, \end{cases}  \quad
z\in\bbC\backslash\si(-\Delta_{\Theta,\Om}), \lb{3.48}
\end{align}
where $u_{\Theta}$ is the unique solution of
\begin{align}
(-\Delta-z)u = 0 \,\text{ in }\Om, \quad u \in
H^{3/2}(\Om), \quad \big(\wti\ga_N + \wti \Theta\gamma_D\big)u = g 
\,\text{ on }\dOm.  \lb{3.49}
\end{align}
We note that Robin-to-Dirichlet maps have also been studied in \cite{Au04}.

We conclude with the following theorem, one of the main results of this paper:
 
\begin{theorem} \lb{t3.5} 
Assume Hypothesis \ref{h3.1}. Then 
\begin{equation}
M_{D,\Theta,\Om}^{(0)}(z) \in \cB\big(H^1(\partial\Om), \LdOm \big), \quad
z\in\bbC\backslash\si(-\Delta_{D,\Om}),   \lb{3.46}
\end{equation}
and 
\begin{equation}
M_{D,\Theta,\Om}^{(0)}(z) = \big(\wti\gamma_N+\wti\Theta\gamma_D\big)\big[
\big(\wti\gamma_N+\wti\Theta\gamma_D\big)
(-\Delta_{D,\Om} - \ol{z}I_\Om)^{-1}\big]^*, 
\quad z\in\bbC\backslash\si(-\Delta_{D,\Om}). \lb{3.47}
\end{equation}
Moreover, 
\begin{equation}
M_{\Theta,D,\Om}^{(0)}(z) \in \cB\big(\LdOm, H^1(\partial\Om) \big), \quad 
z\in\bbC\backslash\si(-\Delta_{\Theta,\Om}),    \lb{3.50}
\end{equation}
and, in fact,  
\begin{equation}
M_{\Theta,D,\Om}^{(0)}(z) \in \cB_\infty\big(\LdOm\big), \quad 
z\in\bbC\backslash\si(-\Delta_{\Theta,\Om}).  \lb{3.51}
\end{equation}
In addition, 
\begin{equation}
M_{\Theta,D,\Om}^{(0)}(z) = \gamma_D\big[\gamma_D
(-\Delta_{\Theta,\Om} - \ol{z}I_\Om)^{-1}\big]^*, \quad
z\in\bbC\backslash\si(-\Delta_{\Theta,\Om}). \lb{3.52}
\end{equation} 

Finally, let 
$z\in\bbC\backslash(\si(-\Delta_{D,\Om})\cup\si(-\Delta_{\Theta,\Om}))$. Then
\begin{equation}
M_{\Theta,D,\Om}^{(0)}(z) = - M_{D,\Theta,\Om}^{(0)}(z)^{-1}.   \lb{3.53}  
\end{equation}
\end{theorem}
\begin{proof}
The membership in \eqref{3.46} is a consequence of Theorem \ref{t3.3}. 
In this context we note that by the first line in \eqref{2.39}, 
$\ga_D (-\Delta_{D,\Om}-\ol{z}I_\Om)^{-1} =0$, and hence 
\begin{equation}
u_D = - \big[\wti\ga_N (-\Delta_{D,\Om}-\ol{z}I_\Om)^{-1}\big]^*f 
= - \big[\big(\wti\ga_N + \wti \Theta \ga_D \big) 
(-\Delta_{D,\Om}-\ol{z}I_\Om)^{-1}\big]^*f  
\lb{3.53a}
\end{equation}
by \eqref{3.35}. Moreover, applying 
$-\big(\wti\gamma_N+\wti\Theta\gamma_D\big)$ 
to $u_D$ in \eqref{3.35} implies formula \eqref{3.47}. 
Likewise, \eqref{3.50} follows from Theorem \ref{t3.2}. In addition, since 
$H^1(\partial\Om)$ embeds compactly into $L^2(\partial\Om; d^{n-1} \omega)$ 
(cf.\ \eqref{EQ1} and \cite[Proposition 2.4]{MM07}), 
$M_{\Theta,D,\Om}^{(0)}(z)$, $z\in\bbC\backslash\si(-\Delta_{\Theta,\Om})$, 
are compact operators in $L^2(\partial\Om; d^{n-1} \omega)$, justifying 
\eqref{3.51}. Applying $\gamma_D$ to $u_{\Theta}$ in \eqref{3.9} implies 
formula \eqref{3.52}. 

There remains to justify \eqref{3.53}. To this end, let $g\in \LdOm$
be arbitrary. Then
\begin{equation}
M_{D,\Theta,\Om}^{(0)}(z) M_{\Theta,D,\Om}^{(0)}(z) g 
= M_{D,\Theta,\Om}^{(0)}(z) \gamma_D u_\Theta
= - \big(\wti \gamma_N +\wti\Theta\gamma_D\big)u_D, 
\quad f=\gamma_D u_\Theta\in H^1(\dOm).  \lb{3.54}
\end{equation}
Here $u_\Theta$ is the unique solution of $(-\Delta-z)u=0$ with 
$u\in H^{3/2}(\Om)$ and $\big(\wti \gamma_N +\wti\Theta\gamma_D\big)u =g$, 
and $u_D$ is the unique solution of $(-\Delta-z)u=0$ with 
$u\in H^{3/2}(\Om)$ and $\gamma_D u = f \in H^1(\dOm)$. Since 
$(u_D - u_\Theta)\in H^{3/2}(\Om)$ and $\gamma_D u_D = f = \gamma_D u_\Theta$, 
one concludes
\begin{equation}
\gamma_D (u_D - u_\Theta)=0 \, \text{ and } \, (-\Delta - z)(u_D - u_\Theta)=0.
\lb{3.55}
\end{equation}
Uniqueness of the Dirichlet problem proved in Theorem \ref{t3.3} then 
yields $u_D=u_\Theta$ which further entails that 
$-\big(\wti\gamma_N +\wti\Theta\gamma_D\big)u_D
=-\big(\wti\gamma_N +\wti\Theta\gamma_D\big)u_\Theta =-g$. Thus, 
\begin{equation}
M_{D,\Theta,\Om}^{(0)}(z) M_{\Theta,D,\Om}^{(0)}(z) g 
= - \big(\wti \gamma_N +\wti\Theta\gamma_D\big)u_D = -g, \lb{3.56}
\end{equation}
implying $M_{D,\Theta,\Om}^{(0)}(z) M_{\Theta,D,\Om}^{(0)}(z) = -I_{\dOm}$. 
Conversely, let $f\in H^1(\dOm)$. Then
\begin{equation}\lb{3.57}
M_{\Theta,D,\Om}^{(0)}(z) M_{D,\Theta,\Om}^{(0)}(z) f 
= M_{\Theta,D,\Om}^{(0)}(z)
\big(-\big(\wti\gamma_N +\wti\Theta\gamma_D\big)u_D\big) 
= \gamma_D u_\Theta, 
\end{equation}
and we set 
\begin{equation}\lb{3.57bis}
g = - \big(\wti\gamma_N +\wti\Theta\gamma_D\big)u_D\in \LdOm.    
\end{equation}
Here $u_D,u_\Theta\in H^{3/2}(\Omega)$ are such that
$(-\Delta - z)u_\Theta=(-\Delta - z)u_D=0$ in $\Omega$ and 
$\gamma_D u_D =f$, $\big(\wti\gamma_N +\wti\Theta\gamma_D\big)u_\Theta=g$. 
Thus $\big(\wti\gamma_N+\wti\Theta\gamma_D\big)(u_\Theta + u_D)=0$, 
$(-\Delta - z)(u_\Theta + u_D)=0$ and $(u_D+u_\Theta) \in H^{3/2}(\Omega)$.
Uniqueness of the generalized Robin problem proved in Theorem \ref{t3.2} 
then yields $u_\Theta=-u_D$ and 
hence $\gamma_D u_\Theta =-\gamma_D u_D =-f$. Thus,
\begin{equation}
M_{\Theta,D,\Om}^{(0)}(z) M_{D,\Theta,\Om}^{(0)}(z) f 
= \gamma_D u_\Theta=-f,   \lb{3.58}
\end{equation}
implying $M_{\Theta,D,\Om}^{(0)}(z) M_{D,\Theta,\Om}^{(0)}(z) 
\subseteq -I_{\dOm}$. The desired conclusion now follows. 
\end{proof}

\begin{remark} \lb{r3.6X}
In the above considerations, the special case $\Theta =0$ represents 
the frequently studied Neumann-to-Dirichlet and Dirichlet-to-Neumann maps
$M_{N,D,\Om}^{(0)}(z)$ and $M_{D,N,\Om}^{(0)}(z)$, respectively. That is, 
$M_{N,D,\Om}^{(0)}(z)=M_{0,D,\Om}^{(0)}(z)$ and 
$M_{D,N,\Om}^{(0)}(z)=M_{D,0,\Om}^{(0)}(z)$. Thus, as a corollary 
of Theorem \ref{t3.5} we have 
\begin{equation}
M_{N,D,\Om}^{(0)}(z) = - M_{D,N,\Om}^{(0)}(z)^{-1},   \lb{3.53X}  
\end{equation}
whenever Hypothesis \ref{h2.1} holds and 
$z\in\bbC\backslash(\si(-\Delta_{D,\Om})\cup\si(-\Delta_{N,\Om}))$. 
\end{remark}

\begin{remark}\lb{r3.6}
We emphasize again that all results in this section immediately extend to  
Schr\"odinger operators $H_{\Theta,\Om} = -\Delta_{\Theta,\Om} + V$, 
$\dom\big(H_{\Theta,\Om}\big) = \dom\big(-\Delta_{\Theta,\Om}\big)$ in 
$\LOm$ for (not necessarily real-valued) potentials $V$ satisfying 
$V \in L^\infty(\Om; d^n x)$, or more generally, for potentials $V$ which 
are Kato--Rellich bounded with respect to $-\Delta_{\Theta,\Om}$ with bound 
less than one. Denoting the corresponding $M$-operators by $M_{D,N,\Om}(z)$ 
and $M_{\Theta,D,\Om}(z)$, respectively, we note, in particular, that 
\eqref{3.44}--\eqref{3.53} extend replacing $-\Delta$ by $-\Delta + V$ 
and restricting $z\in\bbC$ appropriately.
\end{remark}

Our discussion of Weyl--Titchmarsh operators follows the earlier papers  
\cite{GLMZ05} and \cite{GMZ07}. For related literature on Weyl--Titchmarsh 
operators, relevant in the context
of boundary value spaces (boundary triples, etc.), we refer, for
instance, to \cite{ABMN05}, \cite{AP04}, \cite{BL07}, \cite{BMN06},
\cite{BMN00}--\cite{BMNW08}, \cite{DHMS00}-- \cite{DM95}, \cite{GKMT01}, \cite{GM09}, \cite[Ch.\ 3]{GG91}, \cite[Ch.\ 13]{Gr09}, \cite{MM06}, \cite{Ma04}, 
\cite{MPP07}, \cite{Pa87}, \cite{Pa02}, \cite{Po04}, \cite{Po08}, 
\cite{Ry08}, \cite{Ry08a}, \cite{TS77}.

\section{Some Variants of Krein's Resolvent Formula} \label{s4}

In this section we present our principal new results, variants of 
Krein's formula for the difference of resolvents of generalized Robin 
Laplacians and Dirichlet Laplacians on bounded Lipschitz domains.

We start by weakening Hypothesis \ref{h3.1} by using assumption 
\eqref{3.5bisa} below: 

\begin{hypothesis} \lb{h3.1bis}
In addition to Hypothesis \ref{h2.2} suppose that 
\begin{equation}\lb{3.5bisa}
\wti\Theta\in\cB_\infty\big(H^{1/2}(\dOm),H^{-1/2}(\partial\Omega)\big). 
\end{equation}
\end{hypothesis}

We note that condition \eqref{3.5bisa} is satisfied 
if there exists some $\varepsilon>0$ such that 
\begin{equation}\lb{3.5D}
\wti\Theta\in\cB\big(H^{1/2-\varepsilon}(\dOm),H^{-1/2}(\partial\Omega)\big). 
\end{equation}

Before proceeding with the main topic of this section, we will comment to 
the effect that Hypothesis \ref{h3.1} is indeed stronger than 
Hypothesis~\ref{h3.1bis}, as the latter follows from the former 
via duality and interpolation, implying 
\begin{equation}\lb{4.3}
\wti\Theta  \in\cB_{\infty}\big(H^{s}(\dOm),H^{s-1}(\partial\Omega)\big),
\quad 0\leq s\leq 1.
\end{equation}

To see this, one first employs the fact that 
\begin{equation}\lb{Gh-a2}
(H^{s_0}(\partial\Omega),H^{s_1}(\partial\Omega))_{\theta,2}
=H^{s}(\partial\Omega)
\end{equation}
for $s=(1-\theta)s_0+\theta s_1$, $0<\theta<1$, $0\leq s_0,s_1\leq 1$, 
and $s_0\not=s_1$, where $(\cdot,\cdot)_{\theta,q}$ denotes the 
{\it real interpolation} method. 

Second, one uses the fact that if $T:X_j\to Y_j$, $j=0,1$, is a linear and 
bounded operator between two pairs of compatible Banach spaces,  
which is compact for $j=0$, then 
$T\in\cB_{\infty}( (X_0,X_1)_{\theta,p} , (Y_0,Y_1)_{\theta,p} )$ 
for every $\theta\in(0,1)$. This is a result due to Cwikel \cite{Cw92}: 

\begin{theorem} [\cite{Cw92}] \label{T-Cwikel}  
Let $X_j$, $Y_j$, $j=0,1$, be two compatible Banach space couples and suppose 
that the linear operator $T:X_j\to Y_j$ is bounded for $j=0$ and compact for 
$j=1$. Then $T:(X_0,X_1)_{\theta,p}\to (Y_0,Y_1)_{\theta,p}$ is compact for 
all $\theta\in(0,1)$ and $p\in[1,\infty]$. 
\end{theorem}
\noindent (Interestingly, the corresponding result for the complex 
method of interpolation remains open.)

In our next two results below (Theorems \ref{t3.2bis}--\ref{t3.XV}) we 
discuss the solvability of the Dirichlet and Robin boundary value problems
with solution in the energy space $H^1(\Omega)$.  

\begin{theorem} \lb{t3.2bis} 
Assume Hypothesis \ref{h3.1bis} and suppose that 
$z\in\bbC\backslash\si(-\Delta_{\Theta,\Om})$. Then for every 
$g\in H^{-1/2}(\partial\Omega)$, 
the following generalized Robin boundary value problem,
\begin{equation} \lb{3.6bis}
\begin{cases}
(-\Delta - z)u = 0 \text{ in }\,\Om,\quad u \in H^{1}(\Om), \\
\big(\wti\ga_N + \wti \Theta \gamma_D\big) u = g \text{ on } \,\dOm,
\end{cases}
\end{equation}
has a unique solution $u=u_\Theta$. Moreover, there exists a constant 
$C= C(\Theta,\Omega,z)>0$ such that
\begin{equation}\lb{3.7bis}
\|u_\Theta\|_{H^{1}(\Omega)} \leq C\|g\|_{H^{-1/2}(\partial\Omega)}.  
\end{equation}
In particular, 
\begin{equation}\lb{3.8bis}
\big[\ga_D (-\Delta_{\Theta,\Om}-\ol{z}I_\Om)^{-1}\big]^* \in
\cB\big(H^{-1/2}(\partial\Omega), H^{1}(\Om)\big),  
\end{equation}
and the solution $u_\Theta$ of \eqref{3.6bis} is once again given by formula  
\eqref{3.9}. 
\end{theorem}
\begin{proof}
The argument follows a pattern similar to that in the proof of  
Theorem \ref{t2.10}. In a first stage, we look for a solution for 
\eqref{3.6bis} in the form  
\begin{equation}
u(x) = (\cS_z h)(x), \quad x\in\Om,   \lb{4.5}
\end{equation}
for some $h\in H^{-1/2}(\partial\Omega)$. Here the single layer potential 
$\cS_z$ has been defined in \eqref{sing-layer}, and the 
fundamental Helmholtz solution $E_n$ is given by \eqref{2.52} (cf.\ also 
\eqref{C.1}). Any such choice of $h$ guarantees that $u$ belongs to 
$H^{1}(\Om)$ and satisfies $(-\Delta - z)u = 0$ in $\Om$. See \eqref{Sz-MM5}. 
To ensure that the boundary condition in \eqref{3.6bis} is verified, one 
then takes  
\begin{equation}\lb{3.12bis}
h = \big[\big({\textstyle{-\frac12}}I_\dOm+K^{\#}_{z}\big) 
+ \wti \Theta \gamma_D {\mathcal S_{z}}\big]^{-1}g.     
\end{equation}
That the operator 
\begin{equation}\lb{3.1Xbis}
\big({\textstyle{-\frac12}}I_\dOm+K^{\#}_{z}\big) 
+ \wti \Theta \gamma_D {\mathcal S_{z}}:H^{-1/2}(\partial\Omega)\to
H^{-1/2}(\partial\Omega)
\end{equation}
is invertible for all but a discreet set of real values of the parameter 
$z$, can be established based on Hypothesis \ref{h3.1bis} by reasoning as 
before. The key result in this context is that the operator
${\textstyle{-\frac12}}I_\dOm+K^{\#}_{z}\in\cB(H^{-1/2}(\partial\Omega))$
is Fredholm, with Fredholm index zero for every $z\in\bbC$. 
\end{proof}
The special case $\Theta=0$, corresponding to the Neumann Laplacian, is 
singled out below. 
\begin{corollary}\lb{CC3.2bis} 
Assume Hypothesis \ref{h2.1} and suppose that 
$z\in\bbC\backslash\si(-\Delta_{N,\Om})$. Then for every 
$g\in H^{-1/2}(\partial\Omega)$, 
the following Neumann boundary value problem,
\begin{equation}\lb{3.6bisY}
\begin{cases}
(-\Delta - z)u = 0 \text{ in }\,\Om,\quad u \in H^{1}(\Om), \\
\wti\ga_N u = g \text{ on } \,\dOm,
\end{cases}
\end{equation}
has a unique solution $u=u_N$. Moreover, there exists a constant 
$C=C(\Omega,z)>0$ such that
\begin{equation}\lb{3.7bisY}
\|u_N\|_{H^{1}(\Omega)} \leq C\|g\|_{H^{-1/2}(\partial\Omega)}.  
\end{equation}
In particular, 
\begin{equation}\lb{3.8bisY}
\big[\ga_D (-\Delta_{N,\Om}-\ol{z}I_\Om)^{-1}\big]^* \in
\cB\big(H^{-1/2}(\partial\Omega), H^{1}(\Om)\big),  
\end{equation}
and the solution $u_g$ of \eqref{3.6bis} is given by the formula  
\begin{equation}\lb{3.9YY}
u_N=\big(\ga_D(-\Delta_{N,\Om}-\ol{z}I_\Om)^{-1}\big)^*g. 
\end{equation}
Finally, as a byproduct of the well-posedness of \eqref{3.6bisY}, 
the weak Neumann trace $\widetilde\gamma_N$ in \eqref{2.8}, \eqref{2.9} is onto.
\end{corollary}

In the following we denote by $\wti I_{\Om}$ the continuous inclusion 
(embedding) map of $H^1(\Omega)$ into $\bigl(H^1(\Omega)\bigr)^*$. 
By a slight abuse of notation, we also denote the continuous inclusion 
map of $H^1_0(\Omega)$ into $\bigl(H^1_0(\Omega)\bigr)^*$ by the same 
symbol $\wti I_{\Om}$. We recall the ultra weak Neumann trace operator 
$\wti\ga_{\cN}$ in \eqref{2.8X}, \eqref{2.9X}. Finally, assuming 
Hypothesis~\ref{h3.1bis}, we denote by 
\begin{equation} \lb{3.JqY}
- \wti \Delta_{\Theta,\Om}\in\cB\big(H^1(\Om),\big(H^1(\Om)\big)^*\big)
\end{equation}
the extension of $- \Delta_{\Theta,\Om}$ in accordance with \eqref{B.24a}. 
In particular, 
\begin{equation} \lb{3.JqZ}
{}_{H^1(\Om)}\langle u,- \wti \Delta_{\Theta,\Om}v\rangle_{(H^1(\Om))^*}
=\int_{\Om}d^nx\,\ol{\nabla u(x)}\cdot\nabla v(x)
+\big\langle \gamma_D u, \wti \Theta \gamma_D v \big\rangle_{1/2}, 
\qquad u,v\in H^1(\Om),
\end{equation}
and $- \Delta_{\Theta,\Om}$ is the restriction of 
$- \wti \Delta_{\Theta,\Om}$ to $L^2(\Om;d^nx)$ (cf.\ \eqref{B.25}). 

\begin{theorem} \lb{t3.XV} 
Assume Hypothesis \ref{h3.1bis} and suppose that 
$z\in\bbC\backslash\si(-\Delta_{\Theta,\Om})$. Then for every 
$w\in (H^{1}(\Omega))^*$, the following generalized inhomogeneous Robin problem,
\begin{equation} \lb{3.Jq}
\begin{cases}
(-\Delta - z)u = w|_{\Om} \text{ in }\,{\mathcal{D}}'(\Om),
\quad u \in H^{1}(\Om), \\
\wti\ga_{\cN} (u,w)+ \wti \Theta \gamma_D u = 0 \text{ on } \,\dOm,
\end{cases}
\end{equation}
has a unique solution $u=u_{\Theta,w}$. Moreover, there exists a constant 
$C= C(\Theta,\Omega,z)>0$ such that
\begin{equation}\lb{3.Jq2}
\|u_{\Theta,w}\|_{H^{1}(\Omega)} \leq C\|w\|_{(H^{1}(\partial\Omega))^*}.  
\end{equation}
In particular, the operator $(-\Delta_{\Theta,\Om}-zI_\Om)^{-1}$,
$z\in\bbC\backslash\si(-\Delta_{\Theta,\Om})$, originally defined as a bounded 
operator on $\LOm$,  
\begin{align}\label{faH}
(-\Delta_{\Theta,\Om}-zI_\Om)^{-1} \in \cB\big(L^2(\Om;d^nx)\big),
\end{align}
can be extended to a mapping in $\cB\big(\big(H^{1}(\Om)\big)^*,H^1(\Om)\big)$, 
which in fact coincides with 
\begin{equation}\label{fcH}
\big(- \wti \Delta_{\Theta,\Om} -z \wti I_\Om\big)^{-1}
\in\cB\big(\big(H^{1}(\Om)\big)^*,H^1(\Om)\big).
\end{equation} 
\end{theorem}
\begin{proof}
We recall \eqref{jk-9}. Hence, if $w\in \big(H^{1}(\Om)\big)^*$, taking the 
convolution of $w$ with $E_n(z;\dott)$ in \eqref{C.1} and then restricting 
back to $\Omega$ yields a function $u_0\in H^{1}(\Om)$ for which 
$(-\Delta-z)u_0=w|_{\Om}$ in ${\mathcal{D}}'(\Om)$. A solution of 
\eqref{3.Jq} is then given by $u=u_0+u_1$, where $u_1$ satisfies
\begin{equation} \lb{3.6bT}
\begin{cases}
(-\Delta - z)u_1 = 0 \text{ in }\,\Om,\quad u_1 \in H^{1}(\Om), \\
\big(\wti\ga_N + \wti \Theta \gamma_D\big) u_1 
= -\big(\wti\ga_{\cN}(u_0,w) + \wti \Theta \gamma_Du_0\bigr)
\in H^{-1/2}(\partial\Omega) \text{ on } \dOm.
\end{cases}
\end{equation}
Indeed, we have 
\begin{align} \lb{3.6bT2}
\begin{split}
\wti\ga_{\cN}(u,w) &= \wti\ga_{\cN}((u_0,w)+(u_1,0))
=\wti\ga_{\cN}(u_0,w)+\wti\ga_{\cN}(u_1,0)
=\wti\ga_{\cN}(u_0,w)+\wti\ga_Nu_1  \\ 
&= -\wti \Theta \gamma_Du_1-\wti \Theta \gamma_Du_0
=-\wti \Theta \gamma_Du,
\end{split}
\end{align}
by \eqref{2.12X}. 
That the latter boundary problem is solvable is guaranteed by 
Theorem \ref{t3.2bis}. We note that the solution thus constructed satisfies 
\eqref{3.Jq2}. Uniqueness for \eqref{3.Jq} follows from the corresponding 
uniqueness statement in Theorem \ref{t3.2bis}. 

Next, we observe that the inverse operator in \eqref{fcH} is well-defined. 
To prove that 
\begin{equation}\label{fcHX}
- \wti \Delta_{\Theta,\Om} -z \wti I_\Om:H^1(\Om) \to  
\big(H^{1}(\Om)\big)^*,\quad z\in\bbC\backslash \si(-\Delta_{\Theta,\Om}),
\end{equation} 
is onto, assume that $w\in \big(H^{1}(\Om)\big)^*$ is arbitrary and
that $u$ solves \eqref{3.Jq}. Then, for every $v\in H^1(\Om)$ we have
\begin{align}\label{JHK-1}
 {}_{H^1(\Om)}\langle v,(- \wti \Delta_{\Theta,\Om}-z \wti I_\Om)
u\rangle_{(H^1(\Om))^*}
& =\int_{\Om}d^nx\,\ol{\nabla v(x)}\cdot\nabla u(x)
-z\,\int_{\Om}d^nx\,\ol{v(x)}u(x) 
+\big\langle \gamma_D v, \wti \Theta \gamma_D u \big\rangle_{1/2}
\nonumber\\ 
& 
=\int_{\Om}d^nx\,\ol{\nabla v(x)}\cdot\nabla u(x)
-z\,\int_{\Om}d^nx\,\ol{v(x)}u(x) \no \\ 
& \quad -\langle \gamma_D v, \wti\ga_{\cN}(u,w)\rangle_{1/2}
\nonumber\\ 
& 
={}_{H^1(\Om)}\langle v,w\rangle_{(H^1(\Om))^*},
\end{align} 
on account of \eqref{2.9X}, \eqref{3.JqZ}, and \eqref{3.Jq}. 
Since the element $v\in H^1(\Om)$ was arbitrary, this proves that 
$(- \wti \Delta_{\Theta,\Om}-z \wti I_\Om)u=w$, hence the operator
\eqref{fcHX} is onto. In fact, this operator is also one-to-one.
Indeed, assume that $u\in H^1(\Om)$ is such that 
$(- \wti \Delta_{\Theta,\Om}-z \wti I_\Om)u=0$. 
Then, for every $v\in H^1(\Om)$, formula \eqref{3.JqZ} yields 
\begin{align}\label{JHK-2}
\begin{split} 
0 &= {}_{H^1(\Om)}\big\langle v,\big(- \wti \Delta_{\Theta,\Om}-z \wti I_\Om\big)
u\big\rangle_{(H^1(\Om))^*}  \\ 
&= \int_{\Om}d^nx\,\ol{\nabla v(x)}\cdot\nabla u(x)
-z\,\int_{\Om}d^nx\,\ol{v(x)}u(x)
+\big\langle \gamma_D v, \wti \Theta \gamma_D u \big\rangle_{1/2}.
\end{split} 
\end{align} 
Specializing \eqref{JHK-2} to the case when $v\in C^\infty_0(\Om)$ shows
that $(-\Delta -z)u=0$ in the sense of distributions in $\Om$. 
Returning with this into \eqref{JHK-2} we then obtain 
$\big\langle \gamma_D v, (\wti\ga_N+\wti \Theta \gamma_D)u \big\rangle_{1/2}=0$
for every $v\in H^1(\Om)$. Given that the Dirichlet trace $\ga_D$ maps
$H^1(\Om)$ onto $H^{1/2}(\dOm)$, this proves that 
$(\wti\ga_N+\wti \Theta \gamma_D)u=0$ in $H^{-1/2}(\dOm)$ so that ultimately 
$u=0$, since $z\in\bbC\backslash \si(-\Delta_{\Theta,\Om})$. In summary, 
the operator \eqref{fcHX} is an isomorphism. 

Finally, there remains to show that the operators \eqref{faH}, \eqref{fcH}
act in a compatible fashion. To see this, fix 
$z\in\bbC\backslash \si(-\Delta_{\Theta,\Om})$ and assume that 
$w\in L^2(\Om;d^nx)\hookrightarrow (H^1(\Om))^*$. If we then 
set $u= (- \wti \Delta_{\Theta,\Om}-z \wti I_\Om)^{-1}w\in H^1(\Om)$, 
it follows from \eqref{3.JqZ} that 
\begin{align} \lb{2.UU}
\begin{split}
{}_{H^1(\Om)}\langle v,w\rangle_{(H^1(\Om))^*}
&= {}_{H^1(\Om)}\big\langle v,\big(-\wti\Delta_{\Theta,\Om} -z\wti I_{\Om}\big)u
\big\rangle_{(H^1(\Om))^*} \\
&= \int_\Om d^n x\,\ol{\nabla v(x)} \cdot \nabla u(x)  
-z\,\int_\Om d^n x\,\ol{v(x)}u(x)  
 + \big\langle\ga_D v,\wti\Theta\ga_D u \big\rangle_{1/2},
 \end{split}
\end{align}
for every $v\in H^1(\Om)$. Specializing this identity to the case when
$v\in C^\infty_0(\Om)$ yields $(-\Delta -z)u=w\in L^2(\Omega;d^nx)$. 
When used back in \eqref{2.UU}, this observation and \eqref{2.9} 
permit us to conclude that 
\begin{align} \lb{2.UU4}
\langle\ga_D v, (\wti\ga_N+\wti\Theta\ga_D)u \rangle_{1/2}
&= \int_\Om d^n x\,\ol{\nabla v(x)} \cdot \nabla u(x)  
-z\,\int_\Om d^n x\,\ol{v(x)} u(x)  
\nonumber\\ 
& \quad  -{}_{H^1(\Om)}\langle v, \iota(-\Delta u-zu)\rangle_{(H^1(\Om))^*}
+ \big\langle\ga_D v,\wti\Theta\ga_D u \big\rangle_{1/2}
\nonumber\\ 
&= \int_\Om d^n x\,\ol{\nabla v(x)} \cdot \nabla u(x)  
-z\,\int_\Om d^n x\,\ol{v(x)} u(x)  
\nonumber\\[4pt]
& \quad  -{}_{H^1(\Om)}\langle v, w\rangle_{(H^1(\Om))^*}
+ \big\langle\ga_D v,\wti\Theta\ga_D u \big\rangle_{1/2}
\nonumber\\[4pt]
&= 0,
\end{align}
for every $v\in H^1(\Om)$. Upon recalling that the Dirichlet trace $\ga_D$ maps
$H^1(\Om)$ onto $H^{1/2}(\dOm)$, this shows that 
$(\wti\ga_N+\wti\Theta\ga_D)u=0$ in $H^{-1/2}(\dOm)$. 
Thus, $u=(-\Delta_{\Theta,\Om}-zI_{\Om})^{-1}w$, as desired. 
\end{proof}

\begin{remark} \lb{r3.YY}
Similar (yet simpler) considerations also show that the operator 
$(-\Delta_{D,\Om}-zI_\Om)^{-1}$, $z\in\bbC\backslash\si(-\Delta_{D,\Om})$, 
originally defined as bounded operator on $\LOm$,
\begin{equation}\label{fuTT}
(-\Delta_{D,\Om}-zI_\Om)^{-1} \in \cB\big(L^2(\Om;d^nx)\big),
\end{equation}
extends to a mapping
\begin{equation}\label{fcH3}
\big(-\wti \Delta_{D,\Om}-z \wti I_\Om\big)^{-1}
\in\cB\big(H^{-1}(\Om);H^1_0(\Om)\big).
\end{equation}
Here $- \wti \Delta_{D,\Om}\in\cB\big(H^1_0(\Om), H^{-1}(\Om)\big)$ is the 
extension of $- \Delta_{D,\Om}$ in accordance with \eqref{B.24a}. 
Indeed, the Lax--Milgram lemma applies and yields that 
\begin{equation}\label{fuRR}
\big(- \wti \Delta_{D,\Om}-z \wti I_\Om\big) \colon H^1_0(\Om)\to 
\big(H^1_0(\Om)\big)^* = H^{-1}(\Om) 
\end{equation}
is, in fact, an isomorphism whenever $z\in\bbC\backslash\si(-\Delta_{D,\Om})$.
\end{remark}

\begin{corollary} \lb{t3.XU} 
Assume Hypothesis \ref{h3.1bis} and suppose that 
$z\in\bbC\backslash\si(-\Delta_{\Theta,\Om})$. Then the operator 
$M_{\Theta,D,\Om}^{(0)}(z)\in \cB\big(\LdOm\big)$ in \eqref{3.48}, \eqref{3.49} 
extends $($in a compatible manner\,$)$ to 
\begin{equation}
\wti M_{\Theta,D,\Om}^{(0)}(z) 
\in \cB\big(H^{-1/2}(\partial\Om) , H^{1/2}(\partial\Om) \big), \quad 
z\in\bbC\backslash\si(-\Delta_{\Theta,\Om}).    \lb{3.50bis}
\end{equation}
In addition, $\wti M_{\Theta,D,\Om}^{(0)}(z)$ permits the representation 
\begin{equation}
\wti M_{\Theta,D,\Om}^{(0)}(z) 
= \gamma_D \big(-\wti\Delta_{\Theta,\Om} - z \wti I_\Om\big)^{-1}\gamma_D^*, 
\quad z\in\bbC\backslash\si(-\Delta_{\Theta,\Om}). \lb{3.52bis}
\end{equation} 
The same applies to the adjoint 
$M_{\Theta,D,\Om}^{(0)}(z)^*\in \cB\big(\LdOm\big)$ of 
$M_{\Theta,D,\Om}^{(0)}(z)$, resulting in the bounded extension 
$\big(\wti M_{\Theta,D,\Om}^{(0)}(z)\big)^* \in \cB\big(H^{-1/2}(\partial\Om)$, 
$H^{1/2}(\partial\Om) \big)$, $z\in\bbC\backslash\si(-\Delta_{\Theta,\Om})$.   
\end{corollary}
\begin{proof}
The claim \eqref{3.50bis} is a direct consequence of Theorem \ref{t3.2bis}, 
while the claim \eqref{3.52bis} follows from the fact that 
\begin{equation}\label{GGG}
\ga_D^*  \colon \big(H^{1/2}(\dOm)\big)^*
=H^{-1/2}(\dOm)\to \big(H^{1}(\Om)\big)^* 
\end{equation}
in a bounded fashion (cf. \eqref{ga*}, \eqref{fcH} and \eqref{3.52}). 
The rest follows from dualizing these claims.
\end{proof}

The following regularity result for the Robin resolvent will also play 
an important role shortly. 

\begin{lemma}\label{L-BbbZa}
Assume Hypothesis \ref{3.5} and suppose that 
$z\in\bbC\backslash\si(-\Delta_{\Theta,\Om})$. Then 
\begin{equation}\lb{Kr-a.1a}
(-\Delta_{\Theta,\Om} - zI_\Om)^{-1}:\LOm\to\bigl\{u\in H^{3/2}(\Omega)\,|\,
\Delta u\in L^2(\Omega;d^nx)\bigl\}
\end{equation}
is a well-defined bounded operator, where the space 
$\{u\in H^{3/2}(\Omega)\,|\,\Delta u\in L^2(\Omega;d^nx)\}$ is 
equipped with the natural graph norm 
$u\mapsto\|u\|_{H^{3/2}(\Omega)}+\|\Delta u\|_{L^2(\Omega;d^nx)}$. 
\end{lemma}
\begin{proof} 
Consider $f\in \LOm$ and set $u=(-\Delta_{\Theta,\Om} - zI_\Om)^{-1}f$. 
It follows that $u$ is the unique solution of the problem
\begin{equation}\lb{Bv-1Za}
(-\Delta -z)u=f\, \text{ in } \, \Omega,\quad u\in H^1(\Omega),
\quad \big(\wti\gamma_N+\wti\Theta\gamma_D\big)u=0 \text{ on } \, \partial\Omega.
\end{equation}
The strategy is to devise an alternative representation for $u$ 
from which it is clear that $u$ has the claimed regularity in $\Omega$. 
To this end, let $\widetilde{f}$ denote the extension of $f$ by zero to 
${\mathbb{R}}^n$ and denote by $E$ the operator of convolution by 
$E_n(z;\dott)$. Since the latter is smoothing of order $2$, it follows that 
$w=(E\widetilde{f})|_{\Omega}\in H^2(\Omega)$ and $(-\Delta -z)w=f$ in $\Omega$.
Also, let $v$ be the unique solution of the problem 
\begin{equation}\lb{Bv-2Za}
(-\Delta -z)v=0\, \text{ in } \, \Omega,\quad v\in H^{3/2}(\Omega),
\quad \big(\wti\gamma_N+\wti\Theta\gamma_D\big)v
=-\big(\gamma_N+\wti\Theta\gamma_D\big)w\text{ on } \, \partial\Omega.
\end{equation}
That \eqref{Bv-2Za} is solvable is a consequence of Theorem \ref{t3.2bis}.
Then $v+w$ also solves \eqref{Bv-1Za} so that, by uniqueness, $u=v+w$.
This shows that $u$ has the desired regularity properties and, hence, 
the operator \eqref{Kr-a.1a} is well-defined and bounded. 
\end{proof}

Under the Hypothesis \ref{h3.1bis}, \eqref{fcH} and \eqref{2.6} yield 
\begin{eqnarray}\label{Obt.1}
\gamma_D\big(-\wti\Delta_{\Theta,\Omega}-zI_{\Omega}\big)^{-1}
\in \cB\bigl((H^1(\Omega))^*,H^{1/2}(\partial\Omega)\bigr). 
\end{eqnarray}
Hence, by duality, 
\begin{eqnarray}\label{Obt.2}
\big[\gamma_D\big(-\wti\Delta_{\Theta,\Omega}-zI_{\Omega}\big)^{-1}\big]^*
\in \cB\bigl(H^{-1/2}(\partial\Omega),H^1(\Omega)\bigr). 
\end{eqnarray}
We wish to complement this with the following result. 

\begin{corollary}\label{L-fR8}
Assume Hypothesis \ref{3.5} and suppose that 
$z\in\bbC\backslash\si(-\Delta_{\Theta,\Om})$. Then 
\begin{equation}\lb{Kr-a.2}
\gamma_D(-\Delta_{\Theta,\Om}-zI_\Om)^{-1}\in
\cB\bigl(\LOm,H^{1}(\partial\Omega)\bigr). 
\end{equation}
In particular, 
\begin{equation}\lb{Kr-a.3}
\big[\gamma_D(-\Delta_{\Theta,\Om}-zI_\Om)^{-1}\big]^*\in
\cB\bigl(H^{-1}(\partial\Omega),\LOm\bigr)\hookrightarrow
\cB\bigl(L^2(\partial\Omega;d^{n-1}\Omega),\LOm\bigr).
\end{equation}
In addition, the operator \eqref{Kr-a.3} is compatible with \eqref{Obt.2}
in the sense that 
\begin{eqnarray}\label{Obt.3}
\big[\gamma_D(-\Delta_{\Theta,\Omega}-zI_{\Omega})^{-1}\big]^*f
=\big[\gamma_D\big(-\wti\Delta_{\Theta,\Omega}-zI_{\Omega}\big)^{-1}\big]^*f
\mbox{ in } \LOm, \quad f\in H^{-1/2}(\partial\Omega).
\end{eqnarray}
As a consequence, 
\begin{eqnarray}\label{Obt.3bis}
\big[\gamma_D(-\Delta_{\Theta,\Omega}-zI_{\Omega})^{-1}\big]^*f
=\big[\gamma_D\big(-\wti\Delta_{\Theta,\Omega}-zI_{\Omega}\big)^{-1}\big]^*f\mbox{ in }\LOm, 
\quad f\in L^{2}(\partial\Omega;d^{n-1}\omega).
\end{eqnarray}
\end{corollary}
\begin{proof} 
The first part of the statement is an immediate consequence 
of Lemma \ref{L-BbbZa} and Lemma \ref{Gam-L1}. As for \eqref{Obt.3}, 
pick $f\in H^{-1/2}(\partial\Omega)\hookrightarrow H^{-1}(\partial\Omega)$ 
and $u\in\LOm\hookrightarrow (H^1(\Omega))^*$ arbitrary. 
We may then write 
\begin{align}\label{Obt.4}
\langle [\gamma_D(-\Delta_{\Theta,\Omega}-zI_{\Omega})^{-1}]^*f\,,\,u\rangle
_{\LOm} &= \langle f\,,\,\gamma_D
(-\Delta_{\Theta,\Omega}-\ol{z}I_{\Omega})^{-1}u\rangle_{1}
\nonumber\\
&= \langle f\,,\,\gamma_D
(-\Delta_{\Theta,\Omega}-\ol{z}I_{\Omega})^{-1}u\rangle_{1/2}
\nonumber\\
&= \big\langle f\,,\,\gamma_D
\big(-\wti\Delta_{\Theta,\Omega}-\ol{z}I_{\Omega}\big)^{-1}u\big\rangle_{1/2}
\nonumber\\
&= {}_{H^{1}(\Omega)}
\big\langle 
\big[\gamma_D\big(-\wti\Delta_{\Theta,\Omega}-zI_{\Omega}\big)^{-1}\big]^*f\,,\,u
\big\rangle_{(H^1(\Omega))^*}
\nonumber\\
&= \big\langle 
\big[\gamma_D\big(-\wti\Delta_{\Theta,\Omega}-zI_{\Omega}\big)^{-1}\big]^*f\,,\,u
\big\rangle_{\LOm},
\end{align}
since \eqref{faH} and \eqref{fcH} are compatible. This gives \eqref{Obt.3}.
Since $L^2(\partial\Omega;d^{n-1}\omega)\hookrightarrow 
H^{-1/2}(\dOm)$, \eqref{Obt.3bis} also follows. 
\end{proof}

We will need a similar compatibility result for the composition 
between the Neumann trace and resolvents of the Dirichlet Laplacian. 
To state it, recall the restriction operator $R_{\Om}$ from \eqref{jk-10}.
Also, denote by $I_{\bbR^n}$ the identity operator (for spaces of functions
defined in $\bbR^n$). Finally, recall the space \eqref{2.88X} and the 
ultra weak Neumann trace operator in \eqref{2.8X}, \eqref{2.9X}.

\begin{lemma}\label{new-L1}
Assume Hypothesis \ref{h2.1}. Then
\begin{eqnarray}\label{Obt.1P}
\big(\big(-\wti\Delta_{D,\Omega}-z\wti{I}_{\Omega}\big)^{-1}\circ R_{\Omega} 
,I_{\bbR^n}\big):(H^1(\Omega))^*\to W_z(\Om),
\quad z\in\bbC\backslash \si(-\Delta_{D,\Omega}),
\end{eqnarray}
is a well-defined, linear and bounded operator. Consequently, 
\begin{eqnarray}\label{Obt.2P}
\wti\gamma_{\cN}
\big(\big(-\wti\Delta_{D,\Omega}-z\wti{I}_{\Omega}\big)^{-1}\circ R_{\Omega} 
,I_{\bbR^n}\big)\in\cB\bigl((H^1(\Omega))^*,H^{-1/2}(\partial\Omega)\bigr),
\quad z\in\bbC\backslash \si(-\Delta_{D,\Omega}),
\end{eqnarray}
and, hence, 
\begin{eqnarray}\label{Obt.3P}
\big[\wti\gamma_{\cN}
\big(\big(-\wti\Delta_{D,\Omega}-z\wti{I}_{\Omega}\big)^{-1}\circ R_{\Omega} 
,I_{\bbR^n}\big)\big]^*
\in\cB\bigl(H^{1/2}(\partial\Omega),H^1(\Omega)\bigr),\quad
z\in\bbC\backslash \si(-\Delta_{D,\Omega}).
\end{eqnarray}
Furthermore, the operators \eqref{Obt.2P}, \eqref{Obt.3P} are
compatible with \eqref{3.23} and \eqref{3.24}, respectively, in the
sense that for each $z\in\bbC\backslash \si(-\Delta_{D,\Omega})$, 
\begin{align}\label{Obt.3q}
\wti\gamma_N(-\Delta_{D,\Omega}-zI_{\Omega})^{-1}f
=\wti\gamma_{\cN}
\big(\big(-\wti\Delta_{D,\Omega}-z\wti{I}_{\Omega}\big)^{-1}\circ R_{\Omega} 
,I_{\bbR^n}\big)f\mbox{ in }H^{-1/2}(\partial\Omega),\,\,f\in\LOm,
\end{align}
and 
\begin{align}\label{Obt.3bq} 
\begin{split}
& \big[\wti\gamma_N(-\Delta_{D,\Omega}-zI_{\Omega})^{-1}\big]^*f
=\big[\wti\gamma_{\cN}
\big(\big(-\wti\Delta_{D,\Omega}-z\wti{I}_{\Omega}\big)^{-1}\circ R_{\Omega} 
,I_{\bbR^n}\big)\big]^*f \, \text{ in }\LOm,  \\
& \hspace*{7.5cm}  \text{for every element }\, f\in H^{1/2}(\partial\Omega).
\end{split}
\end{align}
\end{lemma}
\begin{proof}
Let $z\in\bbC\backslash \si(-\Delta_{D,\Omega})$.
If $f\in (H^1(\Omega))^*$ and 
$u=\big(-\wti\Delta_{D,\Omega}-zI_{\Omega}\big)^{-1}(f|_{\Omega})$,  
then $u\in H^1_0(\Omega)$ satisfies $(-\Delta-z)u=f|_{\Omega}$ in 
${\mathcal{D}}'(\Om)$. Hence, $(u,f)\in W_z(\Omega)$
which shows that the operator \eqref{Obt.1P} is well-defined and 
bounded. Then \eqref{Obt.2P} is a consequence of this and \eqref{2.8X}, 
whereas \eqref{Obt.3P} follows from \eqref{Obt.2P} and duality. 

Going further, \eqref{Obt.3q} is implied by Lemma \ref{L-BbbZ}, the 
compatibility statement in Lemma \ref{Neu-tr}, and \eqref{2.10X}--\eqref{2.12X}.
There remains to justify \eqref{Obt.3bq}. To this end, 
if $f\in H^{1/2}(\partial\Omega)\hookrightarrow\LdOm$ 
and $u\in\LOm\hookrightarrow (H^1(\Omega))^*$ are arbitrary, we may write
\begin{align}\label{Obt.4q}
\langle [\wti\gamma_N(-\Delta_{D,\Omega}-zI_{\Omega})^{-1}]^*f\,,\,u\rangle
_{\LOm} &= \langle f\,,\,\wti\gamma_N
(-\Delta_{D,\Omega}-\ol{z}I_{\Omega})^{-1}u\rangle_{0}
\nonumber\\
&= \langle f\,,\,
\wti\gamma_N\big(-\Delta_{D,\Omega}-\ol{z}I_{\Omega}\big)^{-1}u\big\rangle_{1/2}
u\rangle_{1/2}
\\
&= \big\langle f\,,\,
\wti\gamma_{\cN}
\big(\big(-\wti\Delta_{D,\Omega}-\ol{z}\wti{I}_{\Omega}\big)^{-1}\circ R_{\Omega} 
,I_{\bbR^n}\big)u\rangle_{1/2}
\nonumber\\
&= {}_{H^{1}(\Omega)}
\big\langle 
\big[\wti\gamma_{\cN}
(\big(-\wti\Delta_{D,\Omega}-z\wti{I}_{\Omega}\big)^{-1}\circ R_{\Omega} 
,I_{\bbR^n}\big)\big]^*f\,,\,u\big\rangle_{(H^1(\Omega))^*}
\nonumber\\
&= \big\langle 
\big[\wti\gamma_{\cN}
(\big(-\wti\Delta_{D,\Omega}-z\wti{I}_{\Omega}\big)^{-1}\circ R_{\Omega}
,I_{\bbR^n}\big)\big]^*f\,,\,u\big\rangle_{\LOm},
\nonumber
\end{align}
where the third equality is based on \eqref{Obt.3q}. 
This justifies \eqref{Obt.3bq} and finishes the proof of the lemma. 
\end{proof}

\begin{lemma} \lb{lA.3BB}
Assume Hypothesis \ref{h3.1bis} and suppose that 
$z\in\bbC\backslash(\si(-\Delta_{\Theta,\Om})\cup\si(-\Delta_{D,\Om}))$.
Then the following resolvent relation holds on $\bigl(H^1(\Omega)\bigr)^*$, 
\begin{align}\lb{Na1B}
\begin{split}
\big(-\wti\Delta_{\Theta,\Om}-z \wti I_\Om\big)^{-1} 
&=\big(-\wti\Delta_{D,\Om}-z \wti I_\Om\big)^{-1}\circ R_{\Omega}   \\
& \quad  + \big(-\wti\Delta_{\Theta,\Om}-z \wti I_\Om\big)^{-1}\ga_D^*
\wti\gamma_{\cN}\big(\big(-\wti\Delta_{D,\Om}-z \wti I_\Om\big)^{-1}\circ R_{\Omega}
,I_{{\bbR}^n}\big).
\end{split}
\end{align}
\end{lemma}
\begin{proof}
To set the stage, we recall \eqref{2.8}--\eqref{2.9} and \eqref{GGG}.  
Together with \eqref{fcH} and \eqref{fcH3}, these ensure that 
the composition of operators appearing on the right-hand side of \eqref{Na1B} 
is well-defined. Next, let $\phi_1,\psi_1\in L^2(\Om;d^nx)$ be arbitrary 
and define
\begin{align}
\begin{split}
\phi &= (-\Delta_{\Theta,\Om}-\ol{z}I_\Om)^{-1}\phi_1 \in \dom(\Delta_{\Theta,\Om})
\subset \big(H^{1}(\Om) \cap \dom(\wti \gamma_N)\big),
\\
\psi &= (-\Delta_{D,\Om}-zI_\Om)^{-1}\psi_1 \in \dom(\Delta_{D,\Om})
\subset \big(H^{1}_0(\Om) \cap \dom(\wti \gamma_N)\big).
\end{split} \lb{Na2}
\end{align} 
As a consequence of our earlier results, 
both sides of \eqref{Na1B} are bounded operators from $(H^1(\Om))^*$ 
into $H^1(\Om)$. Since $L^2(\Om;d^nx)\hookrightarrow (H^1(\Om))^*$ densely, 
it therefore suffices to show that the following identity holds:
\begin{align}
\begin{split}
&(\phi_1,(-\Delta_{\Theta,\Om}-zI_\Om)^{-1}\psi_1)_{L^2(\Om;d^nx)} 
-(\phi_1,(-\Delta_{D,\Om}-zI_\Om)^{-1}\psi_1)_{L^2(\Om;d^nx)}
\\
&\quad = (\phi_1,(-\Delta_{\Theta,\Om}-zI_\Om)^{-1}\ga_D^*\wti\gamma_N
(-\Delta_{D,\Om}-zI_\Om)^{-1}\psi_1)_{L^2(\Om;d^nx)}.
\end{split}
\end{align}
We note that according to \eqref{Na2} one has,
\begin{align}
(\phi_1,(-\Delta_{D,\Om}-zI_\Om)^{-1}\psi_1)_{L^2(\Om;d^nx)}
&= ((-\Delta_{\Theta,\Om}-\ol{z}I_\Om)\phi,\psi)_{L^2(\Om;d^nx)},
\\
(\phi_1,(-\Delta_{\Theta,\Om}-zI_\Om)^{-1}\psi_1)_{L^2(\Om;d^nx)}
&= \big(\big((-\Delta_{\Theta,\Om}-zI_\Om)^{-1}\big)^*
\phi_1,\psi_1\big)_{L^2(\Om;d^nx)} \no
\\
&= ((-\Delta_{\Theta,\Om}-\ol{z}I_\Om)^{-1}\phi_1,\psi_1)_{L^2(\Om;d^nx)}
\no
\\
&= (\phi,(-\Delta_{D,\Om}-zI_\Om)\psi)_{L^2(\Om;d^nx)},
\end{align}
and, further, 
\begin{align}
&(\phi_1,(-\Delta_{\Theta,\Om}-zI_\Om)^{-1}\ga_D^*\wti\gamma_N
(-\Delta_{D,\Om}-zI_\Om)^{-1}\psi_1)_{L^2(\Om;d^nx)} \no
\\
&\quad ={}_{H^1(\Om)}
\langle{(-\Delta_{\Theta,\Om}-\ol{z}I_\Om)^{-1}\phi_1},\ga_D^*\wti\gamma_N
(-\Delta_{D,\Om}-zI_\Om)^{-1}\psi_1\rangle_{(H^1(\Om))^*} \no
\\
&\quad = \langle{\ga_D(-\Delta_{\Theta,\Om}-\ol{z}I_\Om)^{-1}\phi_1},
\wti\gamma_N (-\Delta_{D,\Om}-zI_\Om)^{-1}\psi_1\rangle_{1/2} 
=\langle{\ga_D\phi},\wti\gamma_N\psi\rangle_{1/2}.
\end{align}
Thus, matters have been reduced to proving that
\begin{align} \lb{Na3}
((-\Delta_{\Theta,\Om}-\ol{z}I_\Om)\phi,\psi)_{L^2(\Om;d^nx)} -
(\phi,(-\Delta_{D,\Om}-zI_\Om)\psi)_{L^2(\Om;d^nx)} =
\langle{\ga_D\phi},\wti\gamma_N\psi\rangle_{1/2}.
\end{align}
Using \eqref{wGreen} for the left-hand side of \eqref{Na3} one obtains
\begin{align}
&((-\Delta_{\Theta,\Om}-\ol{z}I_\Om)\phi,\psi)_{L^2(\Om;d^nx)} -
(\phi,(-\Delta_{D,\Om}-zI_\Om)\psi)_{L^2(\Om;d^nx)} \no
\\
&\quad = -(\Delta\phi,\psi)_{L^2(\Om;d^nx)} +
(\phi,\Delta\psi)_{L^2(\Om;d^nx)}
\\
&\quad = (\nabla\phi,\nabla\psi)_{L^2(\Om;d^nx)^n} -
\langle{\wti\gamma_N\phi},\ga_D\psi\rangle_{1/2} -
(\nabla\phi,\nabla\psi)_{L^2(\Om;d^nx)^n} +
\langle{\ga_D\phi},\wti\gamma_N\psi\rangle_{1/2} \no
\\
&\quad = -\langle{\wti\gamma_N\phi},\ga_D\psi\rangle_{1/2} +
\langle{\ga_D\phi},\wti\gamma_N\psi\rangle_{1/2}. \no
\end{align}
Observing that $\gamma_D\psi=0$ since $\psi\in\dom(\Delta_{D,\Om})$,
one concludes \eqref{Na3}.
\end{proof}

The stage is now set for proving the $L^2$-version of Lemma \ref{lA.3BB}. 

\begin{lemma}  \lb{l4.8}
Assume Hypothesis \ref{h3.1} and suppose that 
$z\in\bbC\backslash(\si(-\Delta_{\Theta,\Om})\cup\si(-\Delta_{D,\Om}))$.
Then the following resolvent relation holds on $\LOm$, 
\begin{align}
\begin{split}
(-\Delta_{\Theta,\Om}-zI_\Om)^{-1} 
&= (-\Delta_{D,\Om}-zI_\Om)^{-1} 
+ \big[\ga_D (-\Delta_{\Theta,\Om}-{\ol z} I_\Om)^{-1}\big]^* 
\big[\wti\gamma_N (-\Delta_{D,\Om}-zI_\Om)^{-1}\big]   \lb{Na1C}  \\
&= (-\Delta_{D,\Om}-zI_\Om)^{-1} 
+ \big[\wti\gamma_N (-\Delta_{D,\Om}-{\ol z}I_\Om)^{-1}\big]^*
\big[\ga_D (-\Delta_{\Theta,\Om}-z I_\Om)^{-1}\big]. 
\end{split}
\end{align}
\end{lemma}
\begin{proof}
Consider the first equality in \eqref{Na1C}. To begin with, we 
note that the following operators are well-defined, linear and bounded: 
\begin{eqnarray}\label{Kr-a.5q}
&& (-\Delta_{D,\Om}-zI_\Om)^{-1}\in\cB\bigl(\LOm\bigr),\quad
(-\Delta_{\Theta,\Om}-zI_\Om)^{-1}\in\cB\bigl(\LOm\bigr),
\\[4pt]
\label{Kr-a.6q}
&&\wti\gamma_N(-\Delta_{D,\Om}-zI_\Om)^{-1}
\in\cB\bigl(\LOm,\LdOm\bigr),
\\[4pt]
\label{Kr-a.9q}
&& \big[\gamma_D(-\Delta_{\Theta,\Om}-\ol{z}I_\Om)^{-1}\big]^* 
\in\cB\bigl(\LdOm,\LOm)\bigr).
\end{eqnarray}
Indeed, \eqref{Kr-a.5q} follows from the fact that
$z\in\bbC\backslash\big(\si(-\Delta_{D,\Om})\cup\si(-\Delta_{\Theta,\Om})\big)$,
\eqref{Kr-a.6q} is covered by \eqref{3.23}, and \eqref{Kr-a.9q} is taken 
care of by \eqref{Kr-a.3}. Together, these memberships show that both sides 
of \eqref{Na1C} are bounded operators on $\LOm$. Having established this, 
the first equality in \eqref{Na1C} follows from Lemma \ref{lA.3BB},
granted the compatibility results from Corollary \ref{L-fR8} and
Lemma \ref{new-L1}. Then the second equality in \eqref{Na1C} is a 
consequence of what we have proved so far and of duality. 
\end{proof}

We note that the special case $\Theta =0$ in Lemma \ref{l4.8} was 
discussed by Nakamura \cite{Na01} (in connection with cubic boxes $\Omega$) 
and subsequently in \cite[Lemma A.3]{GLMZ05} (in the case of a Lipschitz 
domain with a compact boundary).  

\begin{lemma} \lb{lA.3CC}
Assume Hypothesis \ref{h3.1bis} and suppose that 
$z\in\bbC\backslash\si(-\Delta_{\Theta,\Om})$. Then
\begin{equation} \lb{NaD}
\big[\wti M_{\Theta,D,\Om}^{(0)}(z)\big]^* = \wti M_{\Theta,D,\Om}^{(0)}(\ol{z})
\end{equation}
as operators in $\cB\big(H^{-1/2}(\partial\Omega);H^{1/2}(\partial\Omega)\big)$.
In particular, assuming Hypothesis \ref{h3.1}, then
\begin{equation} \lb{NaDL2}
\big[M_{\Theta,D,\Om}^{(0)}(z)\big]^* = M_{\Theta,D,\Om}^{(0)}(\ol{z}). 
\end{equation}
\end{lemma}
\begin{proof}
Let $f,g\in H^{-1/2}(\partial\Omega)$. Then using the definition of 
$\wti M_{\Theta,D,\Om}^{(0)}(z)$ one infers  
\begin{equation}\label{miT1}
\big\langle \wti M_{\Theta,D,\Om}^{(0)}(z)f, g \big\rangle_{1/2}
= \big\langle\gamma_D u,
\big(\wti\ga_N + \wti\Theta \gamma_D\big)v\big\rangle_{1/2}
\end{equation}
where $u,v$ solve the Robin boundary value problems
\begin{equation} \lb{3.Wa}
\begin{cases}
(-\Delta - z)u = 0 \text{ in }\,\Om,\quad u \in H^{1}(\Om), \\
\big(\wti\ga_N + \wti \Theta \gamma_D\big) u = f \text{ on } \,\dOm,
\end{cases}
\end{equation}
and
\begin{equation} \lb{3.Wa2}
\begin{cases}
(-\Delta - \ol{z})v = 0 \text{ in }\,\Om,\quad v \in H^{1}(\Om), \\
\big(\wti\ga_N + \wti \Theta \gamma_D\big) v = g \text{ on } \,\dOm,
\end{cases} 
\end{equation} 
respectively. That this is possible is ensured by Theorem \ref{t3.2bis}. 
Using \eqref{wGreen} we may then write 
\begin{align}\label{lonG}
\big\langle\gamma_D u,
\big(\wti\ga_N + \wti\Theta \gamma_D\big)v\big\rangle_{1/2}
&= \langle\gamma_D u,\wti\ga_N v\rangle_{1/2}
+ \big\langle\gamma_D u,\wti\Theta \gamma_D v\big\rangle_{1/2}
\nonumber \\ 
&= (\nabla u, \nabla v)_{L^2(\Om;d^nx)^n} 
+{}_{H^1(\Om)}\langle u,\Delta v\rangle_{(H^{1}(\Om))^*}
+ \big\langle\gamma_D u,\wti\Theta \gamma_D v\big\rangle_{1/2}
\nonumber\\ 
&= (\nabla u, \nabla v)_{L^2(\Om;d^nx)^n} 
- \ol{z} \, (u,v)_{L^2(\Om;d^nx)^n} 
+ \big\langle\gamma_D u,\wti\Theta \gamma_D v\big\rangle_{1/2}
\nonumber\\ 
&= \ol{(\nabla v, \nabla u)_{L^2(\Om;d^nx)^n}
- z \, (v,u)_{L^2(\Om;d^nx)^n}}
+ \big\langle\gamma_D u,\wti\Theta \gamma_D v\big\rangle_{1/2}
\nonumber\\ 
&= \ol{(\nabla v, \nabla u)_{L^2(\Om;d^nx)^n}
+{}_{H^1(\Om)}\langle v,\Delta u\rangle_{(H^{1}(\Om))^*}}
+ \big\langle\gamma_D u,\wti\Theta \gamma_D v\big\rangle_{1/2}
\nonumber\\ 
&= \ol{\langle\gamma_D v,\wti\gamma_N u\rangle_{1/2}}
+ \big\langle\gamma_D u,\wti\Theta \gamma_Dv\big\rangle_{1/2}
\nonumber\\ 
&= \ol{\langle\gamma_D v,\wti\gamma_N u\rangle_{1/2}}
+ \big\langle \wti\Theta\gamma_D u,\gamma_D v\big\rangle_{1/2}
\nonumber\\ 
&= \ol{\langle\gamma_D v,\wti\gamma_N u\rangle_{1/2}
+ \big\langle \gamma_D v,\wti\Theta\gamma_D u\big\rangle_{1/2}}
\nonumber\\ 
&= \ol{\big\langle\gamma_D v,\big(\wti\gamma_N 
+ \wti\Theta\gamma_D\big)u\big\rangle_{1/2}}
\nonumber\\ 
&= \ol{\big\langle \wti M_{\Theta,D,\Om}^{(0)}(\ol{z})g, f\big\rangle_{1/2}}.
\nonumber\\ 
&= \big\langle f, \wti M_{\Theta,D,\Om}^{(0)}(\ol{z})g\big\rangle_{1/2}.
\end{align}
Now \eqref{NaD} follows from \eqref{miT1} and \eqref{lonG}. Finally, 
\eqref{NaDL2} follows from \eqref{NaD} by restriction of the latter to $\LdOm$.
\end{proof}

Next we briefly turn to the Herglotz property of the Robin-to-Dirichlet map.
 We recall that an operator-valued function $M(z)\in\cB(\cH)$, $z\in\bbC_+$ 
(where $\bbC_+=\{z\in\bbC\,|\, \Im(z)>0$), for some separable complex Hilbert 
space $\cH$, is called an {\it operator-valued Herglotz function} if 
$M(\dott)$ is analytic on $\bbC_+$ and
\begin{equation}
\Im(M(z)) \ge 0, \quad z\in\bbC_+.  \lb{4.41}
\end{equation} 
Here, as usual, $\Im(M)=(M-M^*)/(2i)$.

\begin{lemma} \lb{l4.13}
Assume Hypothesis \ref{h3.1bis} and suppose that $z\in\bbC_+$. Then
for every $g\in H^{-1/2}(\dOm)$, $g \neq 0$, 
\begin{equation}\lb{4.42}
\f{1}{2i}\big\langle\ g,\big[\wti M_{\Theta,D}(z) 
- \wti M_{\Theta,D}(z)^*\big]g\big\rangle_{1/2}=\Im(z) \|u_{\Theta}\|^2_{\LOm}
> 0, 
\end{equation}
where $u_{\Theta}$ satisfies
\begin{equation}\lb{4.43}  
\begin{cases}
(-\Delta - z)u = 0 \text{ in }\,\Om,\quad u \in H^{1}(\Om), \\
\big(\wti\ga_N + \wti \Theta \gamma_D\big) u = g \text{ on } \,\dOm.   
 \end{cases}
\end{equation}
In particular, assuming Hypothesis \ref{h3.1}, then 
\begin{equation} \lb{4.44}
\Im\big(M_{\Theta,D,\Om}^{(0)}(z)\big) \ge 0, \quad z\in\bbC_+,
\end{equation}
and hence $M_{\Theta,D,\Om}^{(0)}(\dott)$ is an operator-valued Herglotz 
function on $\LdOm$. 
\end{lemma}
\begin{proof}
Let $u_{\Theta}$ be given by the solution of \eqref{4.43}. Then 
$M_{\Theta,D,\Om}^{(0)}g=\gamma_D u_{\Theta}$ by \eqref{3.48}, and using 
self-adjointness of $\wti \Theta$ (in the sense of \eqref{B.5}) 
and the Green's formula \eqref{wGreen}, one computes,
\begin{align} 
\f{1}{2i}\big\langle\ g,\big[\wti M_{\Theta,D}(z) 
- \wti M_{\Theta,D}(z)^*\big]g\big\rangle_{1/2} 
& =
\f{1}{2i}\big[\big\langle\ g, \wti M_{\Theta,D}(z) g\big\rangle_{1/2} - 
\big\langle\ \wti M_{\Theta,D}(z) g,g\big\rangle_{1/2}\big]  \no \\
& = \f{1}{2i}\big[\big\langle \big(\wti\ga_N + \wti\Theta \gamma_D\big) u_{\Theta}, 
\gamma_D u_{\Theta}\big\rangle_{1/2} -
\big \langle\gamma_D u_{\Theta},\big(\wti\ga_N + \wti\Theta \gamma_D\big) 
u_{\Theta}\big\rangle_{1/2}\big]  \no \\ 
& = \f{1}{2i}\big[\langle \wti\ga_N u_{\Theta}, 
\gamma_D u_{\Theta}\rangle_{1/2} -
 \langle\gamma_D u_{\Theta}, \wti\ga_N u_{\Theta}\rangle_{1/2}\big]  \no \\
&\quad + \f{1}{2i}\big[\big\langle \wti\Theta \gamma_D u_{\Theta}, 
\gamma_D u_{\Theta}\big\rangle_{1/2} -
 \big\langle\gamma_D u_{\Theta}, \wti\Theta \gamma_D u_{\Theta}\big\rangle_{1/2}\big]\no\\
& =\Im(\langle \wti\ga_N u_{\Theta}, 
\gamma_D u_{\Theta}\rangle_{1/2}) \no \\
& = \Im[(\nabla u_{\Theta}, \nabla u_{\Theta})_{\LOm} + 
{}_{H^1(\Om)}\langle \Delta u_{\Theta}, u_{\Theta}\rangle_{(H^{1}(\Om))^*}]\no\\
&= \Im(- {\ol z} \, 
{}_{H^1(\Om)}\langle u_{\Theta}, u_{\Theta}\rangle_{(H^{1}(\Om))^*})  \no \\
& = \Im(z) \, {}_{H^1(\Om)}\langle u_{\Theta}, u_{\Theta}\rangle_{(H^{1}(\Om))^*}
\no \\
& = \Im(z) \|u_{\Theta}\|^2_{\LOm} > 0   \label{4.45}
\end{align}
since $g \neq 0$ implies $u_\Theta \neq 0$. This proves \eqref{4.42}. 
Restriction of \eqref{4.42} to $g\in\LdOm$ then yields \eqref{4.44}. 
\end{proof}

Returning to the principal goal of this section, we now prove the 
following variant of a Krein-type resolvent formula relating 
$\wti\Delta_{\Theta,\Om}$ and $\wti\Delta_{D,\Om}$:

\begin{theorem} \lb{tA.3LL}
Assume Hypothesis \ref{h3.1bis} and suppose that 
$z\in\bbC\backslash(\si(-\Delta_{\Theta,\Om})\cup\si(-\Delta_{D,\Om}))$.
Then the following Krein formula holds on $\bigl(H^1(\Omega)\bigr)^*$, 
\begin{align} \lb{NaK2}
\begin{split}
& \big(- \wti \Delta_{\Theta,\Om}-z \wti I_\Om\big)^{-1} 
= \big(- \wti \Delta_{D,\Om}-z \wti I_\Om\big)^{-1}\circ R_{\Om}   
\\ 
& \quad +\big[\wti\gamma_{\cN}
\big(\big(-\wti\Delta_{D,\Om}-\ol{z}\wti I_\Om\big)^{-1}\circ R_{\Om}, 
I_{\bbR^n}\big)\big]^*\wti M_{\Theta,D,\Om}^{(0)}(z)\big[\wti\gamma_{\cN}
\big(\big(-\wti\Delta_{D,\Om}-z\wti I_\Om\big)^{-1}\circ R_{\Om}, 
I_{\bbR^n}\big)\big].
\end{split}
\end{align}
\end{theorem}
\begin{proof}
Applying $\gamma_D$ from the left to both sides of \eqref{Na1B} yields
\begin{equation} \lb{NaK3}
\gamma_D\big(-\wti \Delta_{\Theta,\Om}-z \wti I_\Om\big)^{-1}
=\gamma_D\big(-\wti \Delta_{\Theta,\Om}-z \wti I_\Om\big)^{-1}
\gamma_D^*\wti\gamma_{\cN}
\big(\big(-\wti\Delta_{D,\Om}-z \wti I_\Om\big)^{-1}\circ R_{\Om}, I_{\bbR^n}\big)
\end{equation}
since $\gamma_D\big(-\wti\Delta_{D,\Om}-z \wti I_\Om\big)^{-1}=0$. Thus, by 
\eqref{3.52bis}, 
\begin{equation} \lb{NaK4}
\gamma_D\big(-\wti\Delta_{\Theta,\Om}-z \wti I_\Om\big)^{-1}=
\wti M_{\Theta,D,\Om}^{(0)}(z)\wti\gamma_{\cN}
\big(\big(-\wti\Delta_{D,\Om}-z \wti I_\Om\big)^{-1}\circ R_{\Om}, I_{\bbR^n}\big),
\end{equation}
as operators in $\cB\big(\big(H^1(\Om)\big)^*, H^{1/2}(\dOm)\big)$. 
Taking adjoints in \eqref{NaK4} (written with $\ol{z}$ in place of $z$) 
then leads to 
\begin{align} \lb{NaK5}
\big(-\wti\Delta_{\Theta,\Om}-z \wti I_\Om\big)^{-1}\gamma_D^* & =  
\big[\gamma_D\big(-\wti\Delta_{\Theta,\Om}-\ol{z} \wti I_\Om\big)^{-1}\big]^* 
\nonumber\\ 
& = \big[\wti\gamma_N\big(-\wti\Delta_{D,\Om}-\ol{z} \wti I_\Om\big)^{-1}\big]^* 
\big[\wti M_{\Theta,D,\Om}^{(0)}(\ol{z})\big]^*
\nonumber\\ 
& =\big[\wti\gamma_{\cN}
\big(\big(-\wti\Delta_{D,\Om}-\ol{z}\wti I_\Om\big)^{-1}\circ R_{\Om}, 
I_{\bbR^n}\big)\big]^*\wti M_{\Theta,D,\Om}^{(0)}(z),
\end{align}
by Lemma \ref{lA.3CC}. Replacing this back into \eqref{Na1B} then readily yields
\eqref{NaK2}.  
\end{proof}

The $\LOm$-analog of Theorem \ref{tA.3LL} then reads as follows:

\begin{theorem} \lb{t4.11}
Assume Hypothesis \ref{h3.1} and suppose that 
$z\in\bbC\backslash(\si(-\Delta_{\Theta,\Om})\cup\si(-\Delta_{D,\Om}))$.
Then the following Krein formula holds on $\LOm$: 
\begin{align} \lb{NaK6}
\begin{split}
(- \Delta_{\Theta,\Om}-zI_\Om)^{-1} &= (- \Delta_{D,\Om}-zI_\Om)^{-1}  \\ 
& \quad +\big[\wti\gamma_N(-\Delta_{D,\Om}-\ol{z}I_\Om)^{-1}\big]^* 
M_{\Theta,D,\Om}^{(0)}(z)
\big[\wti\gamma_N (- \Delta_{D,\Om}-zI_\Om)^{-1}\big].
\end{split}
\end{align}
\end{theorem}
\begin{proof}
This follows from Theorem \ref{tA.3LL} and the compatibility results 
established in Lemma \ref{new-L1}. 
\end{proof}

An attractive feature of the Krein-type formula \eqref{NaK6} lies in 
the fact that $M_{\Theta,D,\Om}^{(0)}(z)$ encodes spectral information about 
$\Delta_{\Theta,\Om}$. This will be pursued in future work. 

Assuming Hypothesis \ref{h2.1}, the special case $\Theta =0$ then 
connects the Neumann and Dirichlet resolvents, 
\begin{align} \lb{NDK2}
& \big(- \wti \Delta_{N,\Om}-z \wti I_\Om\big)^{-1} 
= \big(- \wti \Delta_{D,\Om}-z \wti I_\Om\big)^{-1}\circ R_{\Om}   
\no\\ 
& \quad +\big[\wti\gamma_{\cN}
\big(\big(-\wti\Delta_{D,\Om}-\ol{z}\wti I_\Om\big)^{-1}\circ R_{\Om}, 
I_{\bbR^n}\big)\big]^*\wti M_{N,D,\Om}^{(0)}(z)\big[\wti\gamma_{\cN}
\big(\big(-\wti\Delta_{D,\Om}-z\wti I_\Om\big)^{-1}\circ R_{\Om}, 
I_{\bbR^n}\big)\big], \\
& \hspace*{8.7cm}
z\in\bbC\backslash(\si(-\Delta_{N,\Om})\cup\si(-\Delta_{D,\Om})), \no
\end{align}
on $\bigl(H^1(\Omega)\bigr)^*$, and similarly,
\begin{align} \lb{NDK3}
(- \Delta_{N,\Om}-zI_\Om)^{-1} &= (- \Delta_{D,\Om}-z I_\Om)^{-1}  \no  \\ 
& \quad +\big[\wti\gamma_N(-\Delta_{D,\Om}-\ol{z} I_\Om)^{-1}\big]^* 
M_{N,D,\Om}^{(0)}(z)
\big[\wti\gamma_N (- \Delta_{D,\Om}-z I_\Om)^{-1}\big], \\ 
& \hspace*{4.65cm}
z\in\bbC\backslash(\si(-\Delta_{N,\Om})\cup\si(-\Delta_{D,\Om})), \no 
\end{align}
on $\LOm$. Here $\wti M_{N,D,\Om}^{(0)}(z)$ and $M_{N,D,\Om}^{(0)}(z)$ denote 
the corresponding Neumann-to-Dirichlet operators. 

\begin{remark}\lb{R-C1}
In the case when Hypothesis \ref{h2.8} is enforced, it can be shown that 
\begin{eqnarray}\label{MCv.1}
M_{N,D,\Om}^{(0)}(z)\in 
\cB\bigl(H^{1/2}(\partial\Omega),H^{3/2}(\partial\Omega)\bigr),
\quad
z\in\bbC\backslash(\si(-\Delta_{N,\Om})\cup\si(-\Delta_{D,\Om})),
\end{eqnarray}
and 
\begin{eqnarray}\label{MCv.2}
(-\Delta_{D,\Om}-z I_\Om)^{-1}\in\cB\bigl(\LOm,H^2(\Omega)\bigl),\quad
z\in\bbC\backslash\si(-\Delta_{D,\Om}).
\end{eqnarray}
Note that, by duality, the latter membership also entails 
\begin{eqnarray}\label{MCv.3}
(-\Delta_{D,\Om}-zI_\Om)^{-1}\in\cB\bigl((H^2(\Omega))^*,\LOm\bigl),\quad
z\in\bbC\backslash\si(-\Delta_{D,\Om}),
\end{eqnarray}
and, given \eqref{2.7}, 
\begin{eqnarray}\label{MCv.4}
\gamma_N(-\Delta_{D,\Om}-z I_\Om)^{-1}
\in\cB\bigl(\LOm,H^{1/2}(\partial\Omega)\bigl),\quad
z\in\bbC\backslash\si(-\Delta_{D,\Om}). 
\end{eqnarray}
Since, in the current scenario, we also have 
\begin{eqnarray}\label{MCv.5}
\gamma_N^*\in\cB\bigl(H^{-1/2}(\partial\Omega),(H^{2}(\Omega))^*\bigl),
\end{eqnarray}
it follows that \eqref{NDK3} takes the form 
\begin{align} \lb{NDK3.1}
(- \Delta_{N,\Om}-zI_\Om)^{-1} &= (- \Delta_{D,\Om}-z I_\Om)^{-1}
\nonumber \\ 
& \quad +(-\Delta_{D,\Om}-zI_\Om)^{-1}\gamma_N^* M_{N,D,\Om}^{(0)}(z)
\gamma_N (- \Delta_{D,\Om}-z I_\Om)^{-1}, \\[4pt] 
& \hspace*{3.9cm}
z\in\bbC\backslash(\si(-\Delta_{N,\Om})\cup\si(-\Delta_{D,\Om})), \no 
\end{align}
on $\LOm$, where the composition of the various operators involved is 
well-defined by the above discussion. Formula \eqref{NDK3.1} should be viewed as 
a variant of \eqref{NaK6} in which the Neumann trace operator can be decoupled from 
the two resolvents of $-\Delta_{D,\Om}$ in the second term on the right-hand side of \eqref{NaK6}.  
\end{remark}

Due to the fundamental importance of Krein-type resolvent formulas (and more generally, Robin-to-Dirichlet maps) in connection with the spectral and inverse spectral theory of ordinary and partial differential operators, abstract versions, connected to boundary value spaces (boundary triples) and self-adjoint extensions of closed symmetric operators with equal (possibly infinite) deficiency spaces, have received enormous attention in the literature. In particular, we note that Robin-to-Dirichlet maps in the context of ordinary differential operators reduce to the celebrated (possibly, matrix-valued) Weyl--Titchmarsh function, the basic object of spectral analysis in this context.  Since it is impossible to cover the literature in this paper, we refer, for instance, to 
\cite[Sect.\ 84]{AG93}, \cite{ADKK07}, \cite{AT03}, \cite{AT05}, \cite{BL07}, \cite{BMT01}, \cite{BT04}, 
\cite{BGW08}, \cite{BMNW08}, \cite{BGP07}, \cite{GMT98}, \cite{GM09}, \cite[Ch.\ 13]{Gr09}, 
\cite{KK02}, \cite{Ko00}--\cite{LT77}, \cite{Ma92}, \cite{MM06}, \cite{MPP07}, 
\cite{Ne83}--\cite{Po08}, \cite{Sa65}, \cite{St50}--\cite{St70a}, 
and the references cited therein. We add, however, that the case of infinite deficiency indices in the context of partial differential operators (in our concrete case, related to the deficiency indices of the operator closure of $-\Delta\upharpoonright_{C^\infty_0(\Om)}$ in $\LOm$), is much less studied and the results obtained in this section, especially, under the assumption of Lipschitz (i.e., minimally smooth) domains, to the best of our knowledge, are new.

Finally, we emphasize once more that Remark \ref{r3.6} also applies 
to the content of this section (assuming that $V$ is real-valued in connection with 
Lemmas \ref{lA.3CC} and \ref{l4.13}).

\appendix
\section{Properties of Sobolev Spaces and Boundary Traces \\ 
for $C^{1,r}$ and Lipschitz Domains} \lb{sA}
\renewcommand{\theequation}{A.\arabic{equation}}
\renewcommand{\thetheorem}{A.\arabic{theorem}}
\setcounter{theorem}{0} \setcounter{equation}{0}

The purpose of this appendix is to recall some basic facts in connection with 
Sobolev spaces corresponding to Lipschitz domains 
$\Om\subset\bbR^n$, $n\in\bbN$, $n\geq 2$, and on domains 
satisfying Hypothesis \ref{h2.8}.  

In this manuscript we use the following notation for the standard
Sobolev Hilbert spaces ($s\in\bbR$),
\begin{align}
H^{s}(\bbR^n) &=\bigg\{U\in \cS(\bbR^n)^\prime \,\bigg|\,
\norm{U}_{H^{s}(\bbR^n)}^2 = \int_{\bbR^n} d^n \xi \, \big|\hatt
U(\xi)\big|^2\big(1+\abs{\xi}^{2s}\big) <\infty \bigg\},
\\
H^{s}(\Om) &=\left\{u\in \cD^\prime(\Om) \,|\, u=U|_\Om \text{
for some } U\in H^{s}(\bbR^n) \right\},
\\
H_0^{s}(\Om) &=\{u\in H^s(\bbR^n)\,|\, \supp\,(u)\subseteq\ol{\Om}\}.
\end{align}
Here $\cD^\prime(\Om)$ denotes the usual set of distributions on
$\Omega\subseteq \bbR^n$, $\Omega$ open and nonempty,
$\cS(\bbR^n)^\prime$ is the space of tempered distributions on
$\bbR^n$, and $\hatt U$ denotes the Fourier transform of $U\in
\cS(\bbR^n)^\prime$. It is then immediate that
\begin{equation}\label{incl-xxx}
H^{s_1}(\Omega)\hookrightarrow H^{s_0}(\Omega) \, \text{ for } \,
-\infty<s_0\leq s_1<+\infty,
\end{equation}
continuously and densely.

Next, we recall the
definition of a $C^{1,r}$-domain $\Omega\subseteq\bbR^n$, $\Om$
open and nonempty, for convenience of the reader: Let ${\mathcal
N}$ be a space of real-valued functions in $\bbR^{n-1}$.  One
calls a bounded domain $\Omega\subset\bbR^n$ of class ${\mathcal
N}$ if there exists a finite open covering $\{{\mathcal
O}_j\}_{1\leq j\leq N}$ of the boundary $\partial\Omega$ of $\Om$
with the property that, for every $j\in\{1,...,N\}$, ${\mathcal
O}_j\cap\Omega$ coincides with the portion of ${\mathcal O}_j$
lying in the over-graph of a function $\varphi_j\in{\mathcal N}$
(considered in a new system of coordinates obtained from the
original one via a rigid motion). Two special cases are going to
play a particularly important role in the sequel. First, if
${\mathcal  N}$ is ${\rm Lip}\,(\bbR^{n-1})$, the space of
real-valued functions satisfying a (global) Lipschitz condition in
$\bbR^{n-1}$, we shall refer to $\Omega$ as being a Lipschitz
domain; cf.\ \cite[p.\ 189]{St70}, where such domains are called
``minimally smooth''. Second, corresponding to  the case when
${\mathcal N}$ is the subspace of ${\rm Lip}\,(\bbR^{n-1})$
consisting of functions whose first-order derivatives satisfy a
(global) H\"older condition of order $r\in(0,1)$, we shall say
that $\Omega$ is of class $C^{1,r}$. The classical theorem of
Rademacher of almost everywhere differentiability of Lipschitz
functions ensures that, for any  Lipschitz domain $\Omega$, the
surface measure $d^{n-1} \omega$ is well-defined on  $\partial\Omega$ and
that there exists an outward  pointing normal vector $\nu$ at
almost every point of $\partial\Omega$. 

Call a bounded, open set $\Omega\subset\bbR^n$ a star-like Lipschitz domain
with respect to a point $x^*$ (called center of star-likeness) 
if $\Omega$ is Lipschitz domain and 
\begin{eqnarray}\label{Star-C1}
x^*+t(x-x^*)\in\Omega\,\,\,\mbox{ for every $x\in\Omega$ and $t\in[0.1]$}.
\end{eqnarray}
The above geometrical characterization of Lipschitz domains can be used 
to show that, given a bounded Lipschitz domain $\Omega\subset\bbR^n$
then there exists a finite family of open sets $\Omega_j$, $1\leq j\leq N$, 
such that 
\begin{eqnarray}\label{Lip-S}
\Omega=\bigcup_{j=1}^N\Omega_j,\quad
\mbox{ $\Omega_j$ star-like Lipschitz domain, $1\leq j\leq N$}.
\end{eqnarray}

For a Lipschitz domain $\Omega\subset\bbR^n$ it is known that
\begin{equation}\lb{dual-xxx}
\bigl(H^{s}(\Omega)\bigr)^*=H^{-s}(\Omega), \quad - 1/2 <s< 1/2.
\end{equation}
See \cite{Tr02} for this and other related properties. We also refer to 
our convention of using the {\it adjoint} 
(rather than the dual) space $X^*$ of a Banach space $X$ as described near 
the end of the introduction.

Next, assume that $\Omega\subset\bbR^n$ is the domain lying above
the graph of a function $\varphi\colon\bbR^{n-1}\to\bbR$ of class
$C^{1,r}$. Then for $0\leq s<1+r$, the Sobolev space
$H^s(\partial\Omega)$ consists of functions $f\in
L^2(\partial\Omega;d^{n-1} \omega)$ such that $f(x',\varphi(x'))$,
as a function of $x'\in\bbR^{n-1}$, belongs to $H^s(\bbR^{n-1})$.
This definition is easily adapted to the case when $\Omega$ is a 
domain of class $C^{1,r}$ whose boundary is compact, by using a
smooth partition of unity. Finally, for $-1-r<s<0$, we set
$H^s(\partial\Omega)=\big(H^{-s}(\partial\Omega)\big)^*$. 
The same construction concerning $H^s(\partial\Omega)$ applies in the 
case when $\Om\subset\bbR^n$ is a Lipschitz domain 
(i.e., $\varphi\colon\bbR^{n-1}\to\bbR$ is only Lipschitz) provided 
$0\le s \le 1$. In this scenario we set 
\begin{equation}
H^s(\dOm) = \big(H^{-s}(\dOm)\big)^*, \quad -1 \le s \le 0.   \lb{A.6}
\end{equation}
It is useful to observe that this entails
\begin{equation}\label{Pk-D2}
\|f\|_{H^{-s}(\partial\Omega)}\approx
\|\sqrt{1+|\nabla\varphi(\dott)|^2}f(\dott,\varphi(\dott))\|
_{H^{-s}(\bbR^{n-1})},\quad 0\leq s\leq 1.
\end{equation}

To define $H^s(\dOm)$, $0\leq s \le 1$, when $\Om$ is a Lipschitz domain with 
compact boundary, we use a smooth partition of unity to reduce matters to the 
graph case. More precisely, if $0\leq s\leq 1$ then $f\in H^s(\partial\Omega)$ 
if and only if the assignment 
${\mathbb{R}}^{n-1}\ni x'\mapsto (\psi f)(x',\varphi(x'))$ is in 
$H^s({\mathbb{R}}^{n-1})$ whenever $\psi\in C^\infty_0({\mathbb{R}}^n)$
and $\varphi\colon {\mathbb{R}}^{n-1}\to{\mathbb{R}}$ is a Lipschitz function
with the property that if $\Sigma$ is an appropriate rotation and
translation of $\{(x',\varphi(x'))\in\bbR^n \,|\,x'\in{\mathbb{R}}^{n-1}\}$, 
then $(\supp\, (\psi) \cap\partial\Omega)\subset\Sigma$ (this appears to 
be folklore, but a proof will appear in \cite[Proposition 2.4]{MM07}). 
Then Sobolev spaces with a negative amount of smoothness are defined as 
in \eqref{A.6} above. 

From the above characterization of $H^s(\partial\Omega)$ it follows that 
any property of Sobolev spaces (of order $s\in[-1,1]$) defined in Euclidean 
domains, which are invariant under multiplication by smooth, compactly 
supported functions as well as composition by bi-Lipschitz diffeomorphisms, 
readily extends to the setting of $H^s(\partial\Omega)$ (via localization and
pull-back). As a concrete example, for each Lipschitz domain $\Omega$ 
with compact boundary, one has  
\begin{equation} \label{EQ1}
H^s(\partial\Omega)\hookrightarrow L^2(\partial\Omega;d^{n-1} \omega)
\, \text{ compactly if }\,0<s\leq 1.  
\end{equation}
For additional background 
information in this context we refer, for instance, to \cite{Au04}, 
\cite{Au06}, \cite[Chs.\ V, VI]{EE89}, \cite[Ch.\ 1]{Gr85}, 
\cite[Ch.\ 3]{Mc00}, \cite[Sect.\ I.4.2]{Wl87}.

For a Lipschitz domain $\Om\subset\bbR^n$ with compact boundary, an equivalent 
definition of the Sobolev space $H^1(\partial\Omega)$ is the collection of 
functions in $L^2(\partial\Omega;d^{n-1}\omega)$ with the property that the
(pointwise, Euclidean) norm of their tangential gradient belongs to 
$L^2(\partial\Omega;d^{n-1}\omega)$. To make this precise, 
consider the first-order tangential derivative operators 
$\partial/\partial\tau_{j,k}$, $1\leq j,k\leq n$, acting on a function 
$\psi$ of class $C^1$ in a neighborhood of $\partial\Omega$ by 
\begin{eqnarray}\label{def-TAU}
\partial\psi/\partial\tau_{j,k}=\nu_j(\partial_k\psi)\bigl|_{\partial\Omega}
-\nu_k(\partial_j\psi)\bigl|_{\partial\Omega}.
\end{eqnarray}
For every $f\in L^1(\partial\Omega)$ define the functional 
$\partial f/\partial\tau_{j,k}$ by setting 
\begin{eqnarray}\label{IBP-tau}
\partial f/\partial\tau_{j,k}:C^1({\mathbb{R}}^{n})\ni\psi\mapsto
\int_{\partial\Omega}d^{n-1}\omega\,f\,(\partial\psi/\partial\tau_{k,j})
\end{eqnarray}
When $f\in L^1(\partial\Omega;d^{n-1}\omega)$ has 
$\partial f/\partial\tau_{j,k}\in L^1(\partial\Omega;d^{n-1}\omega)$, 
the following integration by parts formula holds: 
\begin{eqnarray}\label{IBP-t2}
\int_{\partial\Omega} d^{n-1}\omega\,f\,(\partial\psi/\partial\tau_{k,j})
=\int_{\partial\Omega} d^{n-1}\omega\,(\partial f/\partial\tau_{j,k})\,\psi, 
\quad \psi\in C^1({\mathbb{R}}^{n}).
\end{eqnarray}
We then have the Sobolev-type description of $H^1(\partial\Omega)$: 
\begin{equation}\label{M1.1}
H^{1}(\dOm) = \big\{f\in\LdOm \,\big|\, \partial f/\partial\tau_{j,k}
\in\LdOm, \; j,k = 1,\dots,n\big\}
\end{equation}
with 
\begin{equation}\label{M1.1y}
\|f\|_{H^{1}(\dOm)}\approx\|f\|_{\LdOm}+\sum_{j,k = 1}^n
\|\partial f/\partial\tau_{j,k}\|_{\LdOm}, 
\end{equation}
($\approx$ denoting equivalent norms), or equivalently,
\begin{align}
H^{1}(\dOm) &= \bigg\{f\in\LdOm \,\bigg|\, \, 
\text{there exists a constant $c>0$
such that for every $v\in C_0^\infty(\bbR^n)$,}  \no \\
& \qquad \; \bigg|\int_\dOm d^{n-1} \omega f\,\partial
v/\partial\tau_{j,k}\bigg| \leq c \norm{v}_{\LdOm},\;
j,k=1,\dots,n\bigg\}. \lb{A.64}
\end{align}

Let us also point out here that if $\Omega\subset\bbR^n$ is a bounded
Lipschitz domain then for any $j,k\in\{1,...,n\}$ the operator 
\begin{equation}\label{Pf-2}
\partial/\partial\tau_{j,k}:H^s(\partial\Omega)\to H^{s-1}(\partial\Omega),
\quad 0\leq s\leq 1, 
\end{equation}
is well-defined, linear and bounded. This is proved by interpolating 
the case $s=1$ and its dual version. In fact, the following more general
result (extending \eqref{M1.1}) is true. 

\begin{lemma}\label{Lg-T}
Assume that $\Omega\subset\bbR^n$ is a bounded Lipschitz domain.
Then for every $s\in[0,1]$, 
\begin{equation}\label{Pf-2.3}
H^s(\partial\Omega)=\{f\in L^2(\partial\Omega;d^{n-1}\omega)\,|\,
\partial f/\partial\tau_{j,k}\in H^{s-1}(\partial\Omega), \, 
1\leq j,k\leq n\}
\end{equation}
and 
\begin{equation}\label{Pf-2.4}
\|f\|_{H^s(\partial\Omega)}\approx \|f\|_{L^2(\partial\Omega;d^{n-1}\omega)}
+\sum_{j,k=1}^n\|\partial f/\partial\tau_{j,k}\|_{H^{s-1}(\partial\Omega)}.
\end{equation}
\end{lemma}
\begin{proof}
The left-to-right inclusion in \eqref{Pf-2.3} along with the 
right-pointing inequality in \eqref{Pf-2.4} are consequences of the
boundedness of \eqref{Pf-2}. As for the opposite directions, we note that
using a smooth partition of unity and making a rigid transformation of the 
space, matters can be localized near a boundary point where $\partial\Omega$ 
coincides with the graph of a Lipschitz function 
$\varphi:\bbR^{n-1}\to\bbR$. Then for each sufficiently nice function 
$f:\partial\Omega\to\bbR$ the Chain Rule yields 
\begin{equation}\lb{M.3}
\Big(\frac{\partial f}{\partial\tau_{j,n}}\Bigr)(x,\varphi(x))
=\frac{1}{\sqrt{1+|\nabla\varphi(x)|^2}}\frac{\partial}{\partial x_j}
\bigl[f(x,\varphi(x))\bigr], \quad 1\leq j\leq n-1.
\end{equation}
On account of this and \eqref{Pk-D2}, we then deduce (upon noticing 
that $\partial/\partial\tau_{n,n}=0$) that 
\begin{eqnarray}\lb{M.3jx}
\sum_{j=1}^n\|\partial f/\partial\tau_{j,n}\|_{H^{s-1}(\partial\Omega)}
\approx \sum_{j=1}^{n-1}\|\partial_j[f(\dott,\varphi(\dott))]
\|_{H^{s-1}(\bbR^{n-1})}.
\end{eqnarray}
Furthermore, we also have 
\begin{eqnarray}\lb{M.4jx}
\|f\|_{L^2(\partial\Omega;d^{n-1}\omega)}
\approx\|f(\dott,\varphi(\dott))\|_{L^2(\bbR^{n-1};d^{n-1}x)}.
\end{eqnarray}
Next, we recall the general Euclidean lifting result 
\begin{equation}\lb{M.4}
H^{s}(\bbR^{n-1})=\{f\in H^{s-1}(\bbR^{n-1})\,|\,\partial_j f\in 
H^{s-1}(\bbR^{n-1}), \, 1\leq j\leq n-1\},\quad s\in\bbR,
\end{equation}
which can be found in \cite[Section 2.1.4]{RuSi}. 
Now, the right-to-left inclusion in \eqref{Pf-2.3}, as well as 
the left-pointing inequality in \eqref{Pf-2.4}, follow based on 
\eqref{M.3jx}, \eqref{M.4jx} and the estimate which naturally 
accompanies \eqref{M.4}.
\end{proof}

\begin{lemma}\label{Lhf-2}
Assume Hypothesis \ref{h2.1}. Then for every $s\in[0,1]$ and 
$j,k\in\{1,...,n\}$
\begin{eqnarray}\lb{Ibp-w}
\langle\partial f/\partial\tau_{j,k}\,,\,g\rangle_{1-s}
=\langle f\,,\,\partial g/\partial\tau_{k,j}\rangle_{1-s}
\end{eqnarray}
for every $f\in H^s(\partial\Omega)$ and $g\in H^{1-s}(\partial\Omega)$.
\end{lemma}
\begin{proof}
Since for every $s\in[0,1]$
\begin{eqnarray}\lb{Ibp-w2}
C^\infty(\bbR^n)\big|_{\partial\Omega}\hookrightarrow 
H^s(\partial\Omega)\, \mbox{ densely}, 
\end{eqnarray}
it suffices to prove \eqref{Ibp-w} in the case when 
$f=u|_{\partial\Omega}$ and $g=v|_{\partial\Omega}$ for 
$u,v\in C^\infty(\bbR^n)$. In this scenario, we need to establish that
\begin{eqnarray}\lb{Ibp-w3}
\int_{\partial\Omega}d^{n-1}\omega\,(\partial u/\partial\tau_{j,k})v
=\int_{\partial\Omega}d^{n-1}\omega\,u(\partial v/\partial\tau_{k,j}),\quad
1\leq j,k\leq n.
\end{eqnarray}
To this end, we rely on Green's formula (valid for Lipschitz domains) 
to write 
\begin{align}\lb{Ibp-w4}
\int_{\partial\Omega}d^{n-1}\omega\,(\partial u/\partial\tau_{j,k})v
&= \int_{\partial\Omega}d^{n-1}\omega\,(\nu_j\partial_ku-\nu_k\partial_ju)v
\nonumber\\
&=\int_{\Omega}d^nx\,[\partial_j(v\partial_ku)-\partial_k(v\partial_ju)]
\nonumber\\
&=\int_{\Omega}d^nx\,[(\partial_jv)(\partial_ku)-(\partial_kv)(\partial_ju)]. 
\end{align}
One observes that the right-most integrand above is an antisymmetric expression 
in the indices $j,k$. Consequently, so is the left-most integral 
in \eqref{Ibp-w4}. This, however, is equivalent to \eqref{Ibp-w3}. 
\end{proof}

Moving on, we next consider the following bounded linear map
\begin{equation}
\begin{cases} \big\{(w,f)\in
L^2(\Omega;d^nx)^n\times \big(H^1(\Omega)\big)^*
\,\big|\,{\rm div}(w)=f|_{\Omega}\big\} \to
H^{-1/2}(\partial\Omega)
=\big(H^{1/2}(\partial\Omega)\big)^* \\
\hspace*{7.5cm} w\mapsto \nu\cdot (w,f)  \end{cases}  \lb{A.11}
\end{equation}
by setting
\begin{equation}
{}_{H^{1/2}(\dOm)}\langle \phi, \nu\cdot (w,f) \rangle_{(H^{1/2}(\dOm)^*}
=\int_{\Omega}d^nx\, \ol{\nabla\Phi(x)} \cdot w(x)
+ {}_{H^1(\Om)}\langle \Phi,f\rangle_{(H^1(\Om))^*}    \lb{A.11a}
\end{equation}
whenever $\phi\in H^{1/2}(\partial\Omega)$ and $\Phi\in
H^{1}(\Omega)$ is such that $\ga_D\Phi=\phi$. Here  
${}_{H^1(\Om)}\langle \Phi,f\rangle_{(H^1(\Om))^*}$ 
in \eqref{A.11a} is the natural pairing  between functionals 
in $\big(H^1(\Omega)\big)^*$ and elements in $H^1(\Omega)$ 
(which, in turn, is compatible with the (bilinear) distributional 
pairing). It should be remarked that the above definition is independent 
of the particular extension $\Phi\in H^{1}(\Omega)$ of $\phi$. 

Going further, one can introduce the ultra weak Neumann trace operator
$\wti\gamma_{\cN}$ as follows:
\begin{equation}\lb{A.16}
\wti\gamma_{\cN}\colon \begin{cases} 
\big\{(u,f)\in H^1(\Om) \times \big(H^1(\Omega)\big)^*
\,\big|\, \Delta u =f|_{\Omega}\big\}\to H^{-1/2}(\dOm)\\ 
\hspace*{6.13cm}  u \mapsto \wti\gamma_{\cN}(u,f)=\nu\cdot(\nabla u,f),  
\end{cases}
\end{equation}
with the dot product  understood in the sense of \eqref{A.11}. We
emphasize that the ultra weak Neumann trace operator $\wti\gamma_{\cN}$ in
\eqref{A.16} is a re-normalization of the operator $\gamma_N$ introduced
in \eqref{2.7} relative to the extension of $\Delta u\in H^{-1}(\Omega)$
to an element $f$ of the space $\big(H^1(\Omega)\big)^*
=\{g\in H^{-1}(\bbR^n)\,|\, \supp\, (g) \subseteq\ol{\Om}\}$. 
For the relationship between the weak and ultra weak Neumann trace operators, 
see \eqref{2.10X}--\eqref{2.12X}. In addition, one can show that 
the ultra weak Neumann trace operator \eqref{A.16} is onto (indeed, this is 
a corollary of Theorem \ref{t3.XV}). We note that \eqref{A.11a} and 
\eqref{A.16} yield the following Green's formula
\begin{equation}
\langle \ga_D\Phi, \wti\gamma_{\cN}(u,f)\rangle_{1/2} 
= (\nabla \Phi, \nabla u)_{\LOm^n} 
+  {}_{H^1(\Om)}\langle \Phi, f\rangle_{(H^1(\Om))^*},
\lb{wGreen}
\end{equation}
valid for any $u\in H^{1}(\Om)$, $f\in \big(H^{1}(\Om)\big)^*$
with $\Delta u=f|_{\Om}$, and any $\Phi\in H^{1}(\Om)$. The
pairing on the left-hand side of \eqref{wGreen} is between
functionals in $\big(H^{1/2}(\dOm)\big)^*$ and elements in
$H^{1/2}(\dOm)$, whereas the last pairing on the right-hand side
is between functionals in $\big(H^{1}(\Om)\big)^*$ and elements in
$H^{1}(\Om)$. For further use, we also note that the adjoint of
\eqref{2.6} maps boundedly as follows
\begin{equation}\lb{ga*}
\ga_D^* \colon \big(H^{s-1/2}(\dOm)\big)^* \to (H^{s}(\Om)\big)^*,
\quad 1/2<s<3/2. 
\end{equation} 

\begin{remark} \lb{rA.4}
While it is tempting to view $\ga_D$ as an unbounded but densely
defined operator on $\LOm$ whose domain contains the space
$C_0^\infty(\Om)$, one should note that in this case its adjoint
$\ga_D^*$ is not densely defined: Indeed (cf.\ \cite[Remark A.4]{GLMZ05}),
$\dom(\gamma_D^*)=\{0\}$ and hence $\gamma_D$ is not a closable linear 
operator in $\LOm$.
\end{remark}

Next we recall the following result from \cite{GMZ07} (and reproduce 
its proof for subsequent use in the proofs of Lemmas \ref{lA.6} 
and \ref{L-Kb}).  

\begin{lemma} [cf.\ \cite{GMZ07}, Lemma A.6] \lb{lA.6x}
Suppose $\Om\subset\bbR^n$, $n\geq 2$, is an open Lipschitz domain with 
a compact, nonempty boundary $\dOm$. Then the Dirichlet trace operator 
$\ga_D$ $($originally considered as in \eqref{2.6}$)$ satisfies 
\eqref{A.62x}. 
\end{lemma}
\begin{proof}
Let $u\in H^{(3/2)+\eps}(\Om)$, $v\in C_0^\infty(\bbR^n)$,
and $u_\ell \in C^\infty(\ol{\Om})\hookrightarrow
H^{(3/2)+\eps}(\Om)$, $\ell\in\bbN$, be a sequence of functions
approximating $u$ in $H^{(3/2)+\eps}(\Om)$. It follows from
\eqref{2.6} and \eqref{incl-xxx} that $\ga_D u, \ga_D(\nabla u)\in\LdOm$. 
Utilizing \eqref{IBP-t2}, one computes for all $j,k=1,\dots,n$,
\begin{align}
\bigg|\int_\dOm d^{n-1} \omega\, \ga_D u \frac{\partial
v}{\partial\tau_{j,k}} \bigg| &= \bigg|\lim_{\ell\to\infty} \int_\dOm
d^{n-1} \omega\, u_\ell \frac{\partial v}{\partial\tau_{j,k}}\bigg| =
\bigg|\lim_{\ell\to\infty} \int_\dOm d^{n-1} \omega\, v \frac{\partial
u_\ell}{\partial\tau_{j,k}}\bigg| \lb{A.65}
\\ &\leq
c \, \bigg|\lim_{\ell\to\infty} \int_\dOm d^{n-1} \omega\,v\, \ga_D(\nabla
u_\ell)\bigg| \leq c\norm{\ga_D(\nabla u)}_{\LdOm} \norm{v}_{\LdOm}.
\no
\end{align}
Thus, it follows from \eqref{A.64} and \eqref{A.65} that $\ga_D
u \in H^1(\dOm)$.
\end{proof}

Next, we prove the following fact:

\begin{lemma} \lb{lA.6}
Suppose $\Om\subset\bbR^n$, $n\geq 2$, is a bounded Lipschitz domain. 
Then for each $r\in(1/2,1)$, the space $C^r(\partial\Omega)$ is a module 
over $H^{1/2}(\partial\Omega)$. More precisely, if $M_f$ denotes the 
operator of multiplication by $f$, then there exists $C=C(\Omega,r)>0$ 
such that
\begin{equation}\lb{M.x1}
M_f\in{\mathcal{B}}\bigl(H^{1/2}(\partial\Omega)\bigl)\mbox{ and } \, 
\|M_f\|_{{\mathcal{B}}\bigl(H^{1/2}(\partial\Omega)\bigl)}
\leq C\|f\|_{C^r(\partial\Omega)}\mbox{ for every }f\in C^r(\partial\Omega).
\end{equation}
As a consequence, if $\Om$ is actually a bounded $C^{1,r}$-domain
with $r\in(1/2,1)$, then the Neumann and Dirichlet trace operators 
$\ga_N$, $\ga_D$ satisfy
\begin{equation}\lb{A.62}
\ga_N\in \cB\big(H^{2}(\Om), H^{1/2}(\dOm)\big)
\end{equation}
and 
\begin{equation}\lb{A.62B}
\ga_D\in \cB\big(H^{2}(\Om), H^{3/2}(\dOm)\big). 
\end{equation}
\end{lemma}
\begin{proof}
The first part of the lemma is a direct consequence of general results 
about pointwise multiplication of functions in Triebel--Lizorkin spaces 
(a scale which contains both H\"older and Sobolev spaces);
see \cite[Theorem 2 on p.\,177]{RuSi}. Then \eqref{A.62} follows 
from this, \eqref{2.6}, the fact that $\gamma_N=\nu\cdot\gamma_D$, and
$\nu\in C^r(\partial\Omega)$. Next, one observes that for each 
$u\in H^2(\Omega)$ one has  $\gamma_Du\in H^{1}(\partial\Omega)$ 
by Lemma \ref{lA.6x}. In addition, 
\begin{equation}\lb{M.1}
\frac{\partial}{\partial\tau_{j,k}}(\gamma_Du)
=\big(\nu_j\gamma_D(\partial_k u)-\nu_k\gamma_D(\partial_j u)\big) 
\in H^{1/2}(\partial\Omega),
\end{equation} 
with a naturally accompanying estimate, by \eqref{2.6} and the fact that, 
as observed in the first part of the current proof, multiplication by 
$\nu_j$ (for $1\leq j\leq n$) preserves $H^{1/2}(\partial\Omega)$.
Consequently, \eqref{A.62B} follows from this and \eqref{M.2}, \eqref{M.2x}
below.
\end{proof}

Our next result should be compared with \eqref{M1.1} and Lemma \ref{Lg-T}.

\begin{lemma}\label{K-t1}
If $\Om\subset\bbR^n$ is a bounded $C^{1,r}$-domain with $r\in(1/2,1)$ then 
\begin{equation}\lb{M.2}
H^{3/2}(\dOm)=\{f\in H^1(\dOm)\,|\,\partial f/\partial\tau_{j,k}\in 
H^{1/2}(\dOm)\, 1\leq j,k\leq n\}
\end{equation}
and
\begin{equation}\lb{M.2x}
\|f\|_{H^{3/2}(\dOm)}\approx\|f\|_{H^1(\dOm)}
+\sum_{j,k=1}^n\|\partial f/\partial\tau_{j,k}\|_{H^{1/2}(\dOm)}. 
\end{equation}
\end{lemma}
\begin{proof}
To justify \eqref{M.2} and \eqref{M.2x} we use
a smooth cut-off function to localize the problem near a boundary point where 
$\partial\Omega$ coincides with the graph of a $C^{1,r}$ function 
$\varphi:\bbR^{n-1}\to\bbR$. In this setting, the desired conclusions 
follow from \eqref{M.3}, \eqref{M.x1}, and \eqref{M.4} used with $s=3/2$. 
\end{proof}

\section{Sesquilinear Forms and Associated Operators}
\lb{sB}
\renewcommand{\theequation}{B.\arabic{equation}}
\renewcommand{\thetheorem}{B.\arabic{theorem}}
\setcounter{theorem}{0} \setcounter{equation}{0}

In this appendix we describe a few basic facts on sesquilinear forms and 
linear operators associated with them.

Let $\cH$ be a complex separable Hilbert space with scalar product 
$(\dott,\dott)_{\cH}$ (antilinear in the first and linear in the second 
argument), $\cV$ a reflexive Banach space continuously and densely embedded 
into $\cH$. Then also $\cH$ embeds continuously and densely into $\cV^*$. 
\begin{equation}
\cV  \hookrightarrow \cH  \hookrightarrow \cV^*.     \lb{B.1}
\end{equation}
Here the continuous embedding $\cH\hookrightarrow \cV^*$ is accomplished via 
the identification
\begin{equation}
\cH \ni u \mapsto (\dott,u)_{\cH} \in \cV^*,     \lb{B.2}
\end{equation}
and we recall the convention in this manuscript (cf.\ the discussion at 
the end of the introduction) that if $X$ denotes a Banach space, $X^*$ denotes 
the {\it adjoint space} of continuous  conjugate linear functionals on $X$, also 
known as the {\it conjugate dual} of $X$.

In particular, if the sesquilinear form 
\begin{equation}
{}_{\cV}\langle \dott, \dott \rangle_{\cV^*} \colon \cV \times \cV^* \to \bbC
\end{equation}
denotes the duality pairing between $\cV$ and $\cV^*$, then  
\begin{equation}
{}_{\cV}\langle u,v\rangle_{\cV^*} = (u,v)_{\cH}, \quad u\in\cV, \; 
v\in\cH\hookrightarrow\cV^*,   \lb{B.3}
\end{equation}
that is, the $\cV, \cV^*$ pairing 
${}_{\cV}\langle \dott,\dott \rangle_{\cV^*}$ is compatible with the 
scalar product $(\dott,\dott)_{\cH}$ in $\cH$.

Let $T \in\cB(\cV,\cV^*)$. Since $\cV$ is reflexive, $(\cV^*)^* = \cV$, one has
\begin{equation}
T \colon \cV \to \cV^*, \quad  T^* \colon \cV \to \cV^*   \lb{B.4}
\end{equation}
and
\begin{equation}
{}_{\cV}\langle u, Tv \rangle_{\cV^*} 
= {}_{\cV^*}\langle T^* u, v\rangle_{(\cV^*)^*} 
= {}_{\cV^*}\langle T^* u, v \rangle_{\cV} 
= \ol{{}_{\cV}\langle v, T^* u \rangle_{\cV^*}}. 
\end{equation}
{\it Self-adjointness} of $T$ is then defined by $T=T^*$, that is, 
\begin{equation}
{}_{\cV}\langle u,T v \rangle_{\cV^*} 
= {}_{\cV^*}\langle T u, v \rangle_{\cV}
= \ol{{}_{\cV}\langle v, T u \rangle_{\cV^*}}, \quad u, v \in \cV,    \lb{B.5}
\end{equation} 
{\it nonnegativity} of $T$ is defined by 
\begin{equation}
{}_{\cV}\langle u, T u \rangle_{\cV^*} \geq 0, \quad u \in \cV,    \lb{B.6}
\end{equation} 
and {\it boundedness from below of $T$ by $c_T \in\bbR$} is defined by
\begin{equation}
{}_{\cV}\langle u, T u \rangle_{\cV^*} \geq c_T \|u\|^2_{\cH}, 
\quad u \in \cV.   
\lb{B.6a}
\end{equation} 
(By \eqref{B.3}, this is equivalent to 
${}_{\cV}\langle u, T u \rangle_{\cV^*} \geq c_T \,
{}_{\cV}\langle u, u \rangle_{\cV^*}$, $u \in \cV$.)
 
Next, let the sesquilinear form $a(\dott,\dott)\colon\cV \times \cV \to \bbC$ 
(antilinear in the first and linear in the second argument) be 
{\it $\cV$-bounded}, that is, there exists a $c_a>0$ such that 
\begin{equation}
|a(u,v)| \le c_a \|u\|_{\cV} \|v\|_{\cV},  \quad u, v \in \cV. 
\end{equation}
Then $\wti A$ defined by
\begin{equation}
\wti A \colon \begin{cases} \cV \to \cV^*, \\ 
\, v \mapsto \wti A v = a(\dott,v), \end{cases}    \lb{B.7}
\end{equation} 
satisfies
\begin{equation}
\wti A \in\cB(\cV,\cV^*) \, \text{ and } \, 
{}_{\cV}\big\langle u, \wti A v \big\rangle_{\cV^*} 
= a(u,v), \quad  u, v \in \cV.    \lb{B.8}
\end{equation}
Assuming further that $a(\dott,\dott)$ is {\it symmetric}, that is,
\begin{equation}
a(u,v) = \ol{a(v,u)},  \quad u,v\in \cV,    \lb{B.9}
\end{equation} 
and that $a$ is {\it $\cV$-coercive}, that is, there exists a constant 
$C_0>0$ such that
\begin{equation}
a(u,u)  \geq C_0 \|u\|^2_{\cV}, \quad u\in\cV,    \lb{B.10}
\end{equation}
respectively, then,
\begin{equation}
\wti A \colon \cV \to \cV^* \, \text{ is bounded, self-adjoint, and boundedly 
invertible.}    \lb{B.11}
\end{equation}
Moreover, denoting by $A$ the part of $\wti A$ in $\cH$ defined by 
\begin{align}
\dom(A) = \big\{u\in\cV \,|\, \wti A u \in \cH \big\} \subseteq \cH, \quad 
A= \wti A\big|_{\dom(A)}\colon \dom(A) \to \cH,   \lb{B.12} 
\end{align}
then $A$ is a (possibly unbounded) self-adjoint operator in $\cH$ satisfying 
\begin{align}
& A \geq C_0 I_{\cH},   \lb{B.13}  \\
& \dom\big(A^{1/2}\big) = \cV.  \lb{B.14}
\end{align}
In particular, 
\begin{equation}
A^{-1} \in\cB(\cH).   \lb{B.15}
\end{equation}
The facts \eqref{B.1}--\eqref{B.15} are a consequence of the Lax--Milgram 
theorem and the second representation theorem for symmetric sesquilinear forms.
Details can be found, for instance, in \cite[\S VI.3, \S VII.1]{DL00}, 
\cite[Ch.\ IV]{EE89}, and \cite{Li62}.  

Next, consider a symmetric form $b(\dott,\dott)\colon \cV\times\cV\to\bbC$ 
and assume that $b$ is {\it bounded from below by $c_b\in\bbR$}, that is,
\begin{equation}
b(u,u) \geq c_b \|u\|_{\cH}^2, \quad u\in\cV.  \lb{B.19}
\end{equation} 
Introducing the scalar product 
$(\dott,\dott)_{\cV(b)}\colon \cV\times\cV\to\bbC$ 
(with  associated norm $\|\dott\|_{\cV(b)}$) by
\begin{equation}
(u,v)_{\cV(b)} = b(u,v) + (1- c_b)(u,v)_{\cH}, \quad u,v\in\cV,  \lb{B.20}
\end{equation}
turns $\cV$ into a pre-Hilbert space $(\cV; (\dott,\dott)_{\cV(b)})$, 
which we denote by 
$\cV(b)$. The form $b$ is called {\it closed} if $\cV(b)$ is actually 
complete, and hence a Hilbert space. The form $b$ is called {\it closable} 
if it has a closed extension. If $b$ is closed, then
\begin{equation}
|b(u,v) + (1- c_b)(u,v)_{\cH}| \le \|u\|_{\cV(b)} \|v\|_{\cV(b)}, 
\quad u,v\in \cV,   
\lb{B.21}
\end{equation}
and 
\begin{equation}
|b(u,u) + (1 - c_b)\|u\|_{\cH}^2| = \|u\|_{\cV(b)}^2, \quad u \in \cV,   
\lb{B.22}
\end{equation}
show that the form $b(\dott,\dott)+(1 - c_b)(\dott,\dott)_{\cH}$ is a 
symmetric, $\cV$-bounded, and $\cV$-coercive sesquilinear form. Hence, 
by \eqref{B.7} and \eqref{B.8}, there exists a linear map
\begin{equation}
\wti B_{c_b} \colon \begin{cases} \cV(b) \to \cV(b)^*, \\ 
\hspace*{.51cm}  
v \mapsto \wti B_{c_b} v = b(\dott,v) +(1 - c_b)(\dott,v)_{\cH},  
\end{cases}    
\lb{B.23}
\end{equation}
with 
\begin{equation}
\wti B_{c_b} \in\cB(\cV(b),\cV(b)^*) \, \text{ and } \, 
{}_{\cV(b)}\big\langle u, \wti B_{c_b} v \big\rangle_{\cV(b)^*} 
= b(u,v)+(1 -c_b)(u,v)_{\cH}, \quad  u, v \in \cV.    \lb{B.24}
\end{equation}
Introducing the linear map
\begin{equation}
\wti B = \wti B_{c_b} + (c_b - 1)\wti I \colon \cV(b)\to\cV(b)^*,
\lb{B.24a}
\end{equation}
where $\wti I\colon \cV(b)\hookrightarrow\cV(b)^*$ denotes 
the continuous inclusion (embedding) map of $\cV(b)$ into $\cV(b)^*$, one 
obtains a self-adjoint operator $B$ in $\cH$ by restricting $\wti B$ to $\cH$,
\begin{align}
\dom(B) = \big\{u\in\cV \,\big|\, \wti B u \in \cH \big\} \subseteq \cH, \quad 
B= \wti B\big|_{\dom(B)}\colon \dom(B) \to \cH,   \lb{B.25} 
\end{align} 
satisfying the following properties:
\begin{align}
& B \geq c_b I_{\cH},  \lb{B.26} \\
& \dom\big(|B|^{1/2}\big) = \dom\big((B - c_bI_{\cH})^{1/2}\big) 
= \cV,  \lb{B.27} \\
& b(u,v) = \big(|B|^{1/2}u, U_B |B|^{1/2}v\big)_{\cH}    \lb{B.28b} \\
& \hspace*{.97cm} 
= \big((B - c_bI_{\cH})^{1/2}u, (B - c_bI_{\cH})^{1/2}v\big)_{\cH} 
+ c_b (u, v)_{\cH}  
\lb{B.28} \\
& \hspace*{.97cm} 
= {}_{\cV(b)}\big\langle u, \wti B v \big\rangle_{\cV(b)^*},  
\quad u, v \in \cV, \lb{B.28a} \\
& b(u,v) = (u, Bv)_{\cH}, \quad  u\in \cV, \; v \in\dom(B),  \lb{B.29} \\
& \dom(B) = \{v\in\cV\,|\, \text{there exists an $f_v\in\cH$ such that}  \no \\
& \hspace*{3.05cm} b(w,v)=(w,f_v)_{\cH} \text{ for all $w\in\cV$}\},  
\lb{B.30} \\
& Bu = f_u, \quad u\in\dom(B),  \no \\
& \dom(B) \text{ is dense in $\cH$ and in $\cV(b)$}.  \lb{B.31}
\end{align}
Properties \eqref{B.30} and \eqref{B.31} uniquely determine $B$. 
Here $U_B$ in \eqref{B.28} is the partial isometry in the polar 
decomposition of $B$, that is, 
\begin{equation}
B=U_B |B|, \quad  |B|=(B^*B)^{1/2}.   \lb{B.32}
\end{equation}
The operator $B$ is called the {\it operator associated with the form $b$}.

The norm in the Hilbert space $\cV(b)^*$ is given by
\begin{equation}
\|\ell\|_{\cV(b)^*} 
= \sup \{|{}_{\cV(b)}\langle u, \ell \rangle_{\cV(b)^*}| \,|\, 
\|u\|_{\cV(b)} \le 1\}, \quad \ell \in \cV(b)^*,   \lb{B.34}
\end{equation}
with associated scalar product,
\begin{equation}
(\ell_1,\ell_2)_{\cV(b)^*} 
= {}_{\cV(b)}\big\langle 
\big(\wti B+(1-c_b)\wti I\big)^{-1}
\ell_1, \ell_2 \big\rangle_{\cV(b)^*}, 
\quad \ell_1, \ell_2 \in \cV(b)^*.   \lb{B.35}
\end{equation}
Since
\begin{equation}
\big\|\big(\wti B + (1 - c_b)\wti I\big)v \big\|_{\cV(b)^*} 
= \|v\|_{\cV(b)}, \quad v\in\cV,   \lb{B.36}
\end{equation}
the Riesz representation theorem yields 
\begin{equation} 
\big(\wti B+(1-c_b)\wti I\big)\in\cB(\cV(b),\cV(b)^*) \, 
\text{ and }\big(\wti B + (1 - c_b)\wti I\big) \colon \cV(b) 
\to \cV(b)^* \, \text{ is unitary.}   \lb{B.37}
\end{equation}
In addition,
\begin{align}
\begin{split}
{}_{\cV(b)}\big\langle u,\big(\wti B 
+ (1 - c_b)\wti I\big) v \big\rangle_{\cV(b)^*} 
& = \big(\big(B+(1-c_b)I_{\cH}\big)^{1/2}u, 
\big(B+(1-c_b)I_{\cH}\big)^{1/2}v \big)_{\cH}  \\
& = (u,v)_{\cV(b)},  \quad  u, v \in \cV(b).   \lb{B.38}
\end{split}
\end{align}
In particular, 
\begin{equation}
\big\|(B+(1-c_b)I_{\cH})^{1/2}u\big\|_{\cH} = \|u\|_{\cV(b)}, 
\quad u \in \cV(b),   \lb{B.39}
\end{equation}
and hence 
\begin{equation}
(B+(1-c_b)I_{\cH})^{1/2}\in\cB(\cV(b),\cH) \, \text{ and } 
(B + (1 - c_b)I_{\cH})^{1/2} \colon \cV(b) \to \cH \, \text{ is unitary.}   
\lb{B.40}
\end{equation} 
The facts \eqref{B.19}--\eqref{B.40} comprise the second representation 
theorem of sesquilinear forms (cf.\ \cite[Sect.\ IV.2]{EE89}, 
\cite[Sects.\ 1.2--1.5]{Fa75}, and \cite[Sect.\ VI.2.6]{Ka80}).

A special but important case of nonnegative closed forms is obtained as 
follows: Let $\cH_j$, $j=1,2$, be complex separable Hilbert spaces, and 
$T\colon \dom(T)\to\cH_2$, $\dom(T)\subseteq \cH_1$, a densely defined 
operator. Consider the nonnegative form 
$a_T\colon \dom(T)\times \dom(T)\to\bbC$ defined by 
\begin{equation}
a_T(u,v)=(Tu,Tv)_{\cH_2}, \quad u, v \in\dom(T).   \lb{B.42}
\end{equation}
Then the form $a_T$ is closed (resp., closable) if and only if $T$ is. 
If $T$ is closed, the unique nonnegative self-adjoint operator associated 
with $a_T$ in $\cH_1$, whose existence is guaranteed by the second 
representation theorem for forms, then equals $T^*T$. In particular, 
one obtains
\begin{equation}
a_T(u,v) = (|T|u,|T|v)_{\cH_1}, \quad u, v \in\dom(T)=\dom(|T|).   \lb{B.43}
\end{equation} 
In addition, since
\begin{align}
\begin{split}
& b(u,v) +(1-c_b)(u,v)_{\cH} 
= \big((B+(1-c_b)I_{\cH})^{1/2}u, (B+(1-c_b)I_{\cH})^{1/2}v\big)_{\cH}, \\ 
& \hspace*{6.1cm}  u, v \in \dom(b) = \dom\big(|B|^{1/2}\big)=\cV,   \lb{B.43a}
\end{split}
\end{align} 
and $(B+(1-c_b)I_{\cH})^{1/2}$ is self-adjoint (and hence closed) in 
$\cH$, a symmetric, $\cV$-bounded, and $\cV$-coercive form is densely 
defined in $\cH\times\cH$ and closed (a fact we used in the proof of 
Theorem \ref{t2.3}). We refer to \cite[Sect.\ VI.2.4]{Ka80} and 
\cite[Sect.\ 5.5]{We80} for details.

Next we recall that if $a_j$ are sesquilinear forms defined on 
$\dom(a_j)\times\dom(a_j)$, $j=1,2$, bounded from below and closed, then also
\begin{equation}
(a_1+a_2)\colon \begin{cases} (\dom(a_1)\cap\dom(a_2))\times 
(\dom(a_1)\cap\dom(a_2)) \to \bbC, \\
(u,v) \mapsto (a_1+a_2)(u,v) = a_1(u,v) + a_2(u,v) \end{cases}   \lb{B.44}
\end{equation} 
is bounded from below and closed (cf.\ \cite[Sect.\ VI.1.6]{Ka80}). 

Finally, we also recall the following perturbation theoretic fact: 
Suppose $a$ is a sesquilinear form defined on $\cV\times\cV$, bounded 
from below and closed, and let $b$ be a symmetric sesquilinear form 
bounded with respect to $a$ with bound less than one, that is, 
$\dom(b)\supseteq \cV\times\cV$, and that there exist $0\le \alpha < 1$ and 
$\beta\ge 0$ such that
\begin{equation}
|b(u,u)| \le \alpha |a(u,u)| + \beta \|u\|_{\cH}^2, \quad u\in \cV.   \lb{B.45}
\end{equation}
Then
\begin{equation}
(a+b)\colon \begin{cases} \cV\times\cV \to \bbC, \\
\hspace*{.12cm} (u,v) \mapsto (a+b)(u,v) = a(u,v) + b(u,v) \end{cases}  
\lb{B.46}
\end{equation} 
defines a sesquilinear form that is bounded from below and closed 
(cf. \cite[Sect.\,VI.1.6]{Ka80}). In the special case where $\alpha$ 
can be chosen arbitrarily small, the form $b$ is called infinitesimally 
form bounded with respect to $a$.

\section{Estimates for the Fundamental Solution of the Helmholtz Equation}
\lb{sC}
\renewcommand{\theequation}{C.\arabic{equation}}
\renewcommand{\thetheorem}{C.\arabic{theorem}}
\setcounter{theorem}{0} \setcounter{equation}{0}

The principal aim of this appendix is to recall and prove some estimates 
for the fundamental solution (i.e., the Green's function) of the Helmholtz 
equation and its $x$-derivatives up to the second order.

Let $E_n(z;x)$ be the fundamental solution of the Helmholtz equation 
$(-\Delta -z)\psi(z;\dott) =0$ in $\bbR^n$, $n\in\bbN$, $n\ge 2$, already 
introduced in \eqref{2.52}, and reproduced for convenience below: 
\begin{align}
& E_n(z;x) = \begin{cases} \f{i}{4} \Big(\f{2\pi |x|}{z^{1/2}}\Big)^{(2-n)/2} 
H^{(1)}_{(n-2)/2}\big(z^{1/2}|x|\big), & n\ge 2, \; z\in\bbC\backslash\{0\}, \\
\f{-1}{2\pi} \ln(|x|), & n=2, \; z=0, \\
\f{1}{(n-2) \omega_{n-1}} |x|^{2-n}, & n \ge 3, \; z=0, 
\end{cases}    \lb{C.1}  \\
& \hspace*{5.5cm}   \Im\big(z^{1/2}\big)\geq 0, \; x \in\bbR^n\backslash\{0\}, 
\no
\end{align}
where $H^{(1)}_{\nu}(\dott)$ denotes the Hankel function of the first kind 
with index $\nu\geq 0$ (cf.\ \cite[Sect.\ 9.1]{AS72}) and 
$\omega_{n-1}=2\pi^{n/2}/\Gamma(n/2)$ ($\Gamma(\dott)$ the Gamma function, 
cf.\ \cite[Sect.\ 6.1]{AS72}) represents the area of the unit sphere 
$S^{n-1}$ in $\bbR^n$. 

As $z\to 0$, $E_n(z,x)$, $x\in\bbR^n\backslash\{0\}$ is continuous for 
$n\geq 3$,
\begin{align}
E_n(z,x) & \underset{z\to 0}{=} E_n(0,x) 
= \f{1}{(n-2) \omega_{n-1}} |x|^{2-n}, 
\quad x\in\bbR^n\backslash\{0\}, \; n \ge 3,    \lb{C.1a} \\
\intertext{but discontinuous for $n=2$ as}
E_2(z,x) & \underset{z\to 0}{=}  \f{-1}{2\pi} \ln\big(z^{1/2}|x|/2\big)
\big[1+\Oh\big(z|x|^2\big)\big] 
+ \f{1}{2\pi} \psi(1) +\Oh\big(z|x|^2\big), 
\quad  x\in\bbR^2\backslash\{0\}, \; n=2.   \lb{C.1b}
\end{align} 
Here $\psi(w)=\Gamma'(w)/\Gamma(w)$ denotes the digamma function 
(cf. \cite[Sect.\,6.3]{AS72}). Thus, we simply {\it define} 
$E_2(0;x)=\f{-1}{2\pi} \ln(|x|)$, $x\in\bbR^2\backslash\{0\}$ as 
in \eqref{C.1}. 

To estimate $E_n$ we recall that (cf. \cite[Sect.\,9.1]{AS72})
\begin{equation}
H^{(1)}_{(n-2)/2}(\dott) = J_{(n-2)/2}(\dott) + i Y_{(n-2)/2}(\dott), \lb{C.2}
\end{equation}
with $J_\nu$ and $Y_\nu$ the regular and irregular Bessel functions, 
respectively.

We start considering small values of $|x|$ and for this purpose recall 
the following absolutely convergent expansions (cf. \cite[Sect.\,9.1]{AS72}):
\begin{align}
J_\nu(\zeta) &= \bigg(\f{\zeta}{2}\bigg)^{\nu} \sum_{k=0}^\infty 
\f{(-1)^k \zeta^{2k}}{4^k k! \Gamma(\nu+k+1)}, 
\quad \zeta\in\bbC\backslash (-\infty,0], 
\; \nu\in\bbR\backslash (-\bbN),  \lb{C.3} \\
J_{-m}(\zeta) &= (-1)^m J_m(\zeta), \quad \zeta\in\bbC,\, m\in\bbN_0,  
\lb{C.4} \\
Y_\nu (\zeta) &= \f{J_\nu(\zeta)\cos(\nu\pi)- J_{-\nu}(\zeta)}{\sin(\nu \pi)}, 
\quad \zeta\in\bbC\backslash (-\infty,0], \; \nu\in(0,\infty)\backslash\bbN,  
\lb{C.5} \\
Y_m(\zeta) &= - \f{\zeta^{-m}}{2^m \pi} \sum_{k=0}^{m-1} \f{(m-k-1)!}{k!}
\f{\zeta^{2k}}{4^k} + \f{2}{\pi} J_m(\zeta) \, \ln(\zeta/2)  \no \\
& \quad - \f{\zeta^m}{2^m \pi} \sum_{k=0}^\infty [\psi(k+1) + \psi(m+k=1)] 
\f{(-1)^k \zeta^{2k}}{4^k k! (m+k)!},  
\quad \zeta\in\bbC\backslash (-\infty,0], \; 
m\in\bbN_0.   \lb{C.6}
\end{align}
We note that all functions in \eqref{C.3}, \eqref{C.5}, and 
\eqref{C.6} are analytic in $\bbC\backslash (-\infty,0]$ and that 
$J_m(\dott)$ is entire for $m\in\bbZ$. In addition, all functions in 
\eqref{C.3}--\eqref{C.6} have continuous nontangential limits as 
$\zeta \to \eta<0$, with generally different values on either side of 
the cut $(-\infty,0]$ due to the presence of the functions $\zeta^\nu$ 
and $\ln(\zeta)$. (We chose $\nu\in\bbR$ and subsequently usually $\nu\ge 0$ 
for simplicity only; complex values of $\nu$ are discussed in 
\cite[Ch.\,9]{AS72}.)

Due to the presence of the logarithmic term for even dimensions we next 
distinguish even and odd space dimensions $n$:

$(i)$ $n=2m+2$, $m\in\bbN_0$, and $z\in\bbC\backslash\{0\}$ fixed:   

\begin{align}
E_{2m+2}(z;x) &= \f{i}{4} \bigg(\f{2\pi |x|}{z^{1/2}}\bigg)^{-m} H^{(1)}_m 
\big(z^{1/2}|x|\big)  \no \\
& = \f{i}{4} \bigg(\f{2\pi |x|}{z^{1/2}}\bigg)^{-m} 
\big[J_m \big(z^{1/2}|x|\big) 
+ i Y_m \big(z^{1/2}|x|\big)\big]  \no \\
& = \f{i}{4} \bigg(\f{2\pi |x|}{z^{1/2}}\bigg)^{-m} \bigg\{\Oh\big(|x|^m\big) 
+ \f{2i}{\pi} \ln\bigg(\f{z^{1/2}|x|}{2}\bigg) \Oh\big(|x|^m\big)    
\lb{C.7}  \\
& \quad 
-\f{i}{\pi} \bigg(\f{z^{1/2}|x|}{2}\bigg)^{-m} (1-\delta_{m,0}) \bigg[(m-1)! 
+(1-\delta_{m,1}) (m-2)! \bigg(\f{z|x|^2}{4}\bigg)
+\Oh\big(|x|^4\big)\bigg]\bigg\}.   \no 
\end{align}

$(ii)$ $n=2m+1$, $m\in\bbN$, and $z\in\bbC\backslash\{0\}$ fixed:

\begin{align}
E_{2m+1}(z;x) &= \f{i}{4} \bigg(\f{2\pi |x|}{z^{1/2}}\bigg)^{(1/2)-m} 
H^{(1)}_{m-(1/2)} \big(z^{1/2}|x|\big)  \no \\
& = \f{i}{4} z^{m/2}\pi^{-m}2^{1-m} |x|^{1-m} h^{(1)}_{m-1} 
\big(z^{1/2}|x|\big) \no \\
& = \f{i}{2z^{1/2}} (2\pi i |x|)^{-m} e^{iz^{1/2}|x|} \sum_{k=0}^{m-1} 
\f{(m+k-1)!}{k! (m-k-1)!} \big(-2iz^{1/2}|x|\big)^{-k}  \no \\
& \hspace*{-.22cm} \underset{|x|\to 0}{=} \begin{cases}  (4\pi |x|)^{-1} 
\big[1 + i z^{1/2}|x| + \Oh\big(|x|^2\big)\big], & m=1, \\
[(2m-1)\omega_{2m}]^{-1} |x|^{1-2m} \big[1 + \Oh\big(|x|^2\big)\big], 
& m\geq 2,
\end{cases}    \lb{C.8}
\end{align}
with $h^{(1)}_\ell(\dott)$ defined in \cite[Sect.\ 10.1]{AS72}.

Given these expansions we can now summarize the behavior of $E_n(z;x)$ and 
its derivatives up to the second order as $|x|\to 0$: 

\begin{lemma}  \lb{lC.1}
Fix $z\in\bbC\backslash\{0\}$. Then the fundamental solution $E_n(z;\dott)$ 
of the Helmholtz equation $(-\Delta - z)\psi(z;\dott)=0$ and its derivatives 
up to the second order satisfy the following estimates for 
$0<|x| < R$, with $R>0$ fixed: 
\begin{align}
|E_n(z;x) - E_n(0;x)| & \leq \begin{cases} C, & n=2,3, \\
C[|\ln(|x|)| + 1], & n=4, \\
C \big[|x|^{4-n} + 1\big], & n\ge 5, 
\end{cases}    \lb{C.9}  \\[1mm]
|\partial_j E_n(z;x) - \partial_j E_n(0;x)| & \leq \begin{cases} C, & n=2,3, \\
C \big[|x|^{3-n} + 1\big], & n\ge 4, 
\end{cases}    \lb{C.10}  \\[1mm]
|\partial_j \partial_kE_n(z;x) - \partial_j \partial_kE_n(0;x)| 
& \leq \begin{cases} C[|\ln(|x|)| + 1], & n=2, \\
C \big[|x|^{2-n} + 1\big], & n\ge 3. 
\end{cases}    \lb{C.11}  
\end{align}
Here $C=C(R,n,z)$ represent various different constants in 
\eqref{C.9}--\eqref{C.11} and $\partial_j=\partial/\partial x_j$, 
$1\le j \le n$.
\end{lemma}
\begin{proof}
The estimates in \eqref{C.9} follow from combining \eqref{C.1}, \eqref{C.7}, 
and \eqref{C.8}. The estimates in \eqref{C.10} follow from the fact that
\begin{equation}
\partial_j E_n(z;x) = - 2\pi x_j E_{n+2}(z;x), \quad z\in\bbC\backslash\{0\}, 
\; x\in\bbR^n\backslash\{0\}, \; 1\le j \le n, \;  n\ge 2,   \lb{C.12}
\end{equation}
which permits one to reduce them essentially to \eqref{C.9} with $n$ 
replaced by $n+2$. The recursion relation \eqref{C.12} is a consequence 
of the well-known identity (cf. \cite[Sect.\,9.1]{AS72})
\begin{equation}
\f{d}{d\zeta}\big(\zeta^{-\nu} \cC_\nu(\zeta)\big) 
=-\zeta^{-\nu}\cC_{\nu+1}(\zeta), 
\quad \zeta\in\bbC\backslash\{0\}, \; \nu\in\bbR,   \lb{C.13}
\end{equation}
where $\cC_{\nu}(\dott)$ denotes any linear combination of Bessel 
functions of order $\nu$ with $\zeta$ and $\nu$ independent coefficients. 
Iterating \eqref{C.12} yields
\begin{align}
\begin{split}
\partial_j \partial_k E_n(z;x) = 4\pi^2 x_j x_k E_{n+4}(z;x) 
- 2\pi \delta_{j,k} E_{n+2}(z;x),& \\
z\in\bbC\backslash\{0\}, \; x\in\bbR^n\backslash\{0\}, \; 1\le j,k\le n, \;  
n\ge 2.&   \lb{C.14}
\end{split}
\end{align}
Combining \eqref{C.9} and \eqref{C.14} then yields \eqref{C.11}.
\end{proof}

Finally, we mention for completeness that for large values of $|x|$, 
\eqref{C.1} implies the following simple asymptotic behavior 
(cf. \cite[Sect.\,9.1]{AS72}):

\begin{equation}
E_n(z;x) \underset{|x|\to\infty}{=} 
\f{i}{z^{1/2}} \bigg(\f{2 \pi |x|}{z^{1/2}}\bigg)^{(1-n)/2} 
e^{i\big(z^{1/2}|x| -\pi((n-1)/4)\big)} \big[1 + \Oh\big(|x|^{-1}\big)\big], 
\quad z\in\bbC\backslash\{0\}, \; \Im\big(z^{1/2}\big)\ge 0.    \lb{C.15}
\end{equation} 
In particular, as long as $z\in\bbC\backslash [0,\infty)$ (and hence 
$\Im\big(z^{1/2}\big)>0$), $E_n(z;x)$ decays exponentially with respect 
to $x$ as $|x|\to\infty$.

\section{Calder\'on--Zygmund Theory on Lipschitz Surfaces}
\lb{sD}
\renewcommand{\theequation}{D.\arabic{equation}}
\renewcommand{\thetheorem}{D.\arabic{theorem}}
\setcounter{theorem}{0} \setcounter{equation}{0}

This appendix records various useful consequences of the Calder\'on--Zygmund 
theory on Lipschitz surfaces.

Our first result, Lemma \ref{L-cpt} below, is modeled upon a more general 
result in \cite{HMT}. For the sake of completeness we include the full 
argument. 

\begin{lemma}\label{L-cpt}
Let $\Omega\subset{\mathbb{R}}^n$ be a Lipschitz domain with compact 
boundary and let $k(\dott,\dott)$ be a real-valued, measurable function on 
$\partial\Omega\times\partial\Omega$ satisfying
\begin{equation}\label{MM-1.1}
|k(x,y)|\leq\frac{\psi(|x-y|)}{|x-y|^{n-1}},\quad x,y\in\partial\Omega,
\end{equation}
where $\psi$ is monotone increasing and satisfies
\begin{equation}\label{MM-1.2}
\int_0^1 dt \, \frac{\psi(t)}{t} <\infty.
\end{equation}
Consider
\begin{equation}\label{MM-1.3}
(Kf)(x)=\int_{\partial\Omega}d^{n-1}\omega(y)\,k(x,y)f(y),\quad
x\in\partial\Omega.
\end{equation}
Then  
\begin{equation}\label{MM-1.3a}
K\in \cB_\infty\big(L^p(\partial\Omega;d^{n-1}\omega)\big)
\end{equation}
for each $p\in (1,\infty)$.  
\end{lemma}
\begin{proof} For a fixed, arbitrary $\varepsilon>0$, decompose 
$k(x,y)=k_{\varepsilon}(x,y)+k_b(x,y)$, where 
\begin{equation}\label{MM-1.4}
k_{\varepsilon}(x,y)=\begin{cases} 
k(x,y), & |x-y|\leq\varepsilon, \\ 
0, & |x-y|>\varepsilon.
\end{cases}
\end{equation}
Then $K=K_{\varepsilon}+K_b$, where $K_{\varepsilon}$, $K_b$ 
are integral operators on $\partial\Omega$ with integral kernels 
$k_{\varepsilon}(x,y)$ and $k_b(x,y)$, respectively. Setting 
\begin{equation}\label{MM-1.5}
S_j(x)=\big\{y\in\partial\Omega\,|\,2^{-j-1}\leq |x-y|<2^{-j}\big\},\quad 
x\in\partial\Omega, \; j\in\bbN, 
\end{equation}
for each $x\in\partial\Omega$ 
we may then compute (with the logarithm taken in base $2$) 
\begin{align}\label{MM-1.6}
\int_{\partial\Omega}d^{n-1}\omega(y)\,|k_{\varepsilon}(x,y)| & \leq 
C\sum_{j\geq\log 1/\varepsilon}\int_{S_j(x)}d^{n-1}\omega(y)\,
\frac{\psi(|x-y|)}{|x-y|^{n-1}} 
\nonumber\\
& \leq  C\sum_{j\geq\log 1/\varepsilon}\psi(e^{-j})
\leq C\int_0^\varepsilon  dt \, \frac{\psi(t)}{t}. 
\end{align}
Of course, there is a similar estimate for 
$\int_{\partial\Omega}d^{n-1}\omega(x)\,|k_{\varepsilon}(x,y)|$, uniformly for 
$y\in\partial\Omega$. Schur's lemma then yields 
\begin{equation}\label{MM-1.7}
\|K_{\varepsilon}\|_{\cB(L^p(\partial\Omega;d^{n-1}\omega))}
\leq C\int_0^\varepsilon dt \, \frac{\psi(t)}{t} \rightarrow 0\, \text{ as }\, 
\varepsilon\rightarrow 0. 
\end{equation}
Thus, it suffices to show that $K_b$ is compact on each 
$L^p(\partial\Omega;d^{n-1}\omega)$ space, for $p\in (1,\infty)$, under 
the hypothesis that $k_b(x,y)$ is bounded. First note that $K_b$ is compact 
on $L^2(\partial\Omega,d^{n-1}\omega)$, since it is Hilbert-Schmidt, 
due to the fact that $\omega(\partial\Omega)<\infty$. The compactness of 
$K_b$ on $L^p(\partial\Omega,d^{n-1}\omega)$ for each $p\in (1,\infty)$ 
then follows from an interpolation theorem of Krasnoselski 
(see, e.g., \cite[Theorem 2.9, p.\,203]{BS}). 
\end{proof}

We now record a basic result from the theory of singular integral
operators of Calder\'on--Zygmund-type on Lipschitz domains. To state it,
we recall that $\hat f$ denotes the Fourier transform of appropriate functions 
$f\colon{\mathbb{R}}^n\to\bbC$.
Moreover, given a Lipschitz domain $\Omega\subset{\mathbb{R}}^n$, set
$\Omega_+=\Omega$, $\Omega_-={\mathbb{R}}^n\backslash\ol{\Omega}$, 
we define the nontangential approach regions $\Gamma^\pm_{\kappa}(x)$,
$x\in\partial\Omega$, by $\Gamma^\pm_{\kappa}(x)=\{y\in\Omega_\pm\,|\,
|x-y|<(1+\kappa)\,{\rm dist}\,(y,\partial\Omega)\}$, where $\kappa>0$ is a
fixed parameter. Next, at every boundary point the nontangential maximal
function of a mapping $u$ (defined in either $\Omega_{+}$ or $\Omega_{-}$) 
is given by 
\begin{equation}\label{Rb-2}
(M u)(x)=\sup\,\{|u(y)| \,|\, y\in\Gamma^\pm_{\kappa}(x)\}
\end{equation}
(with the choice of sign depending on whether $u$ is 
defined in $\Omega_+$, or $\Omega_-$) and, for $u$ defined in $\Omega_\pm$, 
we set
\begin{equation}\label{MM-0}
(\gamma_{\rm n.t.} u)(x)=\lim_{\stackrel{y\to x}
{y\in\Gamma^\pm_{\kappa}(x)}}u(y) \, 
\text{ for a.e.\ }x\in\partial\Omega.
\end{equation}
For future reference, let us record here a useful
estimate proved in \cite{DM02}, valid for any Lipschitz domain 
$\Omega\subset{\mathbb{R}}^n$ which is either bounded or has 
an unbounded boundary. In this setting, for any $p\in(0,\infty)$ and 
any function $u$ defined in $\Omega$, 
\begin{equation}\label{MD-E}
\|u\|_{L^{np/(n-1)}(\Omega;d^nx)}\leq C(\Omega,n,p)
\|Mu\|_{L^p(\partial\Omega;d^{n-1}\omega)}. 
\end{equation}

\begin{theorem}\label{T-CZ}
There exists a positive integer $N=N(n)$ with the following significance.
Let $\Omega\subset{\mathbb{R}}^n$ be a Lipschitz domain with compact boundary, 
and assume that  
\begin{equation}\label{ker}
k\in C^N({\mathbb{R}}^n\backslash\{0\})\,\text{ with }\, 
k(-x)=-k(x)\, \text{ and }\, k(\lambda x)=\lambda^{-(n-1)}k(x), \; 
\lambda>0, \; x\in\bbR^n\backslash\{0\}. 
\end{equation}
Define the singular integral operator
\begin{equation}\label{T-layer}
({\mathcal{T}}f)(x)=\int_{\partial\Omega}d^{n-1}\omega(y)\,k(x-y)f(y),
\quad x\in{\mathbb{R}}^n\backslash\partial\Omega.
\end{equation}
Then for each $p\in(1,\infty)$ there exists a finite constant
$C=C(p,n,\partial\Omega)>0$ such that
\begin{equation}\label{T-nt}
\|M({\mathcal{T}}f)\|_{L^p(\partial\Omega;d^{n-1}\omega)}
\leq C\|k|_{S^{n-1}}\|_{C^N}\|f\|_{L^p(\partial\Omega;d^{n-1}\omega)}.
\end{equation}
Furthermore, for each $p\in (1,\infty)$, 
$f\in L^p(\partial\Omega;d^{n-1}\omega)$, the limit 
\begin{equation}\label{main-lim}
(Tf)(x)={\rm p.v.}\int_{\partial\Omega}d^{n-1}\omega(y)\,k(x-y)f(y)
=\lim_{\varepsilon\to 0^+}
\int_{\stackrel{y\in\partial\Omega}{|x-y|>\varepsilon}}
d^{n-1}\omega(y)\,k(x-y)f(y)
\end{equation}
exists for a.e. $x\in\partial\Omega$, and the jump-formula 
\begin{equation}\label{main-jump}
\gamma_{\rm n.t.}({\mathcal{T}}f)(x)=
\lim_{\stackrel{z\to x}{z\in\Gamma^{\pm}_{\kappa}(x)}}({\mathcal{T}}f)(z)
=\pm{\textstyle{\frac{1}{2i}}}{\widehat{k}}(\nu(x))f(x) + (Tf)(x)
\end{equation}
is valid at a.e.\ $x\in\partial\Omega$, where $\nu$ denotes the
unit normal pointing outwardly relative to $\Omega$ (recall that `hat' denotes
the Fourier transform in $\bbR^n$). 

Finally, 
\begin{equation}\label{H1/2}
\|{\mathcal{T}}f\|_{H^{1/2}(\Omega)}\leq 
C\|f\|_{L^2(\partial\Omega;d^{n-1}\omega)}. 
\end{equation}
\end{theorem}
\noindent See the discussion in \cite{CDM}, \cite{CMM}, \cite{MMT01}. 

\begin{lemma}\label{L-Kb}
Whenever $\Om$ is a Lipschitz domain with compact boundary in $\bbR^n$, 
\begin{equation}\lb{MM-1.8}
K^{\#}_z \in \cB\big(\LdOm\big), \quad z\in\bbC,   
\end{equation}
and 
\begin{align}\lb{MM-1.9}
&\big(K^{\#}_{z_1} - K^{\#}_{z_2}\big)\in \cB_\infty\big(\LdOm\big), 
\quad z_1,z_2\in\bbC,   \\ 
& \gamma_D\cS_{z}\in \cB\big(\LdOm,H^1(\dOm)\big),\quad z\in\bbC.
\lb{MM-1.9X}
\end{align}
\end{lemma}
\begin{proof} 
We recall the fundamental solution $E_n(z;\dott)$ for the 
Helmholtz equation $(-\Delta-z)\psi(z;\dott)=0$ in ${\mathbb{R}}^n$ 
introduced in \eqref{2.52}. Then the integral kernel of the operator 
$K^{\#}_{z} - K^{\#}_{0}$ is given by 
\begin{equation}\label{MM-1.10}
k(x,y)=\nu(x)\cdot\big(\nabla E_n(z;x-y)-\nabla E_n(0;x-y)\big),\quad
x,y\in\partial\Omega. 
\end{equation}
By \eqref{C.10} we therefore have $|k(x,y)|\leq C|x-y|^{2-n}$, 
hence \eqref{MM-1.1} holds with $\psi(t)=t$. Note that \eqref{MM-1.2} is
satisfied for this choice of $\psi$, so \eqref{MM-1.9} is a 
consequence of Lemma \ref{L-cpt}. In addition, \eqref{MM-1.8} follows 
from \eqref{MM-1.9} and Theorem \ref{T-CZ}, according to which 
$K^{\#}_0 \in \cB\big(\LdOm\big)$. Finally, the reasoning for 
\eqref{MM-1.9X} is similar (here \eqref{M1.1y} is useful). 
\end{proof}

\begin{lemma}\label{L-KbX}
If $\Om$ is a $C^{1,r}$, $r>1/2$, domain in $\bbR^n$ with compact boundary, then
\begin{eqnarray}\lb{MM-1G}
\big(K^{\#}_{z_1} - K^{\#}_{z_2}\big)\in \cB_\infty\big(H^{1/2}(\dOm)\big), 
\quad z_1,z_2\in\bbC.
\end{eqnarray}
\end{lemma}
\begin{proof} 
The integral kernel of the operator $K^{\#}_{z} - K^{\#}_{0}$ is given by 
\eqref{MM-1.10}. By Lemma~\ref{lA.6}, the operator of multiplication 
by components of $\nu\in [C^r(\dOm)]^n$ belongs to 
$\cB\bigl(H^{1/2}(\dOm)\bigr)$. Hence, it suffices to show that
the boundary integral operators whose integral kernels are of the form 
\begin{equation}\label{MK-1}
\partial_j E_n(z;x-y)-\partial_j E_n(0;x-y)\big),\quad
x,y\in\partial\Omega,\,\,\,j\in\{1,...,n\},
\end{equation}
belong to $\cB\bigl(L^2(\dOm;d^{n-1}\omega),H^1(\dOm)\bigr)$.
This, however, is a consequence of \eqref{Pf-2.3}, \eqref{Pf-2.4} (with $s=1$),
\eqref{C.11}, and Lemma~\ref{L-cpt} (with $\psi(t)=t$). 
\end{proof}

\begin{lemma}\label{L-M2}
Let $0<\alpha<(n-1)$ and $1<p<q<\infty$ be related by 
\begin{equation}\label{MM-3.1}
\frac{1}{q}=\frac{1}{p}-\Bigl(\alpha+\frac{1}{p}\Bigr)\frac{1}{n}.
\end{equation}
Then the the operator $J_\alpha$ defined by 
\begin{equation}\label{MM-3.2}
J_\alpha f(x)=\int_{{\mathbb{R}}^{n-1}}d^{n-1}y\,
\frac{1}{|x-y|^{n-1-\alpha}}f(y),\quad x\in\overline{{\mathbb{R}}^n_+}, \; 
f\in L^p({\mathbb{R}}^{n-1};d^{n-1}x), 
\end{equation}
is bounded from $L^p({\mathbb{R}}^{n-1};d^{n-1}x)$ to 
$L^q({\mathbb{R}}^n_+;d^nx)$, that is, for some constant $C_{\alpha,p,q}>0$,
\begin{equation}\label{MM-3.3}
\|J_\alpha f\|_{L^q({\mathbb{R}}^n_+;d^nx)}\leq 
C_{\alpha,p,q}\|f\|_{L^p({\mathbb{R}}^{n-1};d^{n-1}x)}, \quad 
f\in L^p({\mathbb{R}}^{n-1};d^{n-1}x).
\end{equation}
\end{lemma}
\begin{proof} A direct proof appears in \cite{MT05}. An alternative 
argument is to observe that $M(J_\alpha f)(x)\leq CJ_\alpha(|f|)(x)$,
uniformly for $x\in\partial{\mathbb{R}}^n_+$, and then to invoke the general
estimate \eqref{MD-E} in concert with the classical Hardy-Littlewood-Sobolev 
fractional integration theorem (cf., e.g., \cite{St70}, Theorem 1 on p.\,119). 
\end{proof}

Next, we record a lifting result for Sobolev spaces in Lipschitz domains
in \cite{JK95}. 

\begin{theorem}\label{liftOmega}
Let $\Omega\subset \mathbb{R}^n$ be a Lipschitz domain with compact boundary.
Then, for every $\alpha>0$, the following equivalence of norms holds: 
\begin{equation}\label{liftBb}
\|u\|_{H^{\alpha+1}(\Omega)}\approx
\|u\|_{L^2(\Omega)}+\|\nabla u\|_{H^{\alpha}(\Omega)}. 
\end{equation}
\end{theorem}

\begin{theorem}\label{T-SH}
Let $\Omega\subset\bbR^n$ be a bounded Lipschitz domain. Then 
for every $z\in{\mathbb{C}}$, 
\begin{equation}\label{Sz-M6}
{\mathcal{S}}_z\in\cB\big(L^2(\partial\Omega;d^{n-1}\omega),H^{3/2}(\Omega)\big),
\end{equation}
and 
\begin{equation}\label{Sz-MM5}
{\mathcal{S}}_z\in\cB\big(H^{-1}(\partial\Omega),H^{1/2}(\Omega)\big).
\end{equation}
In particular,
\begin{equation}\label{Sz-MM7}
{\mathcal{S}}_z\in\cB\big(H^{s-1}(\partial\Omega),H^{s+(1/2)}(\Omega)\big), 
\quad 0 \le s \le 1.
\end{equation}
\end{theorem}
\begin{proof} Given $f\in L^2(\partial\Omega;d^{n-1}\omega)$, write 
${\mathcal{S}}_zf
={\mathcal{S}}_0f+({\mathcal{S}}_z-{\mathcal{S}}_0)f$. From
\eqref{H1/2} and Lemma \ref{liftOmega}
we know that $\|{\mathcal{S}}_0f\|_{H^{3/2}(\Omega)}\leq C
\|f\|_{L^2(\partial\Omega;d^{n-1}\omega)}$, for some constant $C>0$
independent of $f$. Using \eqref{C.11} and Lemma~\ref{L-M2} (with $\alpha=1$) 
one concludes that 
\begin{equation}\label{Sz-M7}
\nabla^2 ({\mathcal{S}}_z-{\mathcal{S}}_0)
\in\cB\big(L^2(\partial\Omega;d^{n-1}\omega), L^2(\Omega;d^nx)\big),
\end{equation}
and \eqref{Sz-M6} follows from this. The proof of \eqref{Sz-MM5}, 
is analogous and has as starting point the fact that 
${\mathcal{S}}_0f\in\cB\bigl(H^{-1}(\partial\Omega),H^{1/2}(\Omega)\bigr)$, 
itself a consequence of \eqref{H1/2} and the following description of 
$H^{-1}(\partial\Omega)$:
\begin{eqnarray}\label{Mz-2}
H^{-1}(\partial\Omega)
=\biggl\{g+\sum_{1\leq j,k\leq n}(\partial f_{j,k}/\partial
\tau_{j,k})\,\bigg|\,g,f_{j,k}\in\LdOm\biggr\}.
\end{eqnarray}
Then \eqref{Sz-M7} ensures that 
\begin{equation}\label{Sz-M7b}
({\mathcal{S}}_z-{\mathcal{S}}_0)
\in\cB\big(H^{-1}(\partial\Omega), H^1(\Omega)\big),
\end{equation}
and \eqref{Sz-MM5} follows. 
\end{proof}

We recall the adjoint double layer on $\partial\Omega$ introduced 
in \eqref{Ksharp} and denote by 
\begin{equation}\label{Ks-1}
(K_z g)(x)={\rm p.v.}\int_{\partial\Omega}d^{n-1}\omega(y)\,
\partial_{\nu_y}E_n(z;y-x)g(y),
\quad x\in\partial\Omega,
\end{equation}
its adjoint. It is well-known (cf., e.g., \cite{Ve84}) that 
\begin{equation}\label{Ks-2}
{\rm Hypothesis}~\ref{h2.1}\Longrightarrow 
K\in \cB\big(L^2(\dOm;d^{n-1}\omega)\big)\cap 
\cB\big(H^1(\dOm)\big)
\end{equation}
and (cf. \cite{FJR78} and \eqref{MM-1.9}) that 
\begin{equation}\label{Ks-3}
\Omega\mbox{ a bounded $C^1$-domain }\Longrightarrow 
K_z\in \cB_\infty\big(L^2(\dOm;d^{n-1}\omega)\big),\quad z\in\bbC.
\end{equation}
It follows from \eqref{Ks-2}, \eqref{Ks-3}, \eqref{Gh-a2}, and 
Theorem \ref{T-Cwikel} that 
\begin{equation}\label{Ks-4}
\Omega\mbox{ a bounded $C^1$-domain }\Longrightarrow 
K_z\in\cB_\infty\big(H^s(\dOm)\big),\quad s\in(0,1),\; z\in\bbC.
\end{equation}
We wish to complement this with the following compactness result. 

\begin{theorem}\label{K-CPT}
If $\Om\subset\bbR^n$ is a bounded $C^{1,r}$-domain with $r\in(1/2,1)$ then 
\begin{equation}\label{Ks-4B}
K^{\#}_z\in\cB_\infty\big(H^{1/2}(\dOm)\big),\quad z\in\bbC.
\end{equation}
\end{theorem}

The proof of this result (presented at the end of this section)
requires a number of tools from the theory of singular integral 
operators which we now review, or develop. 

\begin{theorem}\label{t2.5X}
Let $A:{\mathbb{R}}^n\to{\mathbb{R}}^m$ be a Lipschitz function, and assume 
that $F:{\mathbb{R}}^m\to{\mathbb{R}}$, $F\in C^1(\mathbb{R}^m)$, $F$ is an odd 
function. For $x,y\in{\mathbb{R}}^n$ with $x\not=y$ set
$K(x,y)= \frac{1}{|x-y|^n}F\left(\frac{A(x)-A(y)}{|x-y|}\right)$, and
for $\varepsilon>0$, $f$ a Lipschitz function with compact support in 
${\mathbb{R}}^n$, define the truncated operator
\begin{equation}\label{eq2.1-3}
(T_\varepsilon f)(x)= \int\limits_{\stackrel{y\in\bbR^n}{|x-y|>\varepsilon}}d^n y\,
K(x,y)f(y),\quad x\in\bbR^n.
\end{equation}
Then, for each $1<p<\infty$, the following assertions hold: 
\begin{enumerate}
\item[(i)] The maximal operator
$(T_* f)(x)= {\rm sup}\,\{|(T_{\varepsilon} f)(x)| \,|\,\varepsilon>0\}$ 
is bounded on $L^p({\mathbb{R}}^n;d^nx)$. 

\item[(ii)] If $1<p<\infty$ and $f\in L^p({\mathbb{R}}^n;d^nx)$ then the limit
$\lim_{\varepsilon\to 0} (T_{\varepsilon} f)(x)$ exists for almost every
$x\in{\mathbb{R}}^n$ and the operator

\begin{equation}\label{operatorT}
(T f)(x)= \lim_{\varepsilon\to 0} (T_{\varepsilon} f)(x)
\end{equation}
\noindent is bounded on $L^p({\mathbb{R}}^n;d^nx)$.
\end{enumerate}
\end{theorem}

A proof of this result can be found in \cite{Me90}.

\begin{theorem}\label{t2.6X} 
Let $A:{\mathbb{R}}^n\to{\mathbb{R}}^m$,
$B=(B_1,...,B_\ell):{\mathbb{R}}^n\to{\mathbb{R}}^\ell$ be two Lipschitz 
functions and let $F:{\mathbb{R}}^m\times{\mathbb{R}}^\ell\to{\mathbb{R}}$ 
be a $C^N$ (with $N=N(n,m,\ell)$ a sufficiently large integer) odd function 
which satisfies the decay conditions
\begin{align}
& |F(a,b)|\leq C(1+|b|)^{-n},   \label{eq2.23} \\
& |\nabla_I F(a,b)|\leq C,    \label{eq2.24} \\
& |\nabla_{II} F(a,b)|\leq C(1+|b|)^{-1},    \label{eq2.25} 
\end{align}
uniformly for $a$ in compact subsets of ${\mathbb{R}}^n$ and
arbitrary $b\in{\mathbb{R}}^\ell$. $($Above, $\nabla_I$ and $\nabla_{II}$ 
denote the gradients with respect to the first and second sets of variables.$)$ 
For $x,y\in{\mathbb{R}}^n$ with $x\not=y$ and $t>0$ we set
\begin{equation}\label{eq2.26}
K^t(x,y)= \frac{1}{|x-y|^n}
F\left(\frac{A(x)-A(y)}{|x-y|},\frac{B_1(x)-B_1(y)+t}{|x-y|},\cdots,
\frac{B_\ell(x)-B_\ell(y)+t}{|x-y|}\right).
\end{equation}
In addition, for each $t>0$ we introduce
\begin{equation}\label{e3.3.5}
(T^t f)(x)= \int_{{\mathbb{R}}^n}d^n y\,K^t(x,y)f(y),
\quad x\in{\mathbb{R}}^n,
\end{equation}
and, for some fixed, positive $\kappa$, 
\begin{equation}\label{e3.3.6}
(T_{**} f)(x)
={\rm sup}\,\{|(T^t f)(z)|\,|\, z\in\bbR^n,\,\,t>0,\,\,|x-z|<\kappa t\},
\quad x\in{\mathbb{R}}^n.
\end{equation}

Then, for each $1<p<\infty$, the following assertions are valid:
\begin{enumerate}
\item[(1)] The nontangential maximal operator 
$T_{**}$ is bounded on $L^p({\mathbb{R}}^n;d^nx)$. 

\item[(2)] For each $f\in L^p({\mathbb{R}}^n;d^nx)$, the limit
\begin{equation}\label{eq2.27}
({\mathcal T} f)(x)= \lim_{\stackrel{|x-z|<\kappa t}{z\to x,\,t\to 0}} (T^t f)(z)
\end{equation}
exists at almost every $x\in{\mathbb{R}}^n$ and the operator 
${\mathcal T}$ is bounded on $L^p({\mathbb{R}}^n;d^nx)$.
\end{enumerate}
\end{theorem}
\begin{proof}
Fix $p\in(1,\infty)$. For $x,y\in{\mathbb{R}}^n$ with $x\not=y$ 
consider the kernel 
\begin{equation}\label{eq2.28}
K(x,y)= \frac{1}{|x-y|^n}
F\left(\frac{A(x)-A(y)}{|x-y|},\frac{B(x)-B(y)}{|x-y|}\right),
\end{equation}
and let $T,T_*$ be the operators canonically associated with this
integral kernel as in Theorem \ref{t2.5X}. The crux of the matter is 
establishing the a.e. pointwise estimate
\begin{equation}\label{eq2.29}
T_{**}f\leq CT_*f+C{\mathcal M}f\, \text{ in }\, \bbR^n,
\end{equation}
uniformly for $f\in L^p({\mathbb{R}}^n;d^nx)$, where ${\mathcal M}$ is 
the Hardy-Littlewood maximal operator in $\bbR^n$. Then the first claim in the 
statement of the theorem follows from Theorem \ref{t2.5X} and the 
well-known fact that ${\mathcal M}$ is bounded on $L^p({\mathbb{R}}^n;d^nx)$.

To this end, fix $x,z\in{\mathbb{R}}^n$, $t>0$ such that
$|x-z|<\kappa t$, and let $\alpha>0$ be a large constant,
to be specified later. Then
\begin{align}\label{eq2.30}
& \left|\int_{{\mathbb{R}}^n}d^n y\,K^t(z,y)f(y)
-\int_{|x-y|>\alpha t}d^n y\,K(x,y)f(y)\right|  \\
& \quad \leq\int_{|x-y|<\alpha t}d^ny\,|K^t(z,y)||f(y)|
+\int_{|x-y|>\alpha t}d^n y\,|K^t(z,y)-K(x,y)||f(y)|= I + II.   \label{eq2.31}
\end{align}
Clearly, it suffices to show that $|I|$, $|II|\leq
C{\mathcal M}f$. To see this, first observe that
\begin{equation}\label{eq2.32}
|K^t(z,y)|\leq Ct^{-n}\,\text{ uniformly for any }
\, z,y\in{\mathbb{R}}^n,\,\, z\not=y
\end{equation}
(in fact, this also justifies that 
$T^t$ is well-defined). Indeed, using the fact that for each 
$j\in\{1,...,\ell\}$ one has $Ct\leq|B_j(z)-B_j(y)+t|+|z-y|$ 
(easily seen by analyzing the cases
$|z-y|\geq\frac{t}{2\|\nabla B_j\|_{L^\infty}}$ and
$|z-y|\leq\frac{t}{2\|\nabla B_j\|_{L^\infty}}$), we may infer that 
\begin{equation}\label{eq2.33}
\left(1+\sum_{j=1}^\ell\frac{|B_j(z)-B_j(y)+t|}{|z-y|}\right)^{-n} \leq
C\left(\frac{t}{|z-y|}\right)^{-n}.
\end{equation}
With this at hand, the estimate (\ref{eq2.32}) is a
direct consequence of (\ref{eq2.23}). Returning to $I$, from
(\ref{eq2.32}), we deduce that $|I|\leq C{\mathcal M}f(x)$.

Thus, we are left with analyzing $II$ in \eqref{eq2.31}. 
To begin with, we shall prove that 
\begin{equation}\label{eq2.34}
|K^t(z,y)-K(x,y)|\leq
Ct|x-y|^{-n-1}\,\text{ for }\, |x-y|>\alpha t.
\end{equation}
Let $G_y(x,t)= K^t(x,y)$. Then
\begin{equation}\label{eq2.35}
|K^t(z,y)-K(x,y)|=|G_y(z,t)-G_y(x,0)|
\end{equation}
can be estimated using the Mean Value Theorem by
\begin{equation}\label{eq2.36}
Ct(|\nabla_IG_y(w,s)|+|\nabla_{II}G_y(w,s)|),
\end{equation}
where $w=(1-\theta)z+\theta x$, $s=(1-\theta)t$ for some
$0<\theta<1$. Next,
\begin{align}
|\nabla_IG_y(w,s)|&\leq \frac{C}{|w-y|^{n+1}}
\left|F\left(\frac{A(w)-A(y)}{|w-y|},\frac{B_1(w)-B_1(y)+s}{|w-y|},...,
\frac{B_\ell(w)-B_\ell(y)+s}{|w-y|}\right)\right|  \no \\[1mm] 
& \quad +\frac{C}{|w-y|^{n+1}}
\left|\nabla_I F\left(\frac{A(w)-A(y)}{|w-y|},
\frac{B_1(w)-B_1(y)+s}{|w-y|},...,\frac{B_\ell(w)-B_\ell(y)+s}
{|w-y|}\right)\right|  \no \\[1mm]
& \quad +\frac{C}{|w-y|^{n}}\left|\nabla_{II} F\left(\frac{A(w)-A(y)}{|w-y|},
\frac{B_1(w)-B_1(y)+s}{|w-y|},...,\frac{B_\ell(w)-B_\ell(y)+s}{|w-y|}
\right)\right|   \no \\[1mm] 
& \hspace*{2.2cm}
\times\left(\frac{C}{|w-y|}+\sum_{i=1}^\ell
\frac{|B_i(w)-B_i(y)+s|}{|w-y|^2}\right).
\end{align}
Keeping in mind the restrictions on the size of the derivatives
of the function $F$ stated in \eqref{eq2.23}--\eqref{eq2.25}, we conclude that
the above expression is bounded by $C|w-y|^{-(n+1)}$.

Similarly, it can be shown that
\begin{equation}\label{eq2.38}
|\nabla_{II}G_y(w,s)|\leq\frac{C}{|w-y|^{n+1}}.
\end{equation}
To continue, one observes that if we choose $\alpha>\kappa$ then, in the 
current context, 
\begin{eqnarray}\label{eq2.40}
|w-x|\leq |z-x|\leq \kappa t=\left(\frac{\kappa}{\alpha}\right)\alpha t
<\left(\frac{\kappa}{\alpha}\right)|x-y|,
\end{eqnarray}
and $|w-x|+|w-y|\geq |x-y|$. Hence, 
\begin{eqnarray}\label{eq2.40G}
|w-y|\geq \left(1-\frac{\kappa}{\alpha}\right)|x-y|,
\end{eqnarray}
and, therefore, 
\begin{equation}\label{eq2.41}
|K^t(z,y)-K(x,y)|\leq Ct|x-y|^{-n-1}.
\end{equation}

Next, we split the domain of integration of $II$ 
(appearing in \eqref{eq2.31}) into dyadic annuli of the form 
$2^j\alpha t\leq|x-y|\leq 2^{j+1}\alpha t$, $j=0,1,2,...$. Then
\begin{align}\label{eq2.42}
 \int_{|x-y|>\alpha t}d^n y\,|K^t(z,y) - K(x,y)||f(y)| 
& \leq \sum_{j=0}^{\infty} \int_{2^{j}\alpha t<|x-y|<2^{j+1}\alpha t}d^n y\,
\frac{t}{|x-y|^{n+1}}|f(y)|  \no \\
&  \leq C\sum_{j=0}^{\infty} 2^{-j}(\cM f)(x)=C (\cM f)(x). 
\end{align}
This yields the desired inequality, that is, $|II|\leq C{\mathcal M}f(x)$.

The proof of last second claim in the statement of the theorem utilizes
a well-known principle (cf., e.g., \cite{FJR78}) to the effect that pointwise 
convergence for a dense class along with the boundedness of the maximal 
operator associated with the type of convergence in question always 
entails a.e. convergence for the entire space $L^p(\bbR^n;d^nx)$. 
Thus, it suffices to identify a dense subspace ${\mathcal V}$ of
$L^p({\mathbb{R}}^n;d^nx)$ such that for any $f\in {\mathcal V}$ the limit 
in question exists for almost every $x\in{\mathbb{R}}^n$. Then the boundedness 
of the maximal operator associated with the type of convergence under 
discussion ensures that this limit exists for any 
$f\in L^p({\mathbb{R}}^n;d^nx)$ at almost every $x\in{\mathbb{R}}^n$. 

In our situation, we may take ${\mathcal V}= C^1_0({\mathbb{R}}^n)$ 
and observe that 
\begin{equation}\label{e3.3.23}
\lim_{\stackrel{|x-z|<\kappa t}{z\to x,\,t\to 0}} (T^t f)(z)
=\lim_{\varepsilon\to 0}\lim_{\stackrel{|x-z|<\kappa t}{z\to x,\,t\to 0}}
\Bigl[I+II+III\Bigr],
\end{equation}
where 
\begin{align}\label{e3.3.24}
I & =   \int_{|x-y|>1}d^n y\,K^t(z,y)f(y),  \no \\
II & =   \int_{1>|x-y|>\varepsilon}d^n y\,K^t(z,y)[f(y)-f(x)],  \\
III & =   f(x)\int_{1>|x-y|>\varepsilon}d^n y\,K^t(z,y).  \no
\end{align}
Consequently, 
\begin{align}
&\lim_{\varepsilon\to 0}\lim_{\stackrel{|x-z|<\kappa t}{z\to x,\,t\to 0}}I
=\int_{|x-y|>1}d^n y\,K(x,y)f(y),  \\
&\lim_{\varepsilon\to 0}\lim_{\stackrel{|x-z|<\kappa t}{z\to x,\,t\to 0}}II
=\int_{1>|x-y|}d^n y\,K(x,y)[f(y)-f(x)],
\end{align}
whereas 
\begin{equation}\label{e3.3.27}
\lim_{\varepsilon\to 0}\lim_{\stackrel{|x-z|<\kappa t}{z\to x,\,t\to 0}}III
=\lim_{\varepsilon\to 0}\int_{1>|x-y|>\varepsilon}d^n y\,K(x,y).
\end{equation}
Now, this last limit is known to exists at a.e. 
$x\in{\mathbb{R}}^n$ (see, e.g., \cite{MW08}). 
Once the pointwise definition of the operator ${\mathcal T}$ has been 
shown to be meaningful, the boundedness of this operator on 
$L^p({\mathbb{R}}^n;d^nx)$, $1<p<\infty$, is implied by that of $T_{\ast\ast}$.
\end{proof}

\begin{theorem}\label{T-CZ2}
There exists a positive integer $N=N(n)$ with the following significance: 
Let $\Omega\subset{\mathbb{R}}^n$ be a Lipschitz domain with compact boundary, 
and assume that  
\begin{equation}\label{kerX}
k\in C^N({\mathbb{R}}^n\backslash\{0\})\,\text{ with }\, 
k(-x)=-k(x)\, \text{ and }\, k(\lambda x)=\lambda^{-(n-1)}k(x), \; 
\lambda>0, \; x\in\bbR^n\backslash\{0\}. 
\end{equation}
Fix $\eta\in C^N({\mathbb{R}}^n)$ and define the singular integral operator
\begin{equation}\label{T-layerX}
(T f)(x)={\rm p.v.}\int_{\partial\Omega}d^{n-1}\omega(y)\,(\eta(x)-\eta(y))
k(x-y)f(y),\quad x\in\partial\Omega.
\end{equation}
Then 
\begin{equation}\label{T-ntX}
T\in\cB\bigl(L^2(\dOm;d^{n-1}\omega),H^1(\dOm)\bigr).
\end{equation}
\end{theorem}
\begin{proof}
Fix an arbitrary $f\in L^2(\dOm;d^{n-1}\omega)$ and consider 
\begin{equation}\label{T-X2}
u(x)= \int_{\partial\Omega}d^{n-1}\omega(y)\,(\eta(x)-\eta(y))
k(x-y)f(y),\quad x\in\Omega.
\end{equation}
Since $Tf=u|_{\partial\Omega}$, it suffices to show that 
\begin{equation}\label{T-X1}
\|M(\nabla u)\|_{L^2(\dOm;d^{n-1}\omega)}\leq C\|f\|_{L^2(\dOm;d^{n-1}\omega)},
\end{equation}
for some finite constant $C=C(\Omega)>0$ (where the nontangential maximal
operator $M$ is as in \eqref{Rb-2}). With this goal in mind, for a given 
$j\in\{1,...,n\}$, we decompose 
\begin{equation}\label{T-X3}
(\partial_ju)(x)=u_1(x)+u_2(x),\quad x\in\Omega,
\end{equation}
where 
\begin{equation}\label{T-X4}
u_1(x)= (\partial_j\eta)(x)\int_{\partial\Omega}d^{n-1}\omega(y)\,k(x-y)f(y)
\end{equation}
and 
\begin{equation}\label{T-X5}
u_2(x)= \int_{\partial\Omega}d^{n-1}\omega(y)\,(\eta(x)-\eta(y))
(\partial_jk)(x-y)f(y).
\end{equation}
Theorem \ref{T-CZ} immediately gives that 
\begin{equation}\label{T-X6}
\|Mu_1\|_{L^2(\dOm;d^{n-1}\omega)}\leq C\|f\|_{L^2(\dOm;d^{n-1}\omega)},
\end{equation}
so it remains to prove a similar estimate with $u_2$ in place of $u_1$. 
To this end, we note that the problem localizes, so we may assume that
$\eta$ is compactly supported and $\Omega$ is the domain above the graph 
of a Lipschitz function $\varphi:\bbR^{n-1}\to\bbR$. In this scenario, 
by passing to Euclidean coordinates and denoting $g(y')= f(y',\varphi(y'))$, 
$y'\in\bbR^{n-1}$, it suffices to show the following. 
For $x'\in\bbR^{n-1}$, $t>0$, set  
\begin{equation}\label{T-X7}
v(x',t)= \int_{\bbR^{n-1}}d^{n-1}y'\,
(\eta(x',\varphi(x')+t)-\eta(y',\varphi(y')))
(\partial_jk)(x'-y',\varphi(x')-\varphi(y')+t)g(y')
\end{equation}
and, with $\kappa>0$ fixed, consider 
\begin{equation}\label{T-X8}
v_{**}(x')= {\rm sup}\,\{|v(z',t)|\,|\,|x'-z'|<\kappa t\},\quad x'\in\bbR^{n-1}.
\end{equation}
Then  
\begin{equation}\label{T-X9}
\|v_{**}\|_{L^2(\bbR^{n-1};d^{n-1}x')}\leq C\|g\|_{L^2(\bbR^{n-1};d^{n-1}x')}.
\end{equation}
To establish \eqref{T-X9}, fix a smooth, even function $\psi$ defined in 
${\bbR}^{n}$, with the property that $\psi\equiv 0$ near the origin and 
$\psi(x)=1$ for $|x|\geq\frac12$. We then further decompose 
\begin{equation}\label{T-X13}
v(x',t)=v^1(x',t)+v^2(x',t)
\end{equation}
where 
\begin{align}\label{T-X10}
& v^1(x',t)= \int_{\bbR^{n-1}}d^{n-1}y'\,
(\eta(x',\varphi(x'))-\eta(y',\varphi(y')))
(\partial_jk)(x'-y',\varphi(x')-\varphi(y')+t)g(y')  \\
& \quad 
=\int_{\bbR^{n-1}}d^{n-1}y'\,\frac{1}{|x'-y'|^{n-1}}
\frac{\eta(x',\varphi(x'))-\eta(y',\varphi(y'))}{|x'-y'|}
(\partial_jk)\Bigl(\frac{x'-y'}{|x'-y'|},
\frac{\varphi(x')-\varphi(y')+t}{|x'-y'|}\Bigr)g(y')
\nonumber\\ 
& \quad 
=\int_{\bbR^{n-1}}d^{n-1}y'\,\frac{1}{|x'-y'|^{n-1}}
\frac{\eta(x',\varphi(x'))-\eta(y',\varphi(y'))}{|x'-y'|}
(\psi\,\partial_jk)\Bigl(\frac{x'-y'}{|x'-y'|},
\frac{\varphi(x')-\varphi(y')+t}{|x'-y'|}\Bigr)g(y')
\nonumber
\end{align}
and
\begin{equation}\label{T-X11}
v^2(x',t)= t^{-1}(\eta(x',\varphi(x')+t)-(\eta(x',\varphi(x'))v^3(x',t), 
\end{equation}
where
\begin{align}\label{T-X12}
v^3(x',t) &=  
\int_{\bbR^{n-1}}d^{n-1}y'\,t(\partial_jk)(x'-y',\varphi(x')-\varphi(y')+t)g(y')  \no \\ 
&=\int_{\bbR^{n-1}}d^{n-1}y'\,\frac{1}{|x'-y'|^{n-1}}\frac{t}{|x'-y'|}
(\partial_jk)\Bigl(\frac{x'-y'}{|x'-y'|},
\frac{\varphi(x')-\varphi(y')+t}{|x'-y'|}\Bigr)g(y')  \no \\
&=\int_{\bbR^{n-1}}d^{n-1}y'\,\frac{1}{|x'-y'|^{n-1}}\frac{t}{|x'-y'|}
(\psi\,\partial_jk)\Bigl(\frac{x'-y'}{|x'-y'|},
\frac{\varphi(x')-\varphi(y')+t}{|x'-y'|}\Bigr)g(y').   
\end{align}
(In each case, the role of the function $\psi$ is to truncate the singularity
of $\partial_j k$ at the origin.) Consequently, if 
\begin{equation}\label{T-X14}
v^j_{**}(x')= {\rm sup}\,\{|v^j(z',t)|\,|\,|x'-z'|<\kappa t\},
\quad x'\in\bbR^{n-1},\,\,\,j=1,3,
\end{equation}
we have 
\begin{equation}\label{T-X15}
\|v_{**}\|_{L^2(\bbR^{n-1};d^{n-1}x')}\leq 
\|v^1_{**}\|_{L^2(\bbR^{n-1};d^{n-1}x')}+C\|v^3_{**}\|_{L^2(\bbR^{n-1};d^{n-1}x')}.
\end{equation}
As a consequence, it is enough to prove that 
\begin{equation}\label{T-X16}
\|v^j_{**}\|_{L^2(\bbR^{n-1};d^{n-1}x')}\leq C\|g\|_{L^2(\bbR^{n-1};d^{n-1}x')},
\end{equation}
for $j=1,3$. We shall do so by relying on Theorem \ref{t2.6X}
(considered with $n$ replaced by $n-1$). When $j=1$, we apply this theorem with 
\begin{eqnarray}\label{T-X17}
\begin{array}{l}
m= n,\quad\ell= 1,\quad a= (a_1,a_2)\in\bbR\times\bbR^{n-1},\quad b\in\bbR,\quad
\\[4pt]
F(a,b)= a_1(\psi\,\partial_jk)(a_2,b),\quad
B= \varphi,\quad
A= (\eta(\cdot,\varphi(\cdot)),\cdot).
\end{array}
\end{eqnarray}
When $j=3$, Theorem \ref{t2.6X} is used with (again $n-1$ in place of $n$) and 
\begin{eqnarray}\label{T-X18}
\begin{array}{l}
m= n,\quad \ell= 2,\quad b= (b_1,b_2)\in\bbR\times\bbR,
\quad a\in\bbR^{n-1},\quad
\\[4pt]
F(a,b)= \zeta(b_1-b_2)b_2(\psi\,\partial_jk)(a,b_1),\quad
B_1= \varphi,\quad B_2= 0,\quad
A= I_{\bbR^{n-1}}, 
\end{array}
\end{eqnarray}
where $\zeta\in C^\infty_0(\bbR)$ is an even function with the property 
that $\zeta\equiv 1$ on $[-M,M]$ (where $M$ is the Lipschitz constant of
$\varphi$). In each case, the hypotheses on $F$ made in the statement of 
Theorem~\ref{t2.6X} are verified, and part $(1)$ in Theorem \ref{t2.6X} 
yields the corresponding version of \eqref{T-X16}. This finishes the 
proof of Theorem \ref{T-CZ2}.
\end{proof}

After these preparations, we are finally ready to present the following proof:

\begin{proof}[Proof of Theorem \ref{K-CPT}.] 
We work under the assumption that $\Omega$ is a $C^{1,r}$-domain, for 
some $r>1/2$. In particular, $\nu\in [C^r(\dOm)]^n$. 
Thanks to Lemma~\ref{L-KbX} it suffices to show that \eqref{Ks-4B} holds 
for $z=0$. To this end, we write 
\begin{equation}\label{Ks-5}
K^{\#}_0=K_0+(K^{\#}_0-K_0)
\end{equation}
and observe that the integral kernel of the operator $R= K^{\#}_0-K_0$ is given by 
\begin{eqnarray}\label{Ks-6}
(\nu(x)-\nu(y))\cdot\nabla E_n(0;x-y),\quad x,y\in\dOm.
\end{eqnarray}
Let $\eta_\alpha\in [C^\infty(\bbR^n)]^n$, $\alpha\in\bbN$, be a sequence 
of vector-valued functions with the property that 
\begin{eqnarray}\label{Ks-7}
\eta_\alpha|_{\dOm}\to \nu\, \text{ in }\, [C^r(\dOm)]^n
\, \text{ as }\, \alpha\to\infty, 
\end{eqnarray}
and denote by $R_\alpha$ the integral operator with kernel 
\begin{eqnarray}\label{Ks-8}
(\eta_\alpha(x)-\eta_\alpha(y))\cdot \nabla E_n(0;x-y),\quad x,y\in\dOm.
\end{eqnarray}
From \eqref{goal-2} we know that
\begin{equation}\label{goal-2B}
\gamma_D\nabla{\mathcal S}_0
\in \cB\big(H^{1/2}(\partial\Omega),H^{1/2}(\partial\Omega)^n\big),
\end{equation}
which implies that for each $j\in\{1,...,n\}$, the principal-value 
boundary integral operator with kernel $\partial_j E_n(0;x-y)$ maps 
$H^{1/2}(\partial\Omega)$ boundedly into itself. From this, 
\eqref{Ks-6}, \eqref{Ks-7}, and Lemma~\ref{lA.6} we may then conclude that 
\begin{eqnarray}\label{Ks-9}
R_\alpha\to R\,\text{ in }\, \cB\bigl(H^{1/2}(\dOm)\bigr)
\,\text{ as }\, \alpha\to\infty.
\end{eqnarray}
Also, from Theorem \ref{T-CZ2} we have 
\begin{eqnarray}\label{Ks-10}
R_\alpha\in\cB\bigl(L^2(\dOm;d^{n-1}\omega),H^1(\dOm)\bigr)
\hookrightarrow\cB_\infty\bigl(H^{1/2}(\dOm)\bigr)
\, \text{ for every }\, \alpha\in\bbN.
\end{eqnarray}
From \eqref{Ks-9} and \eqref{Ks-10} we may then conclude that 
\begin{eqnarray}\label{Ks-11}
R\in\cB_\infty\bigl(H^{1/2}(\dOm)\bigr),
\end{eqnarray}
hence, ultimately, 
\begin{eqnarray}\label{Ks-12}
K^{\#}_0\in\cB_\infty\bigl(H^{1/2}(\dOm)\bigr),
\end{eqnarray}
by \eqref{Ks-5}, \eqref{Ks-11} and \eqref{Ks-4}.
\end{proof}

\noindent {\bf Acknowledgments.}
We wish to thank Gerd Grubb for questioning an inaccurate claim in an 
earlier version of the paper and Maxim Zinchenko for helpful discussions 
on this topic. 


\end{document}